\documentclass[11pt]{article}

\usepackage[latin1]{inputenc} 
\usepackage{amsthm,amsmath,amssymb,bm} 
\usepackage{graphicx}
\usepackage{color}
\usepackage{verbatim}
\usepackage{float}

\usepackage{array}   
\newcolumntype{C}{>{$}c<{$}}
\newcolumntype{L}{>{$}l<{$}}

\usepackage{enumitem} 
\setlist[enumerate]{topsep=0pt,itemsep=-1ex,partopsep=1ex,parsep=1ex}
\setlist[itemize]{topsep=0pt,itemsep=-1ex,partopsep=1ex,parsep=1ex}

\theoremstyle{plain}
\newtheorem{theo}{Theorem}[section]

\newtheorem{lemma}[theo]{Lemma}

\theoremstyle{definition}
\newtheorem{defn}[theo]{Definition}

\newcommand{\mc}[1]{\mathcal{#1}}
\newcommand{\mb}[1]{\mathbb{#1}}

\newcommand{\nib}[1]{\noindent {\bf #1}}

\newcommand{\bsize}[1]{\left| #1 \right|}

\newcommand{\bfl}[1]{\left\lfloor #1 \right\rfloor}
\newcommand{\bcl}[1]{\left\lceil #1 \right\rceil}

\newcommand{\sub}{\subseteq}

\newcommand{\Lra}{\Leftrightarrow}

\newcommand{\sm}{\setminus}
\newcommand{\ov}{\overline}

\newcommand{\wt}{\widetilde}
\newcommand{\eps}{\varepsilon}
\newcommand{\es}{\emptyset}

\newcommand{\ova}{\overrightarrow}
\newcommand{\lova}{\overleftarrow}

\newcommand{\aA}{\alpha}
\newcommand{\bB}{\beta}

\newcommand{\dD}{\delta}

\newcommand{\oO}{\omega}

\newcommand{\sd}{\bigtriangleup}

\newcommand{\GG}{\Gamma}
\newcommand{\DD}{\Delta}

\newcommand{\Ss}{\Sigma}

\newcommand{\LL}{\Lambda}

\newcommand{\wK}{\ova{W}^K_{\! 8}}

\newcommand{\pmin}{p_{\min}}
\newcommand{\pmax}{p_{\rm{max}}}
\newcommand{\geL}{{\ge}\LL}
\newcommand{\leL}{{<}\LL}

\newcommand{\ovXw}{\ov{X}_{\!\! w}}
\newcommand{\edge}[3]{``\phi_#1(#2){=}#3"}

\DeclareMathOperator{\ex}{ex}
\DeclareMathOperator{\hi}{hi}
\DeclareMathOperator{\iv}{int}
\DeclareMathOperator{\im}{Im}
\DeclareMathOperator{\bad}{bad}
\DeclareMathOperator{\lo}{lo}
\DeclareMathOperator{\no}{no}
\DeclareMathOperator{\free}{free}

\title{Ringel's tree packing conjecture in quasirandom graphs}
\author{Peter Keevash\thanks{Mathematical Institute,
University of Oxford, Oxford, UK. Email: keevash@maths.ox.ac.uk.}
\and Katherine Staden\thanks{Mathematical Institute,
University of Oxford, Oxford, UK. Email: staden@maths.ox.ac.uk.
\newline \hspace*{1.8em}Research supported
in part by ERC Consolidator Grant 647678.}}

\begin{document}

\maketitle

\begin{abstract}
We prove that any quasirandom graph 
with $n$ vertices and $rn$ edges
can be decomposed into $n$ copies 
of any fixed tree with $r$ edges.
The case of decomposing a complete graph
establishes a conjecture of Ringel from 1963.
\end{abstract}

\section{Introduction}

This paper concerns the following conjecture 
posed by Ringel~\cite{ringel} in 1963.

\medskip

\nib{Ringel's Conjecture.}
For any tree $T$ with $n$ edges, the complete graph
$K_{2n+1}$ has a decomposition into $2n+1$ copies of $T$.

\medskip

We prove this conjecture for large $n$, via the following theorem
which is a generalisation to decompositions of quasirandom graphs
into trees of the appropriate size. For the statement and
throughout we use the following quasirandomness definition:
we say that a graph $G$ on $n$ vertices is
\emph{$(\xi,s )$-typical} if every set $S$ of at most $s$ vertices 
has $((1  \pm \xi)d(G))^{|S|} n$ common neighbours, where 
$d(G) = e(G) \tbinom{n}{2}^{-1}$ is the density of $G$.

\begin{theo}\label{main}
There is $s \in \mb{N}$ such that
for all $p>0$ there exist $\xi,n_0$ such that 
for any $n \ge n_0$ such that $p(n-1)/2 \in \mb{Z}$
and any tree $T$ of size $p(n-1)/2$, any $(\xi,s)$-typical 
graph $G$ on $n$ vertices of density $p$
can be decomposed into $n$ copies of $T$.
\end{theo}

The case $p=1$ of Theorem \ref{main}
establishes Ringel's conjecture for large $n$, 
a result also recently obtained
independently by Montgomery, Pokrovskiy and Sudakov~\cite{MPS3}
by different methods, along the lines of their proof 
of an asymptotic version in~\cite{MPS2}. They show that certain 
edge-colourings of $K_{2n+1}$ contain a rainbow copy of $T$,
such that the required $T$-decomposition 
can be obtained by cyclically shifting this rainbow copy.
This approach is specific to the complete graph,
and does not apply to the more general setting 
of quasirandom graphs as in Theorem~\ref{main}.

Ringel's conjecture was well-known as one of the major open
problems in the area of \emph{graph packing},
whose history we will now briefly discuss.
In a graph packing problem, one is given a host graph $G$
and another graph $F$ and the task is to fit 
as many edge-disjoint copies of $F$ into $G$ as possible. 
If the size (number of edges) of $F$ divides that of $G$, 
it may be possible to find a perfect packing,
or \emph{$F$-decomposition} of $G$.
More generally, given a family $\mc{F}$ of graphs
of total size equal to the size of $G$,
we seek a partition of (the edge set of) 
$G$ into copies of the graphs in $\mc{F}$.

These problems have a long history, 
going back to Euler in the eighteenth century.
The flavour of the problem depends very much 
on the size of $F$. The earliest results concern $F$ 
of fixed size, in which case $F$-decompositions can 
be naturally interpreted as combinatorial designs.
For example, Kirkman~\cite{kirkman} showed that
$K_n$ has a triangle decomposition whenever $n$
satisfies the necessary divisibility conditions
$n \equiv 1$ or $3$ mod $6$; for historical reasons,
such decompositions are now known as Steiner Triple Systems.
Wilson~\cite{wilson1,wilson2,wilson3,wilson4}
generalised this to any fixed-sized graph in the 70's,
and Keevash~\cite{Kexist} to decompositions into complete
hypergraphs, thus estalishing the Existence Conjecture 
for designs. A different proof and a generalisation 
to $F$-decompositions for hypergraphs $F$ were given
by Glock, K\"uhn, Lo and Osthus \cite{GKLO,GKLO2}.
A further generalisation that captures many other
design-like problems, such as resolvable hypergraph designs
(the general form of Kirkman's celebrated 
`Schoolgirl Problem') was given by Keevash \cite{K2}.

There is also a large literature on $F$-decompositions
where the number of vertices of $F$ is comparable with,
or even equal to, that of $G$. Classical results of
this type are Walecki's 1882 decompositions 
of $K_{2n}$ into Hamilton paths,
and of $K_{2n+1}$ into Hamilton cycles.
There are many further results on Hamilton decompositions
of more general host graphs, notably the solution 
in~\cite{CKLOT} of the Hamilton Decomposition Conjecture,
namely the existence of a decomposition by Hamilton cycles
in any $2r$-regular graph on $n$ vertices,
for large $n$ and $2r \geq \lfloor n/2\rfloor$.

Much of the literature on $F$-decompositions for large $F$
concerns decompositions into trees. Besides Ringel's conjecture,
the other major open problem of this type 
is a conjecture of Gy\'arf\'as~\cite{GL},
saying that $K_n$ should have a decomposition
into any family of trees $T_1,\ldots,T_n$ 
where each $T_i$ has $i$ vertices.
Both conjectures have a large literature of partial results;
we will briefly summarise the most significant of these
(but see also~\cite{BHPT,FLM,KKOT,MRS}).
Joos, Kim, K\"uhn and Osthus~\cite{JKKO} 
proved both conjectures for bounded degree trees.
Ferber and Samotij~\cite{FS} and
Adamaszek, Allen, Grosu and Hladk\'y~\cite{AAGH}
obtained almost-perfect packings of 
almost-spanning trees with maximum degree $O(n/\log n)$.
These results were generalised 
by Allen, B\"ottcher, Hladk\'y and Piguet~\cite{ABHP} 
to almost-perfect packing of spanning graphs 
with bounded degeneracy and maximum degree $O(n/\log n)$.
Allen, B\"ottcher, Clemens and Taraz~\cite{ABCT}
extended~\cite{ABHP} to perfect packings provided
linearly many of the graphs are slightly smaller than 
spanning and have linearly many leaves.
The above results mainly use randomised embeddings,
for which a maximum degree bound $O(n/\log n)$ 
is necessary for concentration of probability.
While the results of Montgomery, Pokrovskiy 
and Sudakov~\cite{MPS1,MPS2} mentioned above
also use probabilistic methods, they are able
to circumvent the maximum degree barrier by methods
such as the cyclic shifts mentioned above.

Our proof proceeds via a rather involved embedding algorithm, 
discussed and formally presented in the next section,
in which the various subroutines are analysed by 
a wide range of methods, some of which 
are adaptations of existing methods 
(particularly from \cite{MPS1} and \cite{ABCT},
and also our own recent methods in \cite{factors} 
for the `generalised Oberwolfach problem',
which are in turn based on \cite{K2}),
but most of which are new, including 
a method for allocating high degree vertices
via partitioning and edge-colouring arguments
and a method for approximate decompositions based
on a series of matchings in auxiliary hypergraphs.

\subsection{Notation}

Given a graph $G = (V,E)$, 
when the underlying vertex set $V$ is clear, 
we will also write $G$ for the set of edges. 
So $|G|$ is the number of edges of $G$. Usually $|V|=n$.
The \emph{edge density} $d(G)$ of $G$ is $|G|/\tbinom{n}{2}$.
We write $N_G(x)$ for the neighbourhood of a vertex $x$ in $G$.
The degree of $x$ in $G$ is $d_G(x)=|N_G(x)|$.
For $A \subseteq V(G)$, we write 
$N_G(A) := \bigcap_{x \in A}N_G(x)$;
note that this is the common neighbourhood 
of all vertices in $A$, not the neighbourhood of $A$.

We often write $G(x)=N_G(x)$ to simplify notation.
In particular, if $M$ is a matching 
then $M(x)$ denotes the unique vertex $y$
(if it exists) such that $xy \in M$.
We also write $M(S) = \bigcup_{x \in S} M(x)$,
which is \emph{not} consistent with our notation
$N_G(S)$ for common neighbourhoods, but we hope
that no confusion will arise, as we only use
this notation if $M$ is a matching, 
when all common neighbourhoods are empty.

We say $G$ is \emph{$(\xi,s )$-typical} 
if $|N_G(S)| = ((1  \pm \xi)d(G))^{|S|} n$ 
for all $S \subseteq V(G)$ with $|S| \le s$. 

In a directed graph $J$ with $x \in V(J)$, 
we write $N_J^+(x)$ for the set of out-neighbours of $x$ in $G$ 
and $N_G^-(x)$ for the set of in-neighbours.
We let $d^\pm_G(x) := |N^\pm_G(x)|$.
We define common out/in-neighbourhoods 
$N_J^\pm(A) = \bigcap_{x \in A} N_J^\pm (x).$

The vertex set $V(G)$ will often come with a cyclic order,
identified with the natural cyclic order
on $[n]=\{1,\dots,n\}$. For any $x \in V$ we write $x^+$
for the successor of $x$, so if $x \in [n]$ then 
$x^+$ is $x+1$ if $x \ne n$ or $1$ if $x=n$.
Write $S^+=\{x^+: x \in S\}$ for $S \subseteq V(G)$.
We define the predecessor $x^-$ similarly. Given $x,y$ in $[n]$
we write $d(x,y)$ for their cyclic distance,
i.e.\ $d(x,y) = \min \{ |x-y|, n-|x-y| \}$.

We say that an event $E$ holds with high probability (whp) 
if $\mb{P}(E) > 1 - \exp(-n^c)$ for some $c>0$ and $n>n_0(c)$.
We note that by a union bound for any fixed collection $\mc{E}$
of such events with $|\mc{E}|$ of polynomial growth  
whp all $E \in \mc{E}$ hold simultaneously. 

We omit floor and ceiling signs for clarity of exposition.

We write $a \ll b$ to mean $\forall\ b>0 \
\exists\ a_0>0 \ \forall\ 0<a<a_0$.

We write $a \pm b$ for an unspecified number in $[a-b,a+b]$.

\section{Proof overview and algorithm} \label{sec:alg}

Suppose we are in the setting of Theorem \ref{main}:
we are given an $(\xi,2^{50\cdot 8^3})$-typical graph $G$ on $n$ vertices
of density $p$, where $n^{-1} \ll \xi \ll p$,
and we need to decompose $G$ into $n$ copies 
of some given tree $T$ with $p(n-1)/2$ edges.
In this section we present the algorithm
by which this will be achieved.
After describing and motivating the algorithm,
we present the formal statement in the next subsection,
then various lemmas analysing certain subroutines
over the following few subsections. 
We defer the analyses of
the approximate decomposition to Section \ref{sec:approx}
and the exact decomposition to Section \ref{sec:exact}.

As discussed in the introduction,
the most significant technical challenge not addressed 
by previous attempts on Ringel's Conjecture 
is the presence of high degree vertices,
so naturally these will receive special treatment.
Our algorithm will consider three separate cases 
for the tree $T$ (similarly to \cite{MPS1}),
one of which (Case L) handles trees in which
almost all (i.e.\ all but $o(n)$) vertices 
belong to large stars (i.e.\ of size $>n^{1-o(1)}$).
Case L is handled by the subroutine LARGE STARS,
which will be discussed later in this overview.
The other two cases for $T$ are Case S, when $T$ has 
linearly many leaves in small stars, and Case P,
when $T$ has linearly many vertices 
in vertex-disjoint long bare paths.
In both Case S and P, we apply essentially 
the same `approximate step'
algorithm to embed edge-disjoint copies 
of $F = T \sm P_{\ex}$, obtained from $T$ 
by removing the part that will be embedded 
in the `exact step', so $P_{\ex}$ consists
of stars in Case S and of bare paths in Case P.
The overview of the proof according to these cases
is illustrated by Figure \ref{fig:overview}.

\begin{figure}
\centering
\includegraphics[scale=0.85]{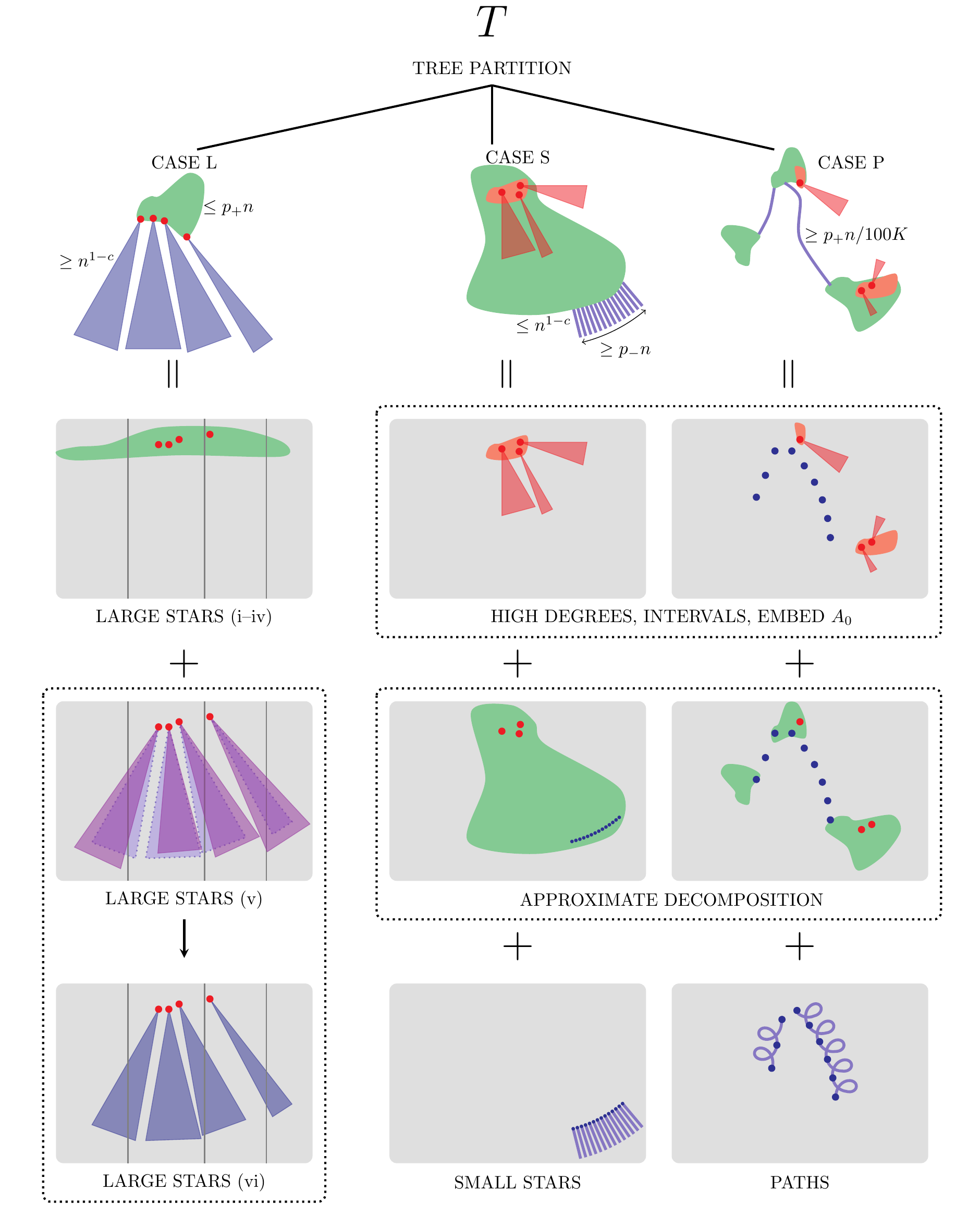}
\caption{The three cases
of the proof and the subroutines of the algorithm
which embed each part of $T$. From left to right, 
Case L: almost all vertices lie in large stars; 
Case S: linearly many vertices lie in small stars;
Case P: linearly many vertices lie in long bare paths.
Red denotes high degree vertices and their neighbours.
Blue denotes the part embedded in the exact step.}
\label{fig:overview}
\end{figure}

The heart of the approximate step algorithm is the
subroutine APPROXIMATE DECOMPOSITION, where in each step 
we extend our partial embeddings $(\phi_w: w \in W)$ of $F$
by defining them on some set $A_i$ which is suitably nice:
$A_i$ is independent, has linear size, has no vertices 
of degree $>n^{o(1)}$,
and every vertex of $A_i$ 
has at most four previously embedded neighbours.
We find these extensions simultaneously via a matching
in an auxiliary hypergraph ${\cal H}_i$
(see Figure \ref{fig:hyp}), which has an edge 
denoted $\edge{w}{u}{x}$ whenever it is possible to define
$\edge{w}{u}{x}$ for some $w \in W$, $u \in A_i$, $x \in V=V(G)$.
We encode the various constraints that must be satisfied
by the embeddings in the definition of these edges.
Thus $\edge{w}{u}{x}$ includes 
(as an auxiliary vertex in $V({\cal H}_i)$)
all arcs $\ova{yx}$ where $y=\phi_w(b)$ is a previously
defined embedding of some neighbour $b$ of $a$;
this ensures that we maintain edge-disjointness
of the embeddings of $F$. We also include 
in $\edge{w}{u}{x}$ auxiliary vertices $uw$ and $xw$,
to ensure that every $\phi_w(u)$ is defined at most once
and $\phi_w$ is injective.

\begin{figure} \label{fig:hyp}
\includegraphics[scale=0.9]{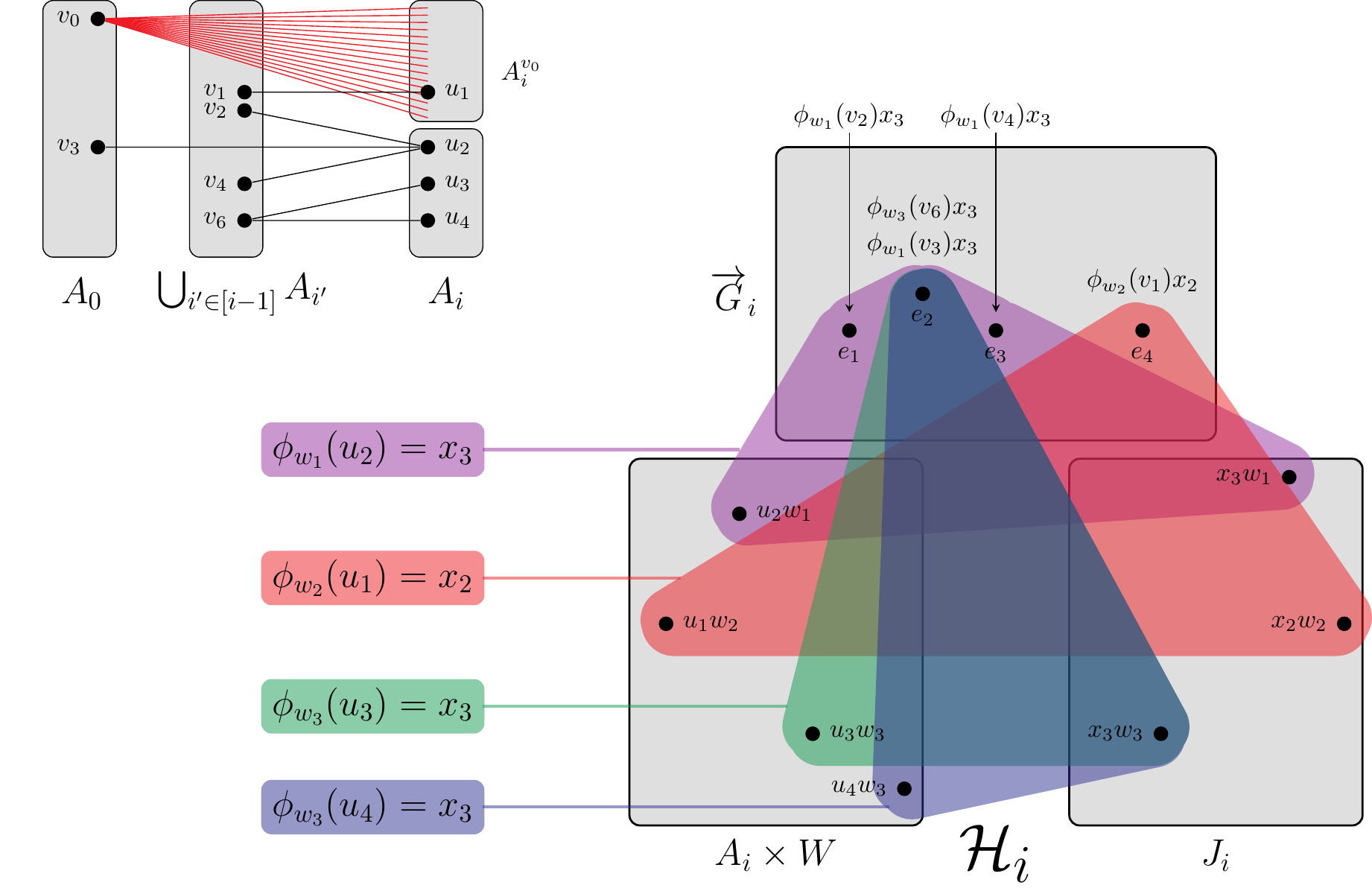}
\caption{Part of the hypergraph $\mc{H}_i$,
where a section of $F[A_i,A_{<i}]$ and some of
the corresponding edges of $\mc{H}_i$
are illustrated.
Here, $u_1 \in A^{\hi}_i$,
$u_2 \in A^{\lo}_i$ and
$u_3,u_4 \in A^{\no}_i$.
In a previous embedding, 
we set $\phi_{w_1}(v_3)=\phi_{w_3}(v_6)=x_1$, 
and now the
arc $e_2=x_3 x_1$ would be used by
the potential embeddings $``\phi_{w_1}(u_2)=x_3"$
(purple edge),
$``\phi_{w_3}(u_3)=x_3"$
(green edge) and
$``\phi_{w_3}(u_4)=x_3"$
(blue edge).
In particular, at most one of these embeddings is allowed.}
\end{figure}

We ensure that ${\cal H}_i$ is suitably nice (its edges 
can be weighted so that every vertex has weighted degree 
$1+o(1)$ and all weighted codegrees are $n^{-o(1)}$),
in which case it is well-known from the large literature
developing R\"odl's semi-random `nibble' \cite{R},
in particular \cite{KaLP}, that one can find 
an almost perfect matching 
that is (in a certain sense) quasirandom
(we use a convenient refined formulation of this
statement recently presented in \cite{EGJ}).
The quasirandomness of this matching is important for
several reasons, including quasirandomness of the
extensions of the embeddings to $A_i$, which in turn
implies that later hypergraphs ${\cal H}_j$ with $j>i$
are suitably nice (with weaker specific parameters),
and so the process can be continued.

The above sketch yields an alternative method
for approximate decomposition results along
the lines of those mentioned in the introduction,
but has not yet dealt with high degree vertices.
We will partition $V(F)$ into $A_0,A_1,\dots,A_{i^*}$,
where $A_i$ for $i \ge 1$ are the nice sets described above,
and $A_0$ is not nice -- in particular, there is no bound
on the degree of vertices in $A_0$. We start the embedding
of $F$ in the subroutine HIGH DEGREES by embedding vertices
sequentially in a suitable order, where when we consider
some $a \in A_0$ we define $\phi_w(a)$ for all $w \in W$
simultaneously via a random matching 
$M_a = \{ \phi_w(a)w: w \in W\}$ in an auxiliary 
bipartite graph $B_a \sub V \times W$,
where the definition of $B_a$ encodes constraints
that must be satisfied by the embedding:
we only allow an edge $vw$ if $v \notin \im \phi_w$
and $v$ is adjacent via unused edges to all $\phi_w(b)$
where $b$ is a previously embedded neighbour of $a$.
(For simplicity we have suppressed several further details
in the above description which will be discussed below.)
The important point about this construction is that
each $v \in V$ has to accommodate the vertex $a$ for
a unique embedding $\phi_w$, so however large 
the degrees in $T$ may be, the total demand 
for `high degree edges' is the same at every vertex,
and can be allocated to a digraph $H$ which
is an orientation of a quasirandom subgraph of $G$.

This digraph $H$ is one of many
oriented quasirandom subgraphs into which $G$
is partitioned by the subroutine DIGRAPH, where each 
piece is reserved for embedding certain subgraphs of $F$,
with arcs directed from earlier to later vertices.
Besides $H$, these include graphs $G^{gg'}_{ii'}$
for embedding subgraphs $F'[A^g_i,A^{g'}_{i'}]$,
according to a partition of each $A_i$ into
$A^{\hi}_i$, $A^{\lo}_i$, $A^{\no}_i$.
Here $A^{\hi}_i$ consists of vertices adjacent to 
some vertex with many neighbours in $A_i$
(which will lie in $A_0$ and be unique),
$A^{\lo}_i$ consists of vertices
adjacent to some vertex in $A_0$
(which will be unique) that does not 
have many neighbours in $A_i$,
and $A^{\no}_i$ consists of vertices
with no neighbours in $A_0$. 
To ensure concentration of probability
the above sets are not defined if they would
have size $o(n)$, in which case the corresponding
vertices are instead added to $A_0$.
By partitioning $G$ in this manner we can ensure
edge-disjointness when embedding different parts 
of $F$ separately. To ensure injectivity of 
the embeddings, we also randomly partition $V \times W$ 
into various subgraphs in which $w$-neighbourhoods
prescribe the allowed images in $\phi_w$
of the various parts of the decomposition of $V(T)$.
In particular, while constructing the high degree digraph $H$,
we also construct $J^{\hi}_i \sub V \times W$ so that each 
$\phi_w(A^{\hi}_i)$ will be approximately 
equal to $J^{\hi}_i(w)$.

\begin{figure} \label{fig:partition}
\includegraphics[scale=0.9]{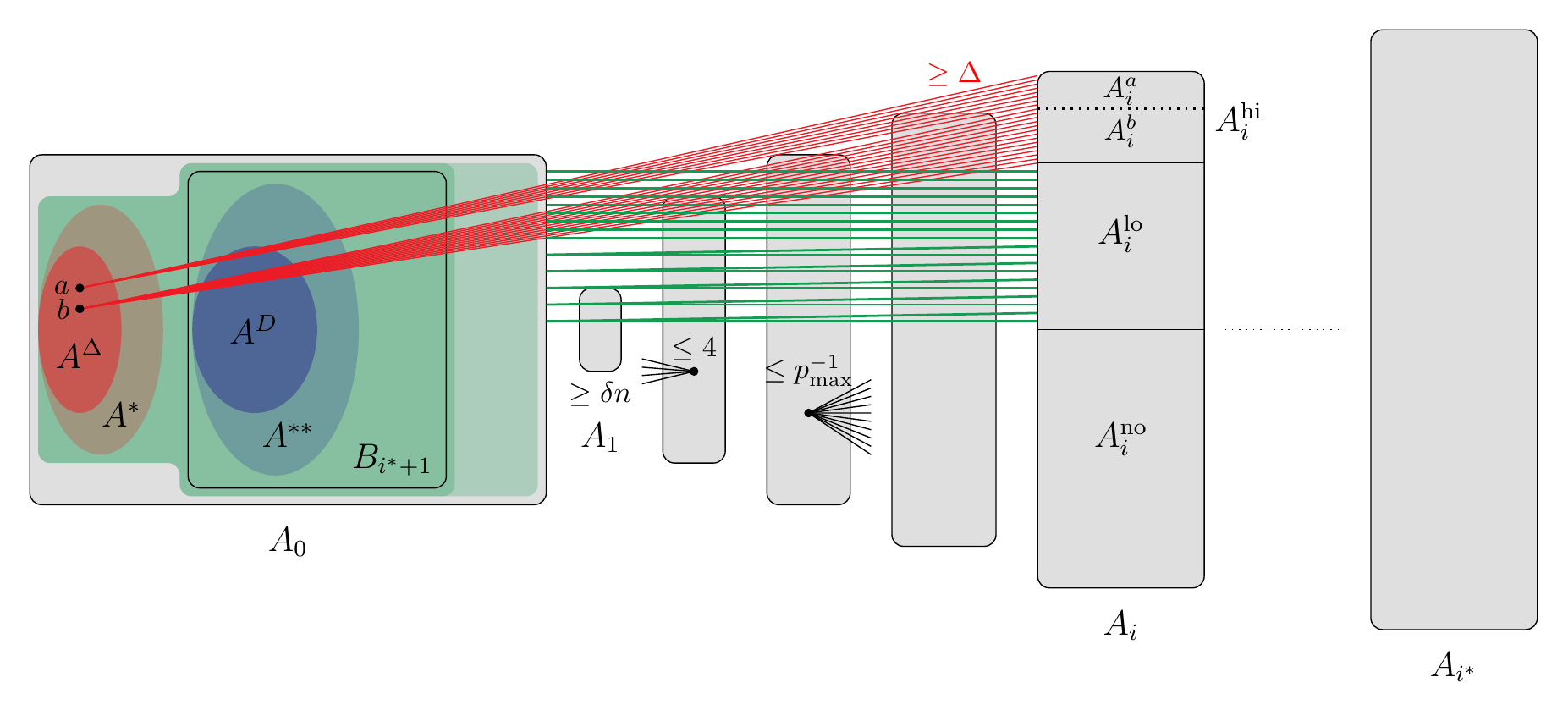}
\caption{Partition of $F$ obtained in TREE PARTITION.
Red edges are from $T\sm F$ and green edges
are $F$-edges from $A_0$ to $A_i$ (so $A^{\lo}_i$).}
\end{figure}

The separate treatment of these parts of $A_i$
and careful construction of $A_0$ to ensure the
uniqueness properties mentioned above is designed
to handle a considerable technical difficulty that 
we glossed over above when describing the embedding of $A_0$.
Our approach to the approximate decomposition discussed
above depends on maintaining quasirandomness,
but we cannot ensure that $|A_0|/n$
is negligible compared with $1/i^*$, 
where $i^*$ is the number of steps 
in the approximate decomposition, 
so a naive analysis will fail due
to blow-up of the error terms.
We therefore partition $A_0$ 
into $A^*$, $A^{**}$ and $A'_0$,
which are embedded sequentially,
where $|A^*|/n$ and $|A^{**}|/n$ are
negligible compared with $1/i^*$, 
and so do not contribute much to the error terms.
For $A'_0$, we cannot entirely avoid large error terms,
but we can confine them to a set of $o(n)$ bad vertices,
via arguments based on Szemer\'edi regularity;
these arguments require degrees in $A'_0$ 
to be bounded independently of $n$, 
so $A^{**}$ is introduced to handle degrees
that are $\oO(1)$ but $<n^{o(1)}$.
The careful choice of partition ensures that these
bad error terms are only incurred by vertices in $A^{\lo}$.

At this point, we return to consider various details
glossed over in the above description of HIGH DEGREES.
While the embedding via random matchings ensures that
every vertex of $G$ has the same demand of high degree edges,
we also need to plan ahead when embedding $A^* \sub A_0$
(which contains the very high degree vertices)
so that it will be possible to
allocate the other ends of these edges to distinct vertices
for each $w$, i.e.\ so that $\phi_w(u) \ne \phi_w(u')$
whenever $u \ne u'$. To achieve this in DIGRAPH, 
we randomly partition $V$ into $(U_h: h \in [m])$, 
with $m=n^{1-o(1)}$,
where each $U_h$ will accommodate those ends of 
high degree edges
corresponding to colour $h$ in a certain properly
$m$-edge-coloured bipartite multigraph in $V \times W$,
i.e.\ $\ova{yx}$ is available for $H$ 
if $yw$ has colour $h$ and $x \in U_h$.
Thus $x=\phi_w(u) \in U_h$ and $x'=\phi_w(u') \in U_{h'}$ 
are distinct automatically if $h \ne h'$,
and due to properness of the colouring if $h=h'$,
as $c$ determines a unique $y \in V(G)$,
so a unique $a=\phi_w^{-1}(y) \in A^*$.

The above multigraph $M$ in $V \times W$ consists of
copies of $M^*_a \approx M_a$ for each $a \in A^*$,
with the copies distinguished by labels $\ell_{aij}$,
where for each $a \in A^*$ and part $A_i$ in which 
$a$ has many neighbours the number of labels $\ell_{aij}$ 
is proportional to the degree of $a$ in $A_i$.
An edge $yw$ of label $\ell_{aij}$ in $M^h$ means that
$H$ arcs $\ova{yx}$ with $x \in U_h$ will be allocated
to edges $au$ of $F$ with $a=\phi_w^{-1}(y)$
and $u \in N_F(a) \cap A_i$. For typicality we require
for any $a$ and $i$ that the number of edges 
in each $M^h$ with some label $\ell_{aij}$ 
is approximately independent of $h$.

This is achieved by a construction based on cyclic shifts,
which we will now sketch, suppressing some details.
We partition $V$ into $V_0$ and $(V_{v^*}: v^* \in V^*)$
and $W$ into $W_0$ and $(W_{w^*}: w^* \in W^*)$, 
where $V_0$ and $W_0$ are small,
$V^*$ and $W^*$ are copies of $[m]$,
and all $V_{v^*},W_{w^*}$ have the same size.
The matchings $M_a$ are chosen as $M^0_a \cup M^*_a$,
where $V(M^0_a) = V_0 \cup W_0$ and if $vw \in M^*_a$
then $v \in V_{v^*}$, $w \in W_{w^*}$ with $v^*=x_a+w^*$,
according to some cyclic shifts $(x_a: a \in A^*)$,
carefully chosen to ensure edge-disjointness.
We construct a labelled multigraph in $V^* \times W^*$
analogously to that in $V \times W$, and
obtain label-balanced matchings $M^h$
for all $h \in [m]$ as cyclic shifts of some fixed 
label-balanced matching $M'$ in $V^* \times W^*$,
where for each $v^*w^* \in M'$ with some label $\ell_{aij}$
we include in $M^h$ all edges of $M^*_a$ of the
same label between $V_{v^*+h}$ and $W_{w^*+h}$.

The above description of $M^h$ is over-simplified,
as in fact we construct two such matchings,
one handling vertices of huge degree
(almost linear) and the other handling vertices
with degree that is high but not huge.
The version of $M'$ for non-huge degrees is constructed
by the same hypergraph matching methods as in the above
description of the approximate step embeddings,
but these do not apply to huge degrees 
(the codegree bound fails) so we instead 
apply a result of 
Bar\'at, Gy\'arf\'as and S\'ark\"ozy \cite{BGS} 
on rainbow matchings 
in properly coloured bipartite multigraphs.
The construction is illustrated in Figure~\ref{fig:Hai}.

\begin{figure}
\centering
\includegraphics[scale=.6]{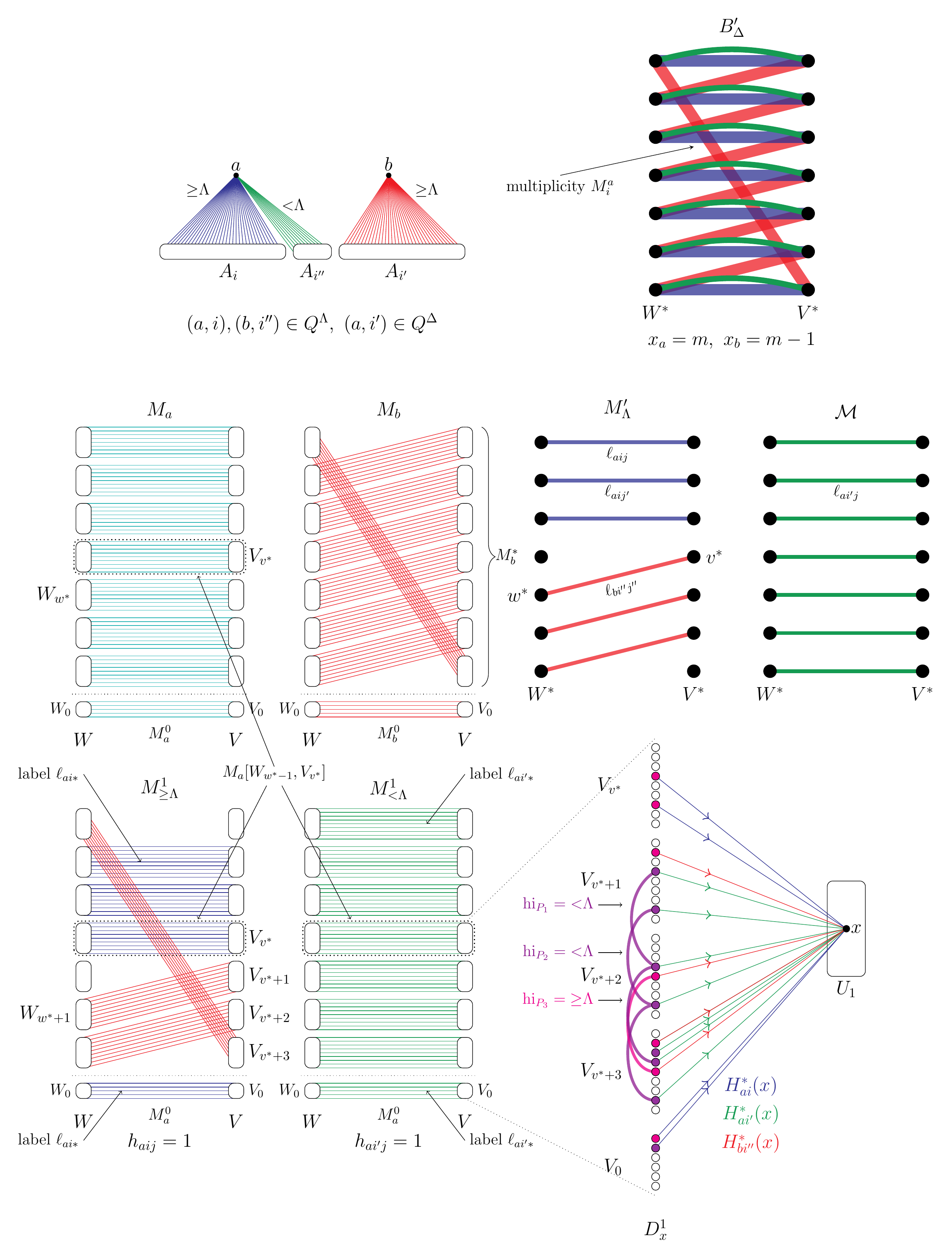}
\caption{From left to right, top to bottom:
two high degree vertices $a,b$; the multigraph $B'_{\DD}$
where line thickness represents multiplicity;
the matchings $M_a,M_b$ between $W$ and $V$;
the matchings $M'_{\LL},\mc{M}$ on $W^*,V^*$;
the matchings $M^1_{\geL}, M^1_{\leL}$;
the graph $D^1_x$ for $x \in U_1$ with components $P$ coloured 
to represent the random choice ${\hi}_P \in \{\geL,\leL\}$;
the resulting edges of $H^*_{ai}$ at $x$.}
\label{fig:Hai}
\end{figure}

The exact steps in Cases S and P are handled
by adapting existing methods in the literature.
In Case P, the subroutines INTERVALS and PATHS
are adaptations of the methods we used in \cite{factors}
for the `generalised Oberwolfach Problem'
of decomposing any quasirandom even regular 
oriented graph into prescribed cycle factors;
we refer the reader to this paper 
for a detailed discussion of these methods.
In Case S, we find the required stars by
adapting an algorithm of \cite{ABCT}:
we find an orientation of the unused graph
so that the outdegree of each vertex is precisely
the total size of stars it requires in all copies of $T$,
and then process each vertex in turn, using random
matchings to partition its outneighbourhood into stars
of the correct sizes, while maintaining injectivity
of the embeddings.

It remains to consider the exact step in Case L,
when almost every vertex of $T$ is a leaf adjacent
to a vertex of very large degree; this is more challenging 
and requires new methods (the arguments used in Case S 
fail due to lack of concentration of probability).
The most difficult constraint to satisfy is 
injectivity of the embeddings, so we build this
into the construction explicitly: we randomly partition 
$V(G)$ into sets $U^a$ for each star centre $a$
and require each embedding to choose most of its
leaves for its copy of $a$ within $U^a$.
Each edge $xy$ of $G$,
say with $x \in U^a$, $y \in U^b$, 
will be randomly allocated one of two options:
(i) $x$ is a leaf of a star in some embedding
$\phi_w$ with $\phi_w(a)=y$, or
(ii) $y$ is a leaf of a star in some embedding
$\phi_{w'}$ with $\phi_{w'}(b)=x$.
A final balancing step will swap edges between stars
(thus slightly bending the rules on leaf allocation)
so that all stars are exactly as required;
see Figure~\ref{fig:modify}.
The above sketch can be implemented
for decomposing a quasirandom graph 
into star forests, but there is a considerable
extra difficulty caused by the constraints imposed
by the initial embedding of the small part of $T$
not contained in the large stars. 

\begin{figure}
\includegraphics[scale=.85]{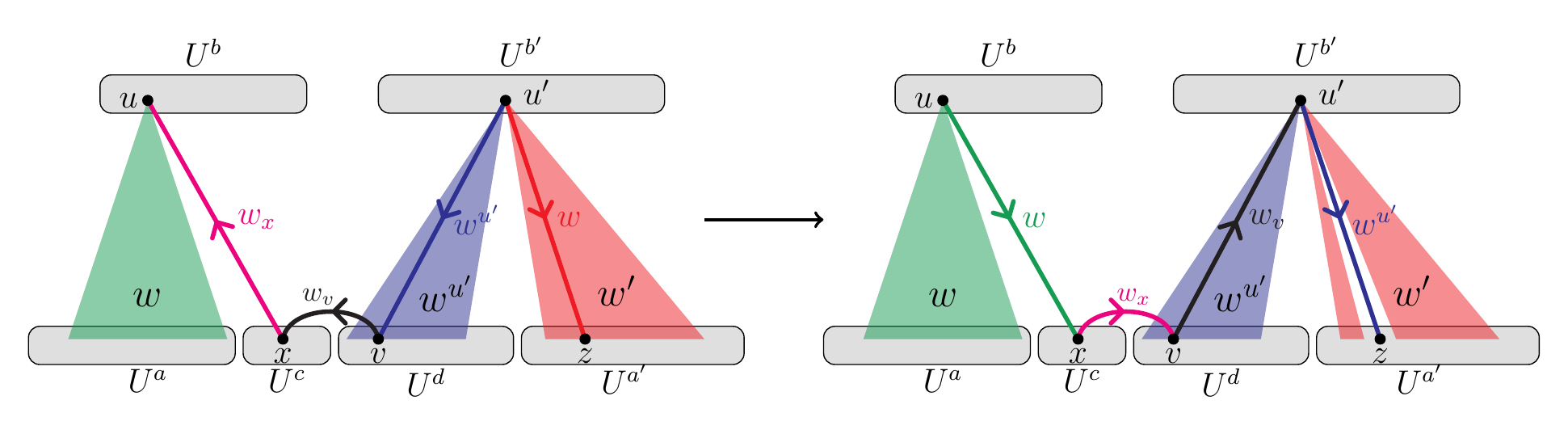}
\caption{A single step of the algorithm
to modify the green star which is too small
and the red star which is too large.}
\label{fig:modify}
\end{figure}

A naive approach
to this embedding can easily cause many edges of $G$
to be unusable according to the rules for $U^a$
as described above. Indeed, for each edge $xy$ of $G$,
the two options as described above will both become 
unavailable during the initial embedding
if we choose both $\phi_w(a')=x$ for some $a'$
and $\phi_{w'}(b')=y$ for some $b'$.
We therefore keep track of a digraph $J$
that records these constraints and choose the
initial embedding so that each edge of $G$
always has at least one of its two options available.
To control these constraints, we also introduce
partitions of each $U^a$ into three parts,
and also of the set $W$ indexing the embeddings
into three parts, and impose two different patterns
for matching parts of $U^a$ with parts of $W$
according to whether or not a vertex has large degree.
The digraph $J$ and its use in defining available sets
for the embedding are illustrated in Figure \ref{fig:Ab}.

\begin{figure}
\centering
\includegraphics[scale=0.8]{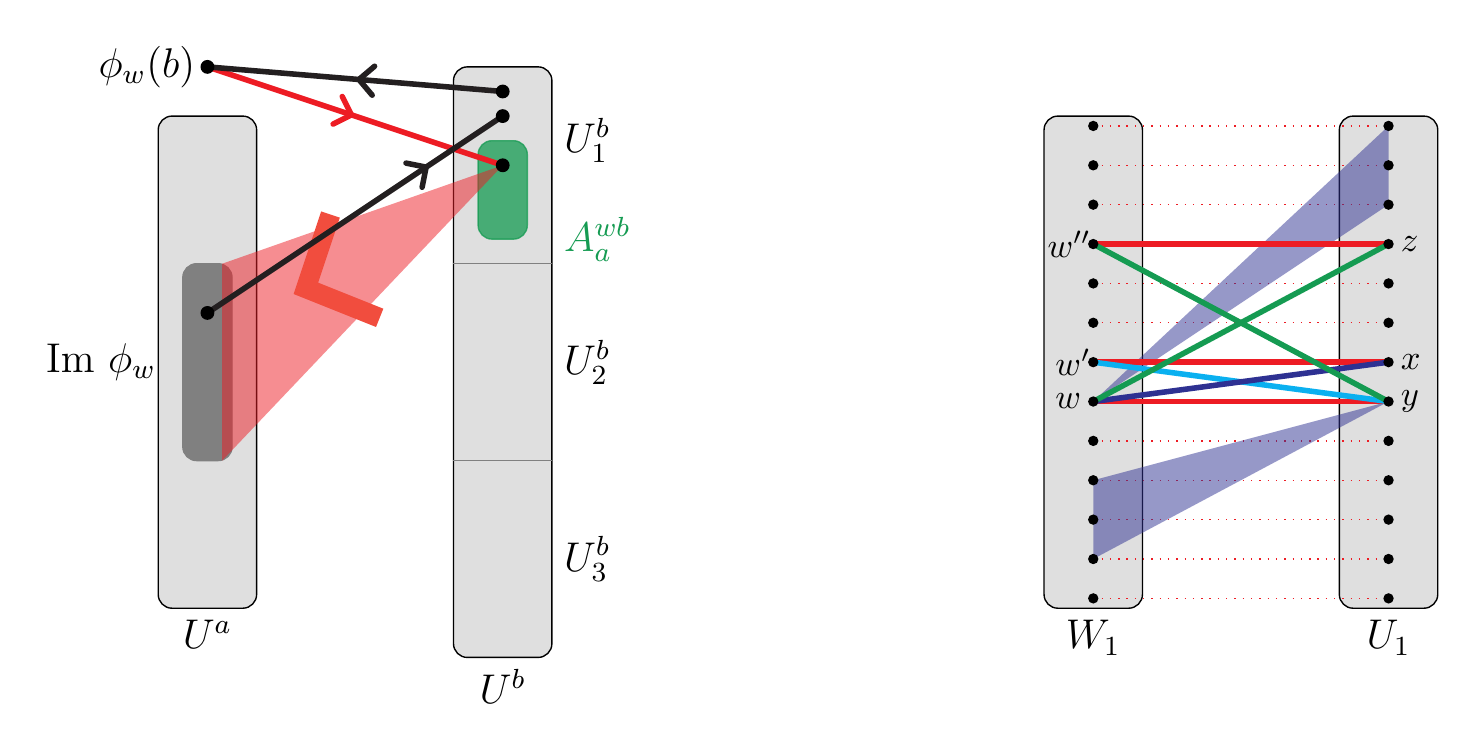}
\caption{(Left) the available set $A^{wb}_a$ 
for $w \in W_1$ and $a \in S$.
The black arcs are some arcs in $J$; 
they forbid their $U^b$-endvertices from $A^{wb}_a$.
The red arcs would be added to $J$ if the labelled
vertex is chosen for $\phi_w(a)$.
(Right) A pair of edges $wy$, $w'x$ that must be avoided 
by the matching defining the embeddings of $a$,
and a swap that may be implemented 
by Lemma \ref{lem:match} to remove $wy$.
Red edges define images of $a$
and blue edges define images of some other vertices.}
\label{fig:Ab}
\end{figure}

\subsection{Formal statement of the algorithm}

The input to the algorithm consists of
a $(\xi,s)$-typical graph $G$ on $n$ vertices
of density $p$, where $s=2^{50\cdot 8^3}$,
$n^{-1} \ll \xi \ll p$,
and a tree $T$ with $p(n-1)/2$ edges. 
We fix $0 < c' \ll c \ll 1$ and parameters 
\[ n^{-1} \ll \xi
\ll \eta_- \ll p_-
\ll \eta_+ \ll p_+ \ll p,
\quad \text{ and } \DD = n^c, \ \LL = n^{1-c}. \]
Recall that a \emph{leaf} in $T$ is a vertex of degree $1$ in $T$.
We call an edge a \emph{leaf edge} if it contains a leaf.
We call a star a \emph{leaf star} if it consists of leaf edges.
We call a path in $T$ a \emph{$k$-path} if it has length $k$
(that is, $k$ edges),
and call it \emph{bare} if its internal vertices
all have degree $2$ in $T$. By Lemma \ref{lem:T} below
we can choose a case for $T$ in \{L,S,P\} satisfying
\begin{itemize}
\item Case L: all but at most $p_+ n$ vertices
of $T$ belong to leaf stars of size $\geL$,
\item Case S: at least $p_- n$ vertices
of $T$ belong to leaf stars of size $\le \LL$,
\item Case P: $T$ contains $p_+ n/100K$ vertex-disjoint
bare $8K$-paths.
\end{itemize}
In Case L go to LARGE STARS, otherwise continue.
We let $\eta=\eta_-$ in Case S or $\eta=\eta_+$ in Case P,
and define further parameters
\begin{gather*}
 \xi \ll \xi' \ll D^{-1} \ll  \dD   \ll \pmin
\ll \eps_1 \ll \ldots \ll \eps_{i^+} \ll \pmax
\ll \eps \ll p_0  \ll \eta \ll s^{-1},p,
\end{gather*}
with
$i^+ = 7 \log \eps^{-1}$ and
$\xi' \ll K^{-1} \ll d^{-1} \ll D^{-1}$ in Case P.
Given $k \in \mb{N}$, a tree $T$ and $S \sub V(T)$,
the \emph{$k$-span} $\text{span}^k_T(S)$ of $S$ in $T$
is obtained by starting with $S^*=S$
and iteratively adding any $S' \sub V(T) \sm S^*$
with $|S'| \in [k]$ such that $T[S^* \cup S']$
has fewer components than $T[S^*]$,
until there is no such $S'$.
Clearly there are at most $|S|$ iterations,
so $|\text{span}^k_T(S)| \le (k+1)|S|$.
Note also that 
$|\text{span}^k_T(\text{span}^k_T(A) \cup B)
\sm \text{span}^k_T(A)|
\le (k+1)|B|$.
For $k \in \mb{N}$ let
$A^k = \{u: d_T(u) \ge k \}$.

\vspace{-0.3cm} \begin{center} 
TREE PARTITION \end{center} \vspace{-0.3cm}

\begin{enumerate}
\item Let $A^* = \text{span}^4_T(A^\DD)$.
In Case S let $P_{\ex}$ 
be a union of leaf stars in $T \sm T[A^*]$,
each of size $\le \LL$,
with $|P_{\ex}| = p_- n/2 \pm \LL$.
In Case P let $P_{\ex}$ be the
vertex-disjoint union 
of two leaf edges in $T \sm T[A^*]$ and 
$p_+ n/101K$ bare $8K$-paths in $T \sm T[A^*]$.\\
Obtain $F$ from $T$ 
by deleting all edges of $P_{\ex}$
and $F^*$ from $F$ 
by deleting all vertices of $A^*$.
\item Define disjoint independent sets
$C_1,\dots,C_{i^*}$ in $F^*$ as follows. 
At step $i \ge 1$,
let $B_i = V(F^*) \sm \bigcup_{j<i} C_j$,
let $C'_i$ be the set of $v \in B_i$
with $d_{F^*[B_i]}(v) \le 3$ and 
$d_{F^*[\bigcup_{j<i} C_j]}(v) \le \pmax^{-1}$,
and let $C_i$ be a maximum independent set in $F^*[C'_i]$.
If $|C_i| < \eps n$ let $i^*=i-1$ and stop,
otherwise go to the next step.
\item 
Let $A_0 = \text{span}^4_T[A^* \cup B_{i^*+1}]$.
Let $A^{**} = \text{span}^4_T(A^D) \sm A^*$
and $A'_0=A_0 \sm (A^* \cup A^{**})$.\\
For $i \in [i^*]$ let $A_i = C_{i^*+1-i} \sm A_0$ and 
for $k \in \mb{N}$ let
$A^k_i = \{a \in A^k: |N_F(a) \cap A_i| \ge \DD\}$.\\
For $a \in A^\DD_i$ let $A^a_i = N_F(a) \cap A_i$.
Let $A^{\hi}_i = \bigcup_{a \in A^\DD_i} A^a_i$ and
$A^{\leL}_i = \bigcup_{a \in A^\DD_i \sm A^\LL_i} A^a_i$ 
and $A^{\geL}_i = \bigcup_{a \in A^\LL_i} A^a_i$.
Obtain $F'$ from $F$ by deleting all edges $ab$
with $a \in A^\DD_i$ and $b \in A^a_i$ for some $i$.
Let $A^{\lo}_i = \{u \in A_i: |N_{F'}(u) \cap A_0| = 1\}$.
Let $A^{\no}_i = \{u \in A_i: N_{F}(u) \cap A_0 = \es\}$.
\item
For $j \in [4]$, let $\dD_j=\dD^{.1j+.6}$ and
let $(\circ_1,\ldots,\circ_4)=
(\no, \geL, \leL, \lo)$.
For each $j \in [4]$,
while any $|A^{\circ_j}_i|<\dD_{j} n$
move $A^{\circ_j}_i$ to $A_0$, let
$A_0 = \text{span}^4_T[A_0 \cup A^{\circ_j}_i]$,
and update $A^{\leL}_i,A^{\geL}_i,
A^{\lo}_i,A^{\no}_i$.
\item If $\bigcup_i A^{\hi}_i=\es$ move $A^*$ to $A^{**}$, i.e.\ 
redefine $A^{**}$ as $A^{**} \cup A^*$ and $A^*$ as $\es$.
\end{enumerate}

Let $A^{\hi}=\bigcup_i A^{\hi}_i$ 
and define $A^{\no}, A^{\lo}$ similarly.

For $k \in \mb{N}$, 
let $Q^{\DD} = \{(a,i) : \DD \le |N_{F}(a) \cap A_i| < \LL\}
\sub A^{\DD} \times [i^*]$
and $Q^{\LL} = \{(a,i) : |N_{F}(a) \cap A_i| \ge \LL\}
\sub A^{\LL} \times [i^*]$.
We introduce parameters 
\begin{gather*}
m^a_i = \bcl{\DD^{-.2} |A^a_i|} 1_{a \in A^\DD_i}, \quad
m_a = \sum_i m^a_i, \quad
m =  \sum_{a \in A^\DD} m_a.
\end{gather*}
Let $\prec$ be an order on $V(T)$ with 
$A^* \prec A^{**} \prec A'_0 \prec A_1 
\prec \dots \prec A_{i^*} \prec V(P_{\ex}) \sm V(T)$
and $|N_<(v) \cap X| \le 1$ whenever 
$v \in X \in \{A^*,A^{**},A'_0\}$.
For $v \in V(T)$ we let ${<}v = \{u: u \prec v\}$,
$N_<(v) = N_{F'}(v) \cap {<}v$,
$N_\le(v) = N_<(v) \cup \{v\}$ and
$N_>(v) = N_{F'}(v) \sm {<}v$.

We stress the use of $F'$ in this notation,
which ensures that $N_>(a) \cap A^{\hi}_i = \es$ 
for all $a \in A_0$: otherwise we would have
a vertex not in $A_0$ adjacent to two vertices of $A_0$,
but this contradicts the definition of $A_0$ as a span.
We list here some other immediate consequences
of the definition of $A_0$
that will often be used without comment.
\begin{itemize}
\item $|A^*| \le 5n/\DD$ and $|A^{**}| \le 5n/D$.
\item Any $u \in A_{\ge 1}$ has $|N_<(u) \cap A_0| \le 1$.
\item Any $uv \in F[A_{\ge 1}]$ 
has $|(N_<(u) \cup N_<(v)) \cap A_0| \le 1$.
\item There is no ${\le}3$-path in $T \sm A_0$
with both ends in $A^{\hi} \cup A^{\lo}$.
\end{itemize}
We also note that $|N_>(v)| \le \pmax^{-1}$
for all $v \in A_{\ge 1}$,
and $|N_<(v)| \le 4$ for all $v \in V(T)$.
To see the latter, note that if $v \in A_{\ge 1}$
then $v$ has at most $3$ earlier neighbours 
in $A_{\ge 1}$ and at most one in $A_0$,
whereas if $v \in A_0$ then $v$ has at most one
earlier neighbour in each of $A^*$, $A^{**}$ and $A'_0$.

Write $n = mn_* + n_0$ with $|n_0 - n\DD^{-.1}| < m$.
Recall that we adopt the
natural cyclic orders on $[m]$ and $[n]$, 
addition wraps, and $d(\cdot,\cdot)$ is cyclic distance.
Whenever an algorithm is required to make a choice,
it aborts if it is unable to do so
(we will show whp it does not abort).

Given bipartite graphs
$B,Z \sub X \times Y$ with $|X|=|Y|$
we write $M{=}$MATCH$(B,Z)$
to mean that $M$ is a random perfect matching
from Lemma \ref{lem:match}.
(The choice of $Z$ will ensure 
edge-disjointness of the embeddings.)

\vspace{-0.1cm} \begin{center} 
HIGH DEGREES \end{center} \vspace{-0.3cm}

\begin{enumerate}
\item
Choose $x_a \in [m]$ for $a \in A^*$ 
in $\prec$ order, arbitrarily subject to
$d(x_a,x_{a'}) > 3d$ for all $a' \prec a$, and
$d(x_a,x_{a'}) \ne d(x_b,x_{b'})$ 
for all $a' \in N_<(a)$ and $bb' \in F[{<}a]$.
\item 
Choose independent uniformly random partitions
of $V(G)$ into $V_0$ of size $n_0$ and
$V_{v^*}$, $v^* \in V^*$ of size $n_*$,
and $W$ into $W_0$ of size $n_0$ and 
$W_{w^*}$, $w^* \in W^*$ of size $n_*$,
where $V^*=W^*=[m]$.
\item
For each $a \in A^*$ in $\prec$ order 
we will define all $\phi_w(a)$ by choosing
a perfect matching $M_a = \{ \phi_w(a) w: w \in W \}$.
Let $B_a \sub V \times W$ consist of all $vw$
where $v \notin \im \phi_w$ and each $\phi_w(b) v$ 
with $b \in N_<(a)$ is an unused edge of $G$.
Let $Z_a \sub V \times W$ consist 
of all $\phi_w(b)w$ with $b \in N_<(a)$.\\
Let $B^0_a=B_a[V_0,W_0]$ and
$B^{w^*}_a=B_a[V_{x_a+w^*},W_{w^*}]$ for $w^* \in W^*$.
Define $Z^0_a$ and $Z^{w^*}_a$ similarly. \\
Let $M_a = M^0_a \cup M^*_a$
with $M^0_a{=}$MATCH$(B^0_a,Z^0_a)$
and $M^*_a{=}\bigcup_{w^*}$MATCH$(B^{w^*}_a,Z^{w^*}_a)$.
\end{enumerate}

We randomly identify $V(G)$ with $[n]$,
cyclically ordered as above. Recall that
each $x \in [n]$ has successor $x^+=x+1$ 
(where $n+1$ means $1$) and predecessor $x^-=x-1$ 
(where $0$ means $n$).
Let $d_i=d/(2s)^{i-1}$ for $i \in [2s+1]$.
We write $n=r_id_i+s_i$ with $r_i \in \mb{N}$
and $0 \le s_i < d_i$, and let
\[ P^i_j = \left\{ \begin{array}{ll}
\{ kd_i+j: 0 \le k \le r_i \} 
& \text{if } j \in [s_i], \\
\{ kd_i+j: 0 \le k \le r_i-1 \} 
& \text{if } j \in [d_i] \sm [s_i].
\end{array} \right. \] 
For each $i \in [s+1]$ and $j \in [d_i]$
we define a partition of $[n]$ into a family
of cyclic intervals $\mc{I}^i_j$ defined 
as all $[x,y^-]$ where $x \in P^i_j$ and $y$ is
the next element of $P^i_j$ in the cyclic order.
(So $|\mc{I}^i_j|=n/d_i \pm 1$,
each $I \in \mc{I}^i_j$ has $|I| \le d_i$, and 
$\mc{I}^i_j \cap \mc{I}^i_{j'} = \es$ for $j \ne j'$.)
We let $\mc{I}^i = \bigcup_{j \in [d_i]} \mc{I}^i_j$.
(So for every $z \in [n]$, 
exactly one $[x,y^-] \in \mc{I}^i$ has $x=z$,
and exactly one $[x,y^-] \in \mc{I}^i$ has $y=z$.)

\vspace{-0.1cm} \begin{center} 
INTERVALS \end{center} \vspace{-0.3cm}

\begin{enumerate}
\item In Case S let $\ovXw = V \sm \phi_w(A^*)$,
 $\ov{p}_w := n^{-1}|\ovXw|$
for all $w \in W$ and go to DIGRAPH;
otherwise (in Case P) continue.
For each $w \in W$ independently choose
$i(w) \in [2s+1]$ and $j(w) \in [d_{i(w)}]$ 
uniformly at random. Let $W_i = \{ w: i(w)=i \}$.
\item For each $w \in W$, let $\mc{A}_w$ include each interval
of $\mc{I}^{i(w)}_{j(w)}$ independently with probability $1/2$. \\
Let $\mc{S}_w$ consist of all $I \in \mc{A}_w$ such that
both neighbouring intervals $I^\pm$ of $I$ are not in $\mc{A}_w$.
\item For each $w \in W$,
let $\mc{X}_w$ include each $I \in \mc{S}_w$
with probability 
$(1-\eta)n^{-1}|P_{\ex}|$
independently, let $X_w = \bigcup \mc{X}_w$,
$\ovXw = V \sm ( \phi_w(A^*) \cup X_w \cup (X_w)^+)$
and $\ov{p}_w = n^{-1}|\ovXw|$.
\item Obtain $\mc{Y}_w \sub \mc{X}_w$ as follows.
Remove any $I$ from $\mc{X}_w$ that intersects $\phi_w(A^*)$,
let $t_i = \min \{ |\mc{X}(I)|: I \in \mc{I}^i \}$,
where $\mc{X}(I) := \{w \in W_i: I \in \mc{X}_w\}$, 
then delete each $I \in \mc{I}^i$, $i \in [2s+1]$ from
$|\mc{X}(I)|-t_i$ sets $\mc{X}_w$ with $w \in \mc{X}(I)$, 
independently uniformly at random. \\
Let $Y_w = \bigcup \mc{Y}_w$ and
$\mc{Y}(I) = \{w \in W_i: I \in \mc{Y}_w\}$.
\end{enumerate}

\vspace{-0.1cm} \begin{center} 
EMBED $A_0$ \end{center} \vspace{-0.3cm}

\begin{enumerate}
\item
For each $xy \in G^*:=G \sm \bigcup_w \phi_w(T[A^*])$
independently let $\mb{P}(xy \in G_0)=p_0/p$.\\
For each $w \in W$ and $x \in \ovXw$ independently
let $\mb{P}(\ova{xw} \in J_0) = p_0/\ov{p}_w$.
\item
Extend the embeddings $\phi_w$ of $T[A^*]$
to $T[A_0]$ in $\prec$ order, where for each
$a \in A_0 \sm A^*$ we choose a perfect matching
$M_a = \{ \phi_w(a) w: w \in W \}
{=}$MATCH$(B_a,Z_a)$,
where $Z_a = \{ \phi_w(b)w: b \in N_<(a) \}$
and $B_a \sub V \times W$ 
consists of all $vw$ with 
$v \in N_{J_0}(w) \sm \im \phi_w$
where each $\phi_w(b) v$ 
with $b \in N_<(a)$ is an unused edge of $G_0$.
\end{enumerate}

For $i,i' \in [i^*]$ and $g,g' \in \{ \hi,\lo,\no\}$,
let $p^{gg'}_{ii'} = n^{-1}|F'[A^g_i,A^{g'}_{i'}]| + \pmin$
and $p^g_{i0} = n^{-1}|F'[A^g_i,A_0]| + \pmin$.
We also write $p^{gg'}_{i0}=p^g_{i0}$
for all $g'$ for uniform notation later.

For $i \in [i^*]$, $g \in A^\DD_i \cup \{\lo,\no\}$
let $\aA^g_i = |A^g_i|n^{-1}$, $\aA_{\lo} = |A^{\lo}|n^{-1}$,
$\aA_{\no} = |A^{\no}|n^{-1}$ and $\aA_0=|A_0|n^{-1}$. 

Let $\aA^{\hi}_i = \DD^{.2}m_i/n$ and
 $\aA_{\hi} = \sum_i \aA^{\hi}_i = \DD^{.2}m/n 
= |A^{\hi}|n^{-1} \pm \DD^{-.9}$.

Let $p_{\ex}=n^{-1}|P_{\ex}|$.
Let $p'_{\ex}=(\tfrac78 -\eta)p_{\ex}$ in Case P
or $p_{\ex} \ll p'_{\ex} \ll 1$ in Case S.

We note some identities and estimates for our parameters: 
\begin{gather*}
p(n-1)/2=|T|=|T[A_0]|+|F'|+|A^{\hi}|+|P_{\ex}|, \\
1+p(n-1)/2=|V(T)|=|V(F)|+|V(P_{\ex}) \sm V(F)|, \\
\sum_{i,i',g,g'} p^{gg'}_{ii'} - n^{-1}|F'| \in [0,{\pmin}^{.9}], 
\quad \sum_i \aA^g_i = n^{-1}|A^g|, \\
p/2 - \sum p^{gg'}_{ii'} - \aA_{\hi} - p_{\ex}
\in [0,{\pmin}^{.9}], \quad
p/2 - (\aA_{\hi}+\aA_{\lo} + \aA_{\no}+p_{\ex})
\in [0,\eps^{.9}], \\
\ov{p}_w = \begin{cases}
1-|A_0|n^{-1}\text{ in Case S, or } \\
\ov{p} \pm d^{-.9} \text{ with }
\ov{p} = (1-\aA_0)(1-(1-\eta)p_{\ex}/8) 
\text{ in Case P}. \end{cases}
\end{gather*}
These estimates imply that the assignment
of probabilities to mutually exclusive events
in DIGRAPH.vii below are valid
(i.e.\ have sum $\le 1$). 
For $\circ \in \{\leL,\geL\}$, let
\begin{gather*}
m_{\circ} = 
 \sum_{(a,i) \in Q^{\circ}} m^a_i, \quad
p_{\circ} = m_{\circ}/m, \quad \text{and define}\\
\text{ labels }
L_{\circ} = \{ \ell_{aij}: (a,i)\in Q^{\circ}, 
j \in [M^a_i] \},
\text{ where }
M^a_i \in \left\{ \bfl{m^a_i/p_{\circ}},
\bcl{m^a_i/p_{\circ}}\right\}
\quad \text{and } |L_{\circ}|=m.
\end{gather*}

\vspace{-0.1cm} \begin{center} 
DIGRAPH \end{center} \vspace{-0.3cm}

\begin{enumerate}
\item
For each $a \in A^\DD$ let $M'_a$ 
denote the perfect matching between $V^*$ 
and $W^*$ consisting of all $v^*w^*$ 
with $w^* \in W^*$ and $v^* = x_a + w^* \in V^*$.
Let $B'_{ai}$ be the bipartite multigraph 
formed by $M^a_i$ copies of $M'_a$ 
labelled by $\ell_{aij}$, $j \in [M^a_i]$.
For $k \in \mb{N}$ let
$B'_k = \bigcup_{(a,i) \in Q^k} B'_{ai}$
and $B_k$ be the bipartite multigraph formed by 
$M^a_i$ copies of $M^*_a$ for each $(a,i) \in Q^k$
labelled by $\ell_{aij}$, $j \in [M^a_i]$.
\item
Let $M'_\LL$ be a largest matching in $B'_\LL$ 
with at most one edge of each label.
Define a partial $m$-edge-colouring
$(M^h_{\geL}: h \in [m])$ of $B_\LL$, where 
for each $h \in [m]$ and edge $v^*w^*$ of $M'_\LL$ 
with some label $\ell_{aij}$ we include in $M^h_{\geL}$
all edges of $M^*_a$ with label $\ell_{aij}$  
between $V_{v^*+h}$ and $W_{w^*+h}$.
\item 
Let $({\cal H},\oO)$ be the weighted $3$-graph 
where for each $v^*w^*$ labelled $\ell_{aij}$ 
with $(a,i) \in Q^\DD$ we include
$v^* w^* \ell_{aij}$ with weight $m^{-1}$.
Let ${\cal M}$ be a random matching obtained from 
Lemma \ref{lem:wEGJ} applied to $({\cal H},\oO)$.
Define matchings $M^h_{\leL} \sub B_\DD$
for $h \in [m]$, where for each edge $v^*w^*\ell_{aij}$ 
of ${\cal M}$ we include in $M^h_{\leL}$
all edges of $M^*_a$ with label $\ell_{aij}$  
between $V_{v^*+h}$ and $W_{w^*+h}$.
\item
Partition $V$ as $(U_h: h \in [m])$ 
uniformly at random.
Fix distinct $h_{aij} \in [m]$ for each $\ell_{aij} \in L_{\circ}$
and $\circ \in\{\leL,\geL\}$ (recalling $|L_{\circ}|=m$).
For all $\ell_{aij} \in L_{\circ}$,
add a copy of $M^0_a$ with every edge labelled 
$\ell_{aij}$ to $M^{h_{aij}}_{\circ}$.
Let $p_1=p-p_0$ and
$\ova{G}_1$ be a uniformly random orientation 
of $G_1 := G^* \sm (G_0 \cup \{xy: d(x,y) \ge 3d\})$. 
\item
For $h \in [m]$ and $x \in U_h$ let $D^h_x$ be the
graph on $N^-_{\ova{G}_1}(x)$ consisting of 
all $yy'$ with $y \neq y'$ such that $yw \in M^h_{\geL}$
and $y'w \in M^h_{\leL}$ 
for some $w \in W$ with $x \in \ovXw$.
For each connected component $P$ of $D^h_x$
independently choose one of $\mb{P}(\hi_P = \circ) = p_\circ$ 
for $\circ \in \{\geL,\leL\}$. For each $y \in P$
with some $yw \in M^h_{\hi_P}$ with label $\ell_{aij}$
include $\ova{yx}$ in $H^*_{ai}$ and let $w(\ova{yx}) = w$.
\item
For each $\ova{yx} \in \ova{G}_1$ 
independently choose at most one of 
$\mb{P}(\ova{yx} \in \ova{G}_{\ex})=2p_{\ex}/p_1$ 
or $\mb{P}(\ova{yx} \in \ova{G}^{gg'}_{ii'})=2p^{gg'}_{ii'}/p_1$
for $1\le i' < i \le i^*$ and $g,g'\in\{\hi,\lo,\no\}$,
or $\mb{P}(\ova{yx} \in \ova{G}^{g}_{i0})=2p^{g}_{i0}/p_1$
for $i \in [i^*]$ and $g\in\{\hi,\lo,\no\}$,
or $\mb{P}(\ova{yx} \in \ova{G}'_i)=2\pmax/p_1$
for $i \in [i^*]$,
or $\mb{P}(\ova{yx} \in H) = 2\aA_{\hi}/p_1\ov{p}_w$
if $x \in \ovXw$, where $w=w(\ova{yx})$,
and if $\ova{yx} \in H^a_i := H \cap H^*_{ai}$
include $\ova{xw} \in J^a_i$.
Let $J^{\hi}_i = \bigcup_{a \in A^\DD_i} J^a_i$ and 
$J^{\hi} = \bigcup_i J^{\hi}_i$
and $\ova{G}_i=\bigcup_{g,g',j}\ova{G}^{gg'}_{ij}$.
\item
Let $J'$ be the set of $\ova{xw} \notin J_0 \cup J^{\hi}$
with $x \in \ovXw$. For $\ova{xw} \in J'$
let $p_{xw}=\ov{p}_w p_1 - 2\aA_{\hi} \hi_{xw}$,
where $\hi_{xw}$ is $1$ if
$w=w(\ova{yx})$ for some $y$ or $0$ otherwise.
For each $\ova{xw} \in J'$
independently choose at most one of 
$\mb{P}(\ova{xw} \in J_{\ex}) = p'_{\ex}/p_{xw}$ 
or $\mb{P}(\ova{xw} \in J^{\lo}_i) = \aA^{\lo}_i/p_{xw}$ 
or $\mb{P}(\ova{xw} \in J^{\no}_i) = \aA^{\no}_i/p_{xw}$ 
or $\mb{P}(\ova{xw} \in J'_i) = \pmax/p_{xw}$.
Let $J^{\lo}=\bigcup_i J^{\lo}_i$
and $J^{\no}=\bigcup_i J^{\no}_i$.
\item In Case P, for each $\ova{yx} \in \ova{G}_{\ex}$ 
independently let
$\mb{P}(\ova{yx} \in J^0_{\ex}) = \tfrac78$ or
$\mb{P}(\ova{yx}^- \in J^K_{\ex}) = \tfrac18$.
\end{enumerate}

Some edges of $G$ may not be allocated by this process.
Note that arcs in $J[V,W]$ are all directed from
$V$ to $W$, so we will often suppress the direction
and think of $J[V,W]$ as a graph.

For $uu' \in F'[A^g_i,A^{g'}_{i'}]$ with
$i,i' \in [0,i^*]$ and $g,g' \in \{ \hi,\lo,\no\}$ 
let every $\ova{G}^{gg'}_{i0}=\ova{G}^g_{i0}$
and let $\ova{G}_{uu'} = \ova{G}^{gg'}_{ii'}$
and $p_{uu'}=p^{gg'}_{ii'}$,
recalling that $p^{gg'}_{i0}=p^g_{i0}$ for all $g'$.

For $g \in A^\DD \cup \{\lo,\no\}$, $u \in A^g_i$
let $A_u = A^g_i$, $J_u=J^g_i$, $\aA_u=\aA^g_i$.

For $\ova{xw} \in J_u$ we also write 
$A_{\ova{xw}}=A_u$, $J_{\ova{xw}}=J_u$, 
$\aA_{\ova{xw}}=\aA_u$.

\vspace{-0.1cm} \begin{center} 
APPROXIMATE DECOMPOSITION
\end{center} \vspace{-0.1cm}

For $i=1,\dots,i^*$ apply the following steps.
\begin{enumerate}
\item Let $({\cal H}_i,\oO)$
be the weighted hypergraph ${\cal H}_i$
with vertex parts $\ova{G}_i$, $J_i$ and 
$A_i \times W$, where 
for each $u \in A_i$ and $\ova{xw} \in J_u$ such that 
$\lova{xy} \in \ova{G}_{uv}$ for all $y=\phi_w(v)$ 
with $v \in N_<(u)$
we include an edge labelled $\edge{w}{u}{x}$
consisting of $uw$, $\ova{xw}$ and all such $\lova{xy}$,
with weight \[ \oO(\edge{w}{u}{x}) := 
 |A_u|^{-1} \prod_{v \in N_<(u)} p_{uv}^{-1}.\]
\item Define $\oO'$ on ${\cal H}_i$ 
by $\oO'(\bm{e}) = (1-.5\eps_i)\oO(\bm{e})/Q(\bm{e})$
where $Q(\bm{e})$ is the maximum of $1$ and all 
$\oO({\cal H}_i[\bm{v}]) :=
\sum \{ \oO(\bm{e}) : \bm{v} \in \bm{e} \}$ 
with $\bm{v} \in \bm{e}$.
Let ${\cal M}_i$ be a random matching obtained
by applying Lemma \ref{lem:wEGJ} to $({\cal H}_i,\oO')$.
For each $\edge{w}{u}{x}$ in ${\cal M}_i$ 
extend $\phi_w$ by setting $\phi_w(u)=x$.
\item
For each $a \in A_i$ in any order,
let $W_a = \{w \in W: \phi_w(a)$ undefined$\}$,
let $V_a \in \tbinom{V}{|W_a|}$
be uniformly random,
and define $\{ \phi_w(a) w: w \in W_a\}
{=}$MATCH$(B_a,Z_a)$,
where $Z_a = \{ \phi_w(b)w: b \in N_<(a) \}$
and $B_a \sub V_a \times W_a$ 
consists of all $vw$ with 
$v \in N_{J'_i}(w) \sm \im \phi_w$
and each $\phi_w(b) v$ for $b \in N_<(a)$ 
an unused edge of $G'_i$.
\end{enumerate}

\medskip

To avoid confusion, we emphasise that $H_i$ 
is a digraph and ${\cal H}_i$ is a hypergraph.
We sometimes use bold font as above
to avoid confusion between 
$\bm{v} \in V({\cal H}_i)$
and $v \in V(H_i) = V(G)$.
We define `time' during the algorithm by a parameter $t$
taking values in a set $\mc{T}$ with the following elements:
$0$ is the start, $t_a$ for $a \in V(T)$
is the time (if it exists) at which some $\phi_w(a)$ 
are defined by choosing a matching $M_a$,
$t_{\hi}$ is the end of HIGH DEGREES,
$t_{\iv}$ is the end of INTERVALS,
$t_{G_0}$ is after choosing $G_0$ and $J_0$,
$t_{**}$ is the end of embedding $A^{**}$,
$t_0$ is the end of EMBED $A_0$,
times $t_i$ and $t^+_i$ for $i \in [i^*]$ 
are just before and just after we extend the
embeddings according to the matching ${\cal M}_i$
(so $t_1$ is the end of DIGRAPH).
For any time $t \ne 0$ we let $t^-$ 
be the time just before $t$.

We write $\mb{P}^t$ and $\mb{E}^t$ for
conditional probability and expectation 
given the history of the algorithm up to time $t$.
For $t \in \mc{T}$ and $w \in W$ let $A_{t,w}$ 
be the set of $w$-embedded vertices at time $t$. 
We write $A_t$ if it is independent of $w$.

We denote the graph remaining after
the approximate decomposition by
$G'_{\ex} = G \sm \bigcup_{w \in W} \phi_w(F)$.

We complete the $T$-decomposition of $G$ 
by the `exact step' algorithms below:
we apply SMALL STARS in Case S, PATHS in Case P,
or LARGE STARS in Case L.

\vspace{-0.1cm} \begin{center} 
SMALL STARS \end{center} \vspace{-0.3cm}

\begin{enumerate}
\item 
For $x \in V(G)$ let $L_x$ 
be the set of all $uw$ where $u$ 
is a leaf of a star in  $P_{\ex}$
with centre $\phi_w^{-1}(x)$.
\item
Let $D$ be a uniformly random 
orientation of $G'_{\ex}$.
While not all $d^+_D(x)=|L_x|$,
choose uniformly random $x,y,z$ 
with $|L_x|>d^+_D(x)$, $|L_y|<d^+_D(y)$,
$z \in N^+_D(y) \cap N^-_D(x)$ 
and reverse $\ova{yz}$, $\ova{zx}$.
\item
For each $x \in V(G)$ in arbitrary order,
define $\phi_w(u)$ for all $uw \in L_x$ by  
$M_x = \{ \{uw,\phi_w(u)\} : uw \in L_x \}
{=}$MATCH$(F_x,\es)$, where
$F_x \sub L_x \times N^+_D(x)$
consists of all $\{uw,y\}$ with $uw \in L_x$, 
$y \in N^+_D(x) \cap N_{J_{\ex}}(w) \sm \im \phi_w$.
\end{enumerate}

\vspace{-0.3cm} \begin{center} 
PATHS \end{center} \vspace{-0.3cm}

\begin{enumerate}
\item 
Call $x \in V(G)$ odd if the parity 
of $d_{G'_{\ex}}(x)$ differs from that
of the number of $w$ such that $x=\phi_w(a)$
where $a$ is the end of a bare path in $P_{\ex}$.
Let $X$ be the set of odd vertices.
Let $a_1 \ell_1$, $a_2 \ell_2$ be the leaf edges
in $P_{\ex}$, with leaves $\ell_1$, $\ell_2$.
Throughout, let $G_{\free} = \{$unused edges$\}$.
\item
Define all $\phi_w(\ell_1)$ by
$M_1 = \{ \phi_w(\ell_1) w: w \in W \}
{=}$MATCH$(B_1,Z_1)$, where 
$Z_1 = \{ \phi_w(a_1)w \}_{w \in W}$
and $B_1 = \{ vw: v \in N_{J_{\ex}}(w),
 v\phi_w(a_1) \in G_{\free} \}$.
\item
Fix $X' \sub X$, $W' \sub W$ 
with $|X'|=|W'|=|X|/2$.
Define $\phi_w(\ell_2)$ for $w \in W'$ by
$M'_2 = \{ \phi_w(\ell_2) w: w \in W' \}
{=}$MATCH$(B'_2,Z'_2)$, where 
$Z'_2 = \{ \phi_w(a_2)w \}_{w \in W'}$
and $B'_2 = \{ vw: w \in W',
 v \in N_{J_{\ex}}(w) \cap X',
 v\phi_w(a_2) \in G_{\free} \}$.
\item
Let $V' = (V \sm X) \cup X'$.
Define $\phi_w(\ell_2)$ for $w \in W \sm W'$ 
by $M_2 = \{ \phi_w(\ell_2) w: w \in W \sm W' \}
{=}$MATCH$(B_2,Z_2)$, where 
$Z_2 = \{ \phi_w(a_2)w \}_{w \in W \sm W'}$
and $B_2 = \{ vw: w \in W \sm W',
 v \in N_{J_{\ex}}(w) \cap V',
 v\phi_w(a_2) \in G_{\free} \}$.
\item
For each $w \in W$ fix 
$8d(x,y)$-paths $P^{xy}_w$
for each $[x,y] \in {\cal Y}_w$
centred in vertex-disjoint
bare $(8d(x,y)+2)$-paths in $P_{\ex}$.
Extend each $\phi_w$ to an embedding
of $P_{\ex} \sm \bigcup_{xy} P^{xy}_w$
so that $\phi_w^{-1}(x)$, $\phi_w^{-1}(y^+)$
are the ends of $P^{xy}_w$,
according to a random greedy algorithm,
where in each step, in any order, 
we define some $\phi_w(a)=z$, 
uniformly at random with 
$z \in J_{\ex}(w) \sm \im \phi_w$
and $zz' \in G_{\free}$ whenever
$z'=\phi_w(b)$ with $b \in N_T(a)$.
\item
Apply Theorem \ref{decompK} to decompose $G_{\free}$
into $(G_w: w \in W)$ such that each $G_w$
is a vertex-disjoint union of $8d(x,y)$-paths
$\phi_w(P^{xy}_w)$, $[x,y] \in {\cal Y}_w$
internally disjoint from $\im \phi_w$.
\end{enumerate}

\vspace{-0.3cm} \begin{center} 
LARGE STARS \end{center} \vspace{-0.3cm}

\begin{enumerate}
\item Let ${\cal S}$ be the union of all maximal
leaf stars in $T$ that have size $\geL$.
Let $F = T \sm {\cal S}$. \\
Let $S$ be the set of star centres of ${\cal S}$
and $S^+ = \{v \in V(T): d_T(v) \geL\}$. \\
Partition $W$ as $W_1 \cup W_2 \cup W_3$
with each $||W_i|-n/3|<1$. For each $v \in V(G)$ 
independently choose exactly one of 
$\mb{P}(v \in U^a_i) = d_{\cal S}(a)/3|{\cal S}|$ 
with $a \in S$, $i \in [3]$. Let $U_i = \bigcup_a U^a_i$.\\
While $\sum_{i=1}^3 ||W_i|-|U_i||>0$ 
relocate a vertex between the $U^a_i$
so as to decrease this sum.
\item Fix an order $\prec$ on $V(F)$ 
starting with some $u_0 \in S^+$ such that
$N_<(u) = \{v \prec u: vu \in F\} = \{u^-\}$
has size $1$ for all $u \ne u_0 \in V(F)$.
Fix distinct $\phi_w(u_0)$, $w \in W$
with $\phi_w(u_0) \in U_i$ whenever $w \in W_i$.
\item Throughout, update $G_{\free} = \{$unused edges$\}$,
the image $\im \phi_w$ of $\phi_w$,
and a digraph $J$ on $V(G)$ consisting 
of all $\ova{yx}$ with $y=\phi_w(a)$ 
and $x \in U^a \cap \im \phi_w$ 
for some $w \in W$, $a \in S$.
\item For each $a \in V(F) \sm \{u_0\}$ 
in $\prec$ order let
$M^a_i = \{ \phi_w(a)w: w \in W_i\}
{=}$MATCH$(B^a_i,Z^a_i)$, $i \in [3]$,
thus defining all $\phi_w(a)$,
with $B^a_i,Z^a_i$ as follows.
\begin{itemize}
\item If $a \notin S^+$ let 
$Z^a_i = \{ \phi_w(a^-)w \}_{w \in W_i}$
and define
$B^a_i \sub U_{i-1} \times W_i$ (with $U_0:=U_3$) by
\[ N_{B^a_i}(w) = A^w_a := 
U_{i-1} \cap N_{G_{\free}}(\phi_w(a^-)) 
\sm \im \phi_w. \]
If $|U_{i-1}|<|W_i|$,
choose $\phi_w(a)w \in B^a_i$
uniformly at random, update $B^a_i$
and remove $w$ from its vertex set.
If $|U_{i-1}|>|W_i|$, remove some 
randomly chosen $u \in U_{i-1}$
from $B^a_i$.
\item If $a \in S^+$ let
$Z^a_i = \big\{ vw: v \in \{ \phi_w(a^-)\} 
\cup (U^a \cap \im \phi_w) \big\}_{w \in W_i}$ 
and define
$B^a_i \sub U_i \times W_i$ by 
$N_{B^a_i}(w) = A^w_a = \bigcup_{b \in S} A^{wb}_a$
where \[  A^{wb}_a =  U^b_i \cap N_{G_{\free}}(\phi_w(a^-))
 \sm \big( \im \phi_w \cup N^+_J(\im \phi_w \cap U^a)
  \cup N^-_J(\phi_w(b)) \big). \]
\end{itemize}
\item Orient $G_{\free}$ as
$D = \bigcup_{w \in W} D_w$,
where for each $xy \in G_{\free}$
with $x \in U^a$ and $y \in U^b$,
if $\ova{xy} \in J$ we have
$\ova{yx} \in D_w$ where $\phi_w(a)=y$,
if $\ova{yx} \in J$ we have
$\ova{xy} \in D_w$ where $\phi_w(b)=x$,
or otherwise we make one of these choices
independently with probability $1/2$.
\item While $\Ss := \sum_{w \in W} \sum_{a \in S}
|d^+_{D_w}(\phi_w(a))-d_{\cal S}(a)|>0$,
we fix $u=\phi_w(a)$ with
$d^+_{D_w}(u) < d_{\cal S}(a)$
and $u'=\phi_{w'}(a')$ with
$d^+_{D_{w'}}(u') < d_{\cal S}(a')$,
and apply a uniformly random 
\emph{$xvz$-move} for $uwu'w'$, defined as follows.
Choose $xvz$ with $\{ \lova{vu}', \lova{zu}',
\ova{vx}, \ova{xu} \} \sub D$ unmoved,
with $x \notin \im \phi_w$,
with $v \notin \im \phi_{w'} \cup \im \phi_{w^x}$
where $\phi_{w^x}(b)=x$, $u \in U^b$,
with $u' \notin \im \phi_{w^v}$
where $\phi_{w^v}(c)=v$, $x \in U^c$,
and with $z \in N^+_{D_{w'}}(u') \sm \im \phi_{w^{u'}}$
where $\phi_{w^{u'}}(d)=u'$, $v \in U^d$.
The $xvz$-move for $uwu'w'$
reverses the path $u'vxu$ in $D$, assigning
$\ova{ux} \in D_w$, $\ova{xv} \in D_{w^x}$,
$\ova{vu}' \in D_{w^v}$ and $\lova{zu}' \in D_{w^{u'}}$. 
\end{enumerate}

\subsection{Preliminaries}

Here we gather some well-known results concerning
concentration of probability and Szemer\'edi regularity,
and also a result on random perfect matchings 
in quasirandom bipartite graphs, which is perhaps new
(although the proof technique via switchings
is somewhat standard).

We start with the following classical 
inequality of Bernstein (see e.g.\ \cite[(2.10)]{BLM})
on sums of bounded independent random variables.
(In the special case of a sum of independent indicator
variables we will simply refer to the `Chernoff bound'.)

\begin{lemma} \label{bernstein}
Let $X = \sum_{i=1}^n X_i$ be a sum of
independent random variables with each $|X_i|<b$.
Let $v = \sum_{i=1}^n \mb{E}(X_i^2)$.
Then $\mb{P}(|X-\mb{E}X|>t) 
< 2e^{-t^2/2(v+bt/3)}$.
\end{lemma}

We also use McDiarmid's bounded differences inequality,
which follows from Azuma's martingale inequality
(see \cite[Theorem 6.2]{BLM}). 

\begin{defn} \label{def:vary}
Suppose $f:S \to \mb{R}$ where $S = \prod_{i=1}^n S_i$
and $b = (b_1,\dots,b_n) \in \mb{R}^n$.
We say that $f$ is \emph{$b$-Lipschitz} if for any 
$s,s' \in S$ that differ only in the $i$th coordinate
we have $|f(s)-f(s')| \le b_i$. 
We also say that $f$ is \emph{$v$-varying} 
where $v=\sum_{i=1}^n b_i^2/4$.
\end{defn}

\begin{lemma} \label{azuma}
Suppose $Z = (Z_1,\dots,Z_n)$ is a sequence 
of independent random variables,
and $X=f(Z)$, where $f$ is $v$-varying.
Then $\mb{P}(|X-\mb{E}X|>t) \le 2e^{-t^2/2v}$.
\end{lemma}

We say that a random variable is \emph{$(\mu,C)$-dominated}
if we can write $Y=\sum_{i \in [m]}Y_i$
such that $|Y_i| \leq C$ for all $i$
and $\sum_{i \in [m]}\mb{E}'|Y_i|<\mu$,
where $\mb{E}'|Y_i|$ denotes the expectation
conditional on any given values of $Y_j$ for $j<i$.
The following lemma follows easily from Freedman's inequality~\cite{freedman}.

\begin{lemma}\label{freedman}
If $Y$ is $(\mu,C)$-dominated, then 
$\mb{P}(|Y|>2\mu)<2e^{-\mu/6C}$.
\end{lemma}

Next we recall some definitions
(not quite in standard form)
pertaining to Szemer\'edi regularity.
A bipartite graph $B \sub X \times Y$
with $|B|=d|X||Y|$ is \emph{$\eps$-regular}
if $|B[X',Y']|=d|X'||Y'| \pm \eps |X||Y|$
for every $X' \subseteq X$, $Y' \subseteq Y$.
If also $|B(x) \cap Y|=(1 \pm \eps)d|Y|$
and $|B(y) \cap X|=(1 \pm \eps)d|X|$
for all $x \in X$, $y \in Y$ 
then $B$ is \emph{$\eps$-super-regular}.
We will need the well-known `pair condition' 
discovered independently by several pioneers
in the theory of Szemer\'edi regularity
(we refer to \cite{KR} for the history and a
version of the following statement).

\begin{lemma} \label{lem:DLR}
Let $\eps<2^{-200}$ and  
$B \sub X \times Y$ with $|X|=|Y|=m$, where
$|N_B(x) \cap Y|>(d-\eps) m$ for all $x \in X$ and
$|N_B(xx') \cap Y)<(d+\eps)^2 m$ for all but 
$\le 2\eps m^2$ pairs $xx'$ in $X$.
Then $B$ is $\eps^{1/6}$-regular.
\end{lemma}

We also require the following lemma;
the proof is standard, so we omit it.

\begin{lemma} \label{lem:SRLtoH}
Let $n^{-1} \ll \aA \ll \bB
\ll d, r^{-1}, D^{-1}$ and 
$G$ be an $\aA$-super-regular bipartite graph
with parts $X$ and $Y$ of size $\ge n$
and density  $d(G) \ge d$.
Suppose $H$ is a ${\le}r$-multigraph on $Y$
of maximum degree $D$. Then for all but 
at most $\bB |X|$ vertices $x$ we have
$\sum_{e \in H[N_G(x)]} d(G)^{-|e|}
= |H| \pm \bB |Y|$.
\end{lemma}

Next we present a result
on random perfect matchings 
in super-regular bipartite graphs.
Given $M,Z \sub X \times Y$, 
an $MZMZ$ is a $4$-cycle that
alternates between $M$ and $Z$.
We also write $MZMZ$ 
for the number of $MZMZ$'s.

\begin{lemma} \label{lem:match}
Let $n^{-1} \ll \aA \ll d$ and 
$B,Z \sub X \times Y$ with $|X|=|Y|=n$.
Suppose $Z$ has maximum degree $<n^{.4}$
and $B$ is $\aA$-super-regular 
with density $d(B) \ge d$.
Then there is a distribution on perfect matchings 
$M$ of $B$ with $MZMZ=0$ such that
$\mb{P}(xy \in M) = (1 \pm \aA^{.98})(d(B)n)^{-1}$
for any edge $xy$,
and for any $X' \sub X$, $Y' \sub Y$ whp 
$|M[X',Y']| = |B[X',Y']|(d(B)n)^{-1} \pm n^{.8}$. 
\end{lemma}

\begin{proof}
Let ${\cal M}$ be the set of perfect matchings of $B$.
It is well-known (and easy to see by Hall's theorem)
that ${\cal M} \ne \es$. We consider a Markov chain
on ${\cal M}$ where the transition from any $M \in {\cal M}$
is a uniformly random swap, defined by choosing 
a $6$-cycle $C$ in $B$ that is $M$-alternating
(every other edge is in $M$) and swapping
$C \cap M$ with $C \sm M$, subject to the
new edges $C \sm M$ not forming any new $MZMZ$'s.
It is well-known that every Markov chain
on a finite state space has a stationary distribution
(which is not necessarily unique).
Fix some stationary distribution $\mu$
and let $M \sim \mu$.

To analyse the chain, we start with an estimate
for the number of swaps for any given $M$.
Let $G_M$ be the auxiliary tripartite graph
with parts $X_1,X_2,X_3$ each a copy of $X$,
where for $i \in [3]$, $x_i \in X_i$, 
$x'_{i+1} \in X_{i+1}$ 
we have $x_i x'_{i+1} \in G_M$
if $M(x_i)x'_{i+1} \in B \sm M$ (and $X_4:=X_1$).
Note that $M$-alternating $6$-cycles in $G$
correspond to triangles in $G_M$.
Each $G_M[X_i,X_{i+1}]$ is a copy of $B \sm M$,
so is $2\aA$-super-regular, and so
by the triangle counting lemma $G_M$ 
has $(1 \pm \aA^{.99}) (d(B)n)^3$ triangles.
Each edge in $M$ forms an $MZMZ$
with $\le n^{.8}$ other edges, each forbidding
$\le n$ possible swaps, so the number of swaps
is $(1 \pm \aA^{.99}) (d(B)n)^3 \pm n^{2.8}
= (1 \pm 1.1\aA^{.99}) (d(B)n)^3$.

Next we claim that $\mu$ is supported 
on ${\cal M}_0 := \{M : MZMZ=0\}$.
To see this, first note that in any step 
of the chain $MZMZ$ is non-increasing. 
Also, the $M$-alternating 
$6$-cycles that remove any given $e$ from $M$
correspond to triangles in $G_M$ containing
some given vertex. There are
$(1 \pm \aA^{.99}) d(B)^3 n^2$ such triangles,
of which $\le n^{1.8}$ are forbidden.
Letting $p^{-e}_M$ denote the probability
that $e$ is removed by a transition 
from $M$ we have 
$p^{-e}_M = (1 \pm 2.2\aA^{.99}) n^{-1}$.
In particular, if $MZMZ>0$ then 
it decreases with positive probability.
Thus ${\cal M}_0$ is an absorbing class,
so the claim holds.

Next we estimate $\mb{P}(e \in M)$ 
for any given $e \in B$.
Let ${\cal M}[e] = \{M \in {\cal M}: e \in M\}$.
For $M \in {\cal M} \sm {\cal M}[e]$
let $p^{+e}_M$ denote the probability 
that $e$ is added by a transition from $M$.
The $M$-alternating $6$-cycles 
for adding $e$ correspond to a choice 
in some common neighbourhood
$N_{G_M}(x_1) \cap N_{G_M}(x_2)$.
Thus there are $(1 \pm \aA^{.99})d(B)^2 n$
such $6$-cycles, of which $\le 4n^{.4}$ are forbidden, 
so $p^{+e}_M = (1 \pm 2.2\aA^{.99}) d(B)^{-1} n^{-2}$.
Now \[ \mb{P}(e \in M) 
= \sum_{M \in {\cal M}[e]} \mu_M
= \sum_{M \in {\cal M}[e]} \mu_M (1-p^{-e}_M)
+ \sum_{M \in {\cal M} \sm {\cal M}[e]}
 \mu_M p^{+e}_M,\]
so $\sum_{M\in{\cal M}[e]}\mu_M p^{-e}_M
=\sum_{M\in{\cal M}\sm {\cal M}[e]}\mu_M p^{+e}_M$
and so $(1 \pm 4.5\aA^{.99}) n^{-1}\mb{P}(e \in M)
= d(B)^{-1} n^{-2}\mb{P}(e \notin M)$,
giving $\mb{P}(e \in M)
= (1 \pm 5\aA^{.99})(d(B)n)^{-1}$.

To obtain the final property,
we consider uniformly random partitions
$(X_i: i \in I)$ of $X$
and $(Y_i:i \in I)$ of $Y$
with each $|X_i|=|Y_i|=\sqrt{n} \pm 1$.
We let $M = \bigcup_{i \in I} M_i$
where each $M_i \sim \mu_i$ independently
with $\mu_i$ a stationary distribution
of the above chain for 
$B_i=B[X_i,Y_i]$ and $Z_i=Z[X_i,Y_i]$.
By Chernoff bounds whp each $B_i$
is $1.1\aA$-super-regular with
$d(B_i) = d(B) \pm n^{-.1}$.
By the above analysis, each
$\mb{P}(e \in M) = \sum_i \mb{P}(e \in M_i)
= \sum_i (n^{-1}|X_i|)^2
(1 \pm 5(1.1\aA)^{.99})(d(B_i)|X_i|)^{-1}
= (1 \pm 6\aA^{.99}) (d(B)n)^{-1}$.
 
It remains to estimate 
$|M[X',Y']| = \sum_i M[X'_i,Y'_i]$, where
$X'_i = X' \cap X_i$, $Y'_i = Y' \cap Y_i$.
By Chernoff bounds, whp each 
$B[X'_i,Y'_i] = n^{-1}|B[X',Y']| \pm n^{.76}$,
so \[ \mb{E} |M[X'_i,Y'_i]| 
= (1 \pm 6\aA^{.99}) (d(B)\sqrt{n})^{-1}
(n^{-1}|B[X',Y']| \pm n^{.76}).\]
Also, $\mb{E} |M[X'_i,Y'_i]|^2 
= \mb{E} |M[X'_i,Y'_i]| 
+ \sum_{e \ne e' \in B[X'_i,Y'_i]}
\mb{P}(\{e,e'\} \sub M_i) < 2n$,
as each $\mb{P}(\{e,e'\} \sub M_i)
= (1 \pm 6\aA^{.99})(d(B_i)|X_i|)^{-2}$ 
by similar arguments to those above.
The required estimate for $|M[X',Y']|$
now follows from Lemma \ref{bernstein}.
\end{proof}

We conclude this subsection with a result
on matchings in weighted hypergraphs,
along the lines of the literature stemming
from the R\"odl nibble mentioned in 
the overview above. The following lemma
is a slight adaptation of a convenient
general setting of the nibble recently
provided by Ehard, Glock and Joos \cite{EGJ}.
Given a weighted hypergraph $(H,\oO)$, we call 
a function $f:\tbinom{H}{{\le}r} \to \mb{R}$
clean if $f(I)=0$ whenever $I$ is not a matching.
For $H' \sub H$ let $f(H') = 
 \sum \{ f(E): E \in \tbinom{H'}{{\le}r} \}$,
and $f(H',\oO) = 
 \sum \{ \oO(E) f(E): E \in \tbinom{H'}{{\le}r} \}$,
where $\oO(E) = \prod_{e \in E} \oO(e)$.
For $S,T \in \binom{H}{\le r}$
we also let $f_S(T) = f(S \cup T)$
if $T \cap S = \es$, or $f_S(T)=0$ otherwise.

\begin{lemma} \label{lem:wEGJ}
Let $C^{-1} \ll \aA \ll \bB \ll r^{-1},\ell^{-1}$ and 
$(H,\oO)$ be a weighted ${\le}r$-graph 
with $\oO(e) \ge C^{-1}$ for all $e \in H$,
and $\oO(H[v]) \le 1$, $\oO(H[uv]) < C^{-\bB}$ 
for all $u \ne v \in V(H)$.
Suppose $f$ is a clean function
on $\tbinom{H}{{\le}\ell}$ with
$f_S(H,\oO) \le C^{-\bB} f(H,\oO)$
whenever $S \ne \es$.
Then there is a distribution 
on matchings $M$ in $H$ such that
$f(M) = (1 \pm C^{-\bB}) f(H,\oO)$ 
with probability $\ge 1-e^{-C^\aA}$.
\end{lemma}

The proof of Lemma \ref{lem:wEGJ} 
is essentially the same as that of
\cite[Theorem 1.3]{EGJ},
with a few modifications as follows.
The statement in \cite{EGJ}: 
\begin{itemize}
\item applies to unweighted hypergraphs
of maximum degree $\DD$ 
and maximum codegree $<\DD^{1-\bB}$;
our version can be reduced to this version
by considering a multihypergraph
where the multiplicity of an edge $e$
is $\bfl{\DD^2 \oO(e)}$, say.
\item gives
a (deterministic) matching $M$ satisfying
the required conclusion for a suitably small
set of functions $f$; this is obtained
by proving the existence of a distribution
on matchings as in our statement 
and taking a union bound,
\item applies to functions on $\tbinom{H}{\ell}$, 
from which a version for 
functions on $\tbinom{H}{{\le}\ell}$
is easily deduced.
\end{itemize}

\subsection{Tree partition}

We start our analysis of the algorithm
by considering the subroutine TREE PARTITION.

\begin{lemma} \label{lem:T}
We can choose a case in \{L,S,P\}
for $T$ and we have $i^*<i^+$.
Also, $|A_0| \le 6\eps n$, each $|A^a_i| \ge \DD$ 
and $|A^{\circ_j}_i| \ge \dD n$ if non-empty,
with $A^a_i \cap A^{a'}_{i'} = \es$ for $ai \ne a'i'$.
\end{lemma}

\begin{proof}
To see that we can choose a case for $T$,
we suppose that $T$ does not satisfy
Case L or Case S, and show
that it must satisfy Case P.
Here we rely on the well-known fact that
any tree with few leaves must have many vertices
in long bare paths (we will use the precise statement
given by \cite[Lemma 4.1]{MPS2}).
Let $T'$ be the tree obtained from $T$ by removing
all leaf stars of size $\geL$.
Then $|V(T')| \ge p_+ n$, 
as $T$ does not satisfy Case L.

We claim that $T'$ has $<2p_- n$ leaves.
To see this, let ${\cal S}$ be the set 
of maximal leaf stars of $T'$.
For each $S \in {\cal S}$ obtain $S'$
from $S$ by deleting all leaves of $T'$
that are not leaves of $T$.
Note that $|S'| \le \LL$, or we would have
removed $S$ when defining $T'$. 
Then $\sum_{S \in {\cal S}} |S'| < p_- n$,
as $T$ does not satisfy Case S.
Also, $\sum_{S \in {\cal S}} |S \sm S'| < n/\LL$
as each leaf in any $S \sm S'$ is the centre 
of a leaf star in $T$ of size $\geL$.
The claim follows.

Now \cite[Lemma 4.1]{MPS2} implies that $T'$ has 
$>p_+ n/50K$ vertex-disjoint bare $8K$-paths.
At most $n/\LL$ of these contain the centre of some
star removed when obtaining $T'$ from $T$,
so $>p_+ n/100K$ are bare paths of $T$, as required.

Next we bound $i^*$. Recall that at step $i \ge 1$
we let $B_i = V(F^*) \sm \bigcup_{j<i} C_j$
and $C'_i$ be the set of $v \in B_i$
with $d_{F^*[B_i]}(v) \le 3$ and 
$d_{F^*[\bigcup_{j<i} C_j]}(v) \le \pmax^{-1}$.
We have $|C'_i| > |B_i|/3 - 2\pmax n$,
as $<2\pmax n$ vertices fail the second condition,
and the set $X$ of vertices failing the first condition
satisfies $3|X| \le \sum_{x \in X}d_{F^*[B_i]}(v) < 2|B_i|$.
Next we let $C_i$ be a maximum independent set in $F^*[C'_i]$;
we have $|C_i| \ge |C'_i|/2$ as trees are bipartite.
If $|C_i|<\eps n$ we let $i^*=i-1$ and stop,
otherwise we proceed to the next step,
noting that $|C_i| > |B_i|/7$.
There can be at most $i^+ = 7\log \eps^{-1}$ steps,
otherwise we would continue past a step $i$ 
with $|B_i| < (6/7)^{i^+} n < \eps n$.

For the remaining statements,
we first note that the bounds for
and disjointness of the sets $A^a_i$
are immediate from the algorithm
and the definition of $A_0$ as a span.
Finally we consider step (iv) of TREE PARTITION.
For each $j \in [4]$, there are at most $i^+$
steps where we move some $A^{\circ_j}_i$ 
to $A_0$ if it has size $< \dD_{j} n$,
thus adding $< 5\dD_{j} n$ vertices to $A_0$
after including any forced by the definition as a span.
Note that by choice of the order $\circ_1,\dots,\circ_4$
it is not possible for some $A^{\circ_j}_i$
to be moved to $A_0$ and then 
to reappear at a later step.
At the end of the process, 
any surviving $A^{\circ_j}_{i}$ has size
$|A^{\circ_j}_i| > \dD_{j}n-\sum_{j'>j}5i^+\dD_{j'} n
\ge \dD_j n (1-5(4-j)\dD^{-.1}i_+) \ge \dD n$.
This completes the proof.
\end{proof}

\subsection{High degrees}

Continuing through the algorithm, the following
lemma shows that the subroutine HIGH DEGREES
is whp successful, and the image of each embedding
is well-distributed with respect 
to common neighbourhoods in $G$.

\begin{lemma} \label{lem:hi}
whp HIGH DEGREES does not abort, and
$\mb{P}^{t_a^-}(\phi_w(a) \in N_G(S) ) 
= (1 \pm \dD)p^{|S|}$
for any $a \in A^*$, $w \in W$ 
and $S \sub V(G)$, $|S| \le s$.
\end{lemma}

We make some preliminary observations
before giving the proof. 
First, we write the proof assuming $|N_{<}(a)| \le 4$
for all $a \in A^*$ rather than using our real bound
$|N_{<}(a)| \le 1$, so that it is more obvious
how to apply the same proof to obtain 
Lemma~\ref{lem:rga}.
We note that
the choices of $x_a$ for $a \in A^*$ are possible.
Indeed, at each step, we forbid $\le 6d|A^*| \le 30d n/\DD$  
choices of $x_a$ with $d(x_a,x_{a'}) \le 3d$ 
for some $a' \prec a$, and $< 5 n/\DD$ 
choices with $d(x_a,x_{a'}) = d(x_b,x_{b'})$ 
for some $a' \in N_<(a)$ and $bb' \in F[{<}a]$.
We also note the following estimate for common
neighbourhoods, which is immediate from
a (hypergeometric) Chernoff bound:
whp for any $S \sub V(G)$ with $|S| \le s$ 
and $X=V_0$ or $X=V_{v^*}$ with $v^* \in V^*$ 
we have $|N_G(S) \cap X|
= ((1  \pm 1.1\xi)p)^{|S|} |X|$.
 
\begin{proof}
We can condition on partitions of $V$ and $W$ 
satisfying the above estimates for $|N_G(S) \cap X|$.
First we consider the choices of $(M^*_a: a \in A^*)$,
which are independent of $(M^0_a: a \in A^*)$.
For each $a \in A^*$, $w^* \in W^*$, $v^* = x_a + w^*$
we will show that Lemma \ref{lem:match} applies
to choose $M^*_a[V_{v^*},W_{w^*}]
{=}$MATCH$(B^{w^*}_a,Z^{w^*}_a)$,
where $B^{w^*}_a \sub V_{v^*} \times W_{w^*}$  
satisfies $B^{w^*}_a(w) = 
N_{G_{t_a^-}}(\phi_w(N_<(a))) \cap V_{v^*} \sm \phi_w({<}a)$,
where $G_t \subseteq G$ is the graph of unused edges at time $t$.
For each $v \in V$ there is a unique edge
$w \phi_w(a) \in M_a$ with $\phi_w(a)=v$,
which uses $|N_<(a)|$ edges at $v$,
so $G \sm G_t$ has maximum degree 
$\le |T[A^*]| \le 5n/\DD$.
Note that the constraint that
$\phi_w(a)\phi_w(a')$ is unused for all $a' \in N_<(a)$
is automatically satisfied as
$d(x_a,x_{a'}) \ne d(x_b,x_{b'})$ for all $bb' \in F[{<}a]$.

We also note that
$Z_a = \{\phi_w(b)w:b \in N_<(a)\}$
has maximum degree $\le 4$.

At time $t$, let 
$H^t_{w^*}$ be the hypergraph on $V(G)$ with edges 
$e^t_w = \phi_w(N_<(a) \cap A_t)$ for $w \in W_{w^*}$.
Note that $H^t_{w^*}$ is a matching, 
as $M^*_b$ is a matching for each $b \in N_<(a)$ 
and $\phi_w(b) \in V_{w^*+x_b}$ with distinct $x_b$.
We let ${\cal B}_t$ be the `bad' event that 
$|H^t_{w^*}[N_G(v)]| \neq
((1 \pm 7\xi) p)^{|N_<(a) \cap A_t|} n_*$
for some $v \in V(G)$, $a \in A^*$, $w^* \in W^*$.
We let $\tau$ be the smallest $t$ such that
${\cal B}_t$ occurs, or $\infty$ if there is no such $t$.
We fix $a \in A^*$ and bound $\mb{P}(\tau=t_a)$.

We claim that $B^{w^*}_a$ is $\xi'$-super-regular
of density $(1 \pm 5\xi) p^{|N_<(a)|}$.
To show this, we first lighten our notation,
writing $t=t_a^-$ and $H=H^t_{w^*}$, which has 
edges $e_w = \phi_w(N_<(a))$ for $w \in W_{w^*}$.
Any $v \in V_{v^*}$ has degree 
$|H[N_{G_t}(v)]|\pm |\{w \in W_{w^*}: v \in \phi_w({<}a)\}|
= |H[N_G(v)]|$
as $v\phi_w(a')$ is unused for all $a' \in N_<(a)$
and all $w \in W_{w^*}$, and $v= \phi_w(a'')$
for some $a'' \in {<}a$ and $w \in W_{w^*}$ iff
$x_a=v^*-w^*=x_{a''}$.
As ${\cal B}_t$ does not hold for $t<t_a$,
$v$ has degree
$((1 \pm 7\xi)p)^{|N_<(a)|} n_*$.
Any $w \in W_{w^*}$ has degree
$|N_G(\phi_w(N_<(a))) \cap V_{v^*}|
= ((1  \pm 1.1\xi)p)^{|N_<(a)|} n_*
= (1 \pm 5\xi) p^{|N_<(a)|} n_*$.
For any $V' \sub V_{v^*}$, $W' \sub W_{w^*}$ 
we have $|B^{w^*}_a[V',W']| = \sum_{w \in W'}
|N_G(\phi_w(N_<(a)) \cap V'|
= \sum_{v \in V'} |H'[N_G(v)]|$,
where $H' = \{ e_w : w \in W'\}$, so
$|B^{w^*}_a[V',W']|
= p^{|N_<(a)|} |V'||W'| \pm \xi' n_*^2 $
by Lemma \ref{lem:SRLtoH}. This proves the claim.

Thus Lemma \ref{lem:match} applies, giving
$\mb{P}^{t_a^-}(vw \in M^*_a) 
= (1 \pm \dD/2) (p^{|N_<(a)|} n_*)^{-1}$
for any $vw \in B^{w^*}_a$,
so $\mb{P}^{t_a^-}(\phi_w(a) \in N_G(S) )  
= (1 \pm \dD)p^{|S|}$
for any $w \in W_{w^*}$ and $S \subseteq V(G)$,
$|S| \le s$.

To bound $\mb{P}({\cal B}_{t^a})$, note that
$H^t_{w^*}$ only changes when we choose
$M^*_b$ for $b \in N_<(a)$. 
Fix $v$ and write
$W^t = \{w: e^t_w \in H^t_{w^*}[N_G(v)] \}$,
where $t = t^-_b$.
For any $w \in W^t$ we have 
$e^{t_b}_w \in H^{t_b}_{w^*}$ 
iff $\phi_w(b) \in N_G(v)$,
so $|H^{t_b}_{w^*}|=|M^*_b[W^t,N_G(v)]|$,
which by Lemma \ref{lem:match} is whp 
$|B^{w^*}_a[N_G(v),W^t]|
((1 \pm 5\xi) p^{|N_<(a)|}n)^{-1} \pm n^{.8}
= (1 \pm 7\xi)p|W^t|$.
Thus whp ${\cal B}_{t^a}$ does not hold
for any $a$, so $\tau=\infty$.

Now we consider the choice of 
$M^0_a{=}$MATCH$(B_a^0,Z_a^0)$,
where $Z_a = Z_a[V_0,W_0]$
and $B_a \sub V_0 \times W_0$ is defined by
$B_a^0(w) = N_{G_t}(\phi_w(N_<(a))) \cap V_0
 \sm \phi_w({<}a)$.
Let $H^t_a$ be the hypergraph on $V_0$ with edges 
$e^t_w = \phi_w(N_<(a) \cap A_t) \cap V_0$ for $w \in W_0$.
Let ${\cal B}'_t$ be the `bad' event 
that for some $a \in A^*$ 
we have some $|H^t_a[N_{G_t}(v)]| \neq 
((1 \pm \xi') p)^{|N_<(a) \cap A_t|} n_0$.
We let $\tau'$ be the smallest $t$ such that
${\cal B}'_t$ occurs, or $\infty$ if there is no such $t$.
We fix $a \in A^*$ and bound $\mb{P}(\tau'=t_a)$.

Similarly to the arguments for $M^*_a$,
as ${\cal B}'_t$ does not hold,
$B_a$ is $\xi'$-super-regular
of density $(1 \pm 5\xi) p^{|N_<(a)|}$,
using the maximum degree bound on $G \sm G_t$
to estimate common neighbourhoods 
$|N_{G_t}(S) \cap V_0|$ for degrees of $w \in W_0$,
recalling that $n_0 > n\DD^{-.1}/2$,
and estimating $|B_a[V',W']| 
= \sum_{v \in V'} |H'[N_{G_t}(v)]|$
by Lemma \ref{lem:SRLtoH} applied to $G_t$
and $H' = \{ e_w^t: w \in W' \}$,
which has maximum degree at most $16$,
as for each $b,b' \in N_<(a)$
and $w \in W_0$ there is a unique $w' \in W_0$
with $\phi_w(b)=\phi_{w'}(b')$.

Again we similarly deduce
$\mb{P}^{t_a^-}(vw \in B_a) 
= (1 \pm \dD/2) (p^{|N_{<}(a)|}n_0)^{-1}$
for any $vw \in B_a$,
so $\mb{P}^{t_a^-}(\phi_w(a) \in N_G(S) )  
= (1 \pm \dD)p^{|S|}$
for any $w \in W_0$ and $S \subseteq V(G)$,
$|S| \le s$.
To bound $\mb{P}({\cal B}'_{t^a})$, note that 
the same argument as for $M^*_a$ gives 
whp $|H^t_a[N_G(v)]| =
((1 \pm \xi'/2) p)^{|N_<(a) \cap A_t|} n_0$,
and by the maximum degree bound on $G \sm G_t$
we can replace $G$ by $G_t$ in this estimate,
changing $\xi'/2$ to $\xi'$.
Thus whp ${\cal B}'_{t^a}$ does not hold
for any $a$, so $\tau'=\infty$, as required.
\end{proof}

Note that $G \sm G^*$ has maximum degree 
$\le |T[A^*]| < 5n/\DD$, as for each $v \in V(G)$
there is a unique edge
$w \phi_w(a) \in M_a$ with $\phi_w(a)=v$,
which uses $|N_<(a)|$ edges at $v$.
Thus $G^*$ is $(1.1\xi,s)$-typical,
so whp the graphs $G_0$ and $G_1$ defined in 
EMBED $A_0$ are $(1.2\xi,s)$-typical.

We omit the proof of the following lemma,
as it is similar to and simpler than the previous.

\begin{lemma} \label{lem:rga}
For any $a \in A_0 \sm A^*$, $w \in W$, 
$x,y \in V(G)$, writing $A^w_a$ for the set
of $y$ such that $\phi_w(a)=y$ is possible
given the history at time $t_a^-$,
whp $\mb{P}^{t_a^-}(\phi_w(a)=y)  
= (1 \pm D^{-.9} \pm \aA_0^{.9}1_{a \in A'_0})
|A^w_a|^{-1}$, so whp every
$\mb{P}^{t_a^-}(\phi_w(a) \in N_{G_1}(x) )  
= (1 \pm D^{-.9} \pm \aA_0^{.9}1_{a \in A'_0}) p_1$.
\end{lemma}

\subsection{Intervals}

Next we record some properties of the subroutine
INTERVALS that are needed for the exact step in Case P
(handled by the subroutine PATHS). We omit the proof,
which is essentially the same as that of the corresponding
lemma in \cite{factors} (the only change is the deletion
of the negligible sets $\phi_w(A^*)$).
We say that $S \sub [n]$ is \emph{$d$-separated}
if $d(a,a') \ge d$ for all distinct $a,a'$ in $S$.
For disjoint $S,S' \sub [n]$ we say $(S,S')$
is \emph{$d$-separated} if $d(a,a') \ge d$ 
for all $a \in S$, $a' \in S'$.

\begin{lemma} \label{lem:INT} In Case P,
\begin{enumerate}
\item $\mb{P}^{t_{\hi}}(x \in \ovXw) = \ov{p}_w \pm \DD^{-.9}$
for all $w \in W$ and $x \in V(G)$, 
\item any subset of 
$\{ \{x \in \ovXw\} : w \in W, x \in V(G) \}$  
is independent if it does not include any pair
$\{x \in \ovXw\}$, $\{x' \in \ovXw\}$
with $d(x,x') \le 3d$,
\item whp $|\mc{Y}(I)| = t_i 
= \tfrac{(1-\eta)|P_{\ex}|}{8(2s+1)d_i} \pm n\DD^{-.9}$ 
for all $I \in \mc{I}^i$, $i \in [2s+1]$,
\item whp all $|Y_w| = (1-\eta)|P_{\ex}|/8 \pm n\DD^{-.9}$,
\item for any $U \subseteq V(G)$, 
whp for any disjoint $R,R' \sub W$ of sizes $\le s$ we have
\[ \bsize{U \cap N^-_{J_{\iv}}(R) \cap N^-_{\ov{J}}(R')}
= |U| (\tfrac{1}{8}(1-\eta)p_{\ex})^{|R|}
\prod_{w \in R'}\ov{p}_w \pm n\DD^{-.1} \]
where  $J_{\iv}=\{\ova{xw}: x \in Y_w\}$ and
$\ov{J}=\{\ova{xw}: x \in \ov{X}_w\}$,
\item whp for any disjoint $S,S' \sub V$ of sizes $\le s$,
\begin{align*}
& \text{If } S \cup S' \text{ is } 3d\text{-separated then } 
|\{w: S \sub Y_w, S' \sub \ovXw\}| 
= \sum_{w \in W} (\tfrac{1}{8}(1-\eta)p_{\ex})^{|S|} \ov{p}_w^{|S'|}
 \pm n\DD^{-.1}, \\
& \text{If } (S,S') \text{ is } 3d\text{-separated then } 
|\{w: S \sub Y_w, S' \sub \ovXw\}|   
 \ge 2^{-2s} n (\tfrac{1}{8}(1-\eta)p_{\ex})^{|S|}.
\end{align*}
\end{enumerate}
\end{lemma}

\subsection{Digraph}

Our next lemma summarises various properties of
the decompositions of $G$ and $W \times V(G)$
constructed in the subroutine DIGRAPH.
Many of these properties are straightforward
consequences of the definition and Chernoff bounds.
The most significant conclusion is part (viii),
showing that the high degree digraph $H$ 
allocates roughly the correct number of edges 
to each vertex $x$ for each role $ai$
where $i \in [i^*]$, $a \in A^\DD_i$.
For each such $ai$ we let $M'_{ai}$ 
consist of all $v^*w^*$ with some label $\ell_{aij}$,
where $v^*w^* \in M'_\LL$ if $(a,i) \in Q^\LL$ or 
$v^*w^*\ell_{aij} \in {\cal M}$ if $(a,i) \in Q^\DD$.

We write $G_{\ex}$ for the underlying graph of 
$\ova{G}_{\ex}$ and define other underlying graphs similarly.
We define $J^{K'}_{\ex}$ by $J^{K'}_{\ex}[V,W]=J_{\iv}
=\{ \ova{xw}: x \in Y_w\}$ and
$\ova{xy} \in J^{K'}_{\ex}[V] \Lra \ova{xy}^- \in J^K_{\ex}[V]$,
thus removing the `twist': if for some edge $xy$ of $G_{\ex}$
we add $\ova{xy}^-$ to $J^K_{\ex}$
then we add $\ova{xy}$ to $J^{K'}_{\ex}$.

\begin{lemma} \label{lem:DI} $ $
\begin{enumerate}
\item $\mb{P}^{t_{\hi}}(xy \in \GG) = d^*(\GG)/p$
independently for each $xy \in G^*$,
where $\GG \in \{G_{\ex},G^{gg'}_{ii'},G'_i\}$,
and $d^*(G_{\ex})=2p_{\ex}$, 
$d^*(G^{gg'}_i)=p^{gg'}_{ii'}$
and $d^*(G'_i)=p_{\max}$,
\item
each $\mb{P}^{t_{\hi}}(\ova{xw} \in \Psi)=d^*(\Psi)$ for 
$\Psi \in \{J^{\hi},J^{\lo}_i,J^{\no}_i,J_{\ex},J'_i\}$,
where $d^*(J^g_i)=\aA^g_i$ for $g \in \{\lo,\no\}$,
$d^*(J^{\hi})=\aA_{\hi}$, $d^*(J_{\ex})=p'_{\ex}$, $d^*(J'_i)=\pmax$, 
\item any subset ${\cal E}$ of the events in (i) and (ii)
is conditionally independent given any history
of the algorithm at time $t_0$ if it has no pairs 
that are equivalent or mutually exclusive,
\item whp each $\GG$ as in (i) is $(1.2\xi,s)$-typical 
of density $d(\GG) = d^*(\GG) \pm \DD^{-.9}$,
\item for any $w \in W$, $u \in V(F)$,
distinct $v_1,\dots,v_{s'} \in N_F(u)$ with $s' \le s$
and $x_1,\dots,x_{s'} \in V(G)$ whp 
$|N^-_{J_u}(w) \cap \bigcap_{i=1}^{s'} 
 N^+_{\ova{G}_{uv_i}}(x_i)|
= |A_u| \prod_{i=1}^{s'} (1 \pm 1.2\xi)p_{uv_i}$,
\item in Case P, for all disjoint $S_-,S_+ \subseteq V$
and $R \subseteq W$ each of size $\le s$,
for any $k,k_-,k_+ \in \{0,K'\}$,
writing $\ell_0=7/8$ and $\ell_{K'}=1/8$, we have
\[  \Big| N^-_{J^k_{\ex}}(R) \cap 
N^+_{J^{k_+}_{\ex}}(S_+) \cap 
N^-_{J^{k_-}_{\ex}}(S_-) \Big|
= (\ell_{k_-}p_{\ex})^{|S_-|}(\ell_{k_+}p_{\ex})^{|S_+|}
(\ell_k p_{\ex})^{|R|}n \pm \eta^{.9}n. \]
Also $\big| W \cap N^+_{J_{\iv}}(S_-) \cap 
N^+_{J_{\ex}^0}(S_+) \big|$ is 
$(p_{\ex}/8)^{|S_-|} (7p_{\ex}/8)^{|S_+|}n \pm \eta^{.9}n$
if $S_- \cup S_+$ is $3d$-separated, or 
is $\ge 2^{-3s} (p_{\ex}/8)^{|S_-|} 
(7p_{\ex}/8)^{|S_+|}n$ if $(S_-,S_+)$ is $3d$-separated,
\item whp each $|M'_{ai}|/M^a_i \in (1-n^{-c'},1]$,
\item whp all $d^\pm_{H^a_i}(x)$, $d^+_{J^a_i}(x)$
and $d^-_{J^a_i}(w)$
are $(1 \pm \dD)|A^a_i|$ and
$N^-_{J^a_i}(w) \cap N^-_{J^{a'}_{i'}}(w) =
H^a_i \cap H^{a'}_{i'} = \es$ whenever $ai \ne a'i'$.
\end{enumerate} 
\end{lemma}

\begin{proof}
We start by briefly justifying statements (i--iv),
which are fairly straightforward from the definition
of the algorithm. The outcome of HIGH DEGREES 
determines $G^*$ at time $t_{\hi}$,
where $G \sm G^*$ has maximum degree $<|A^*| \le 5n/\DD$.
For each $xy$ independently, 
we include it in $G_0$ with probability $p_0/p$.
Excluding $<6dn$ such $xy$ with $d(x,y) \le 3d$,
the remainder have $\mb{P}(xy \in G_1)=1-p_0/p=p_1/p$.
In DIGRAPH.iv each is then directed as
$\ova{xy}$ or $\ova{yx}$ each with probability $1/2$,
and then in (vi) independently included 
in at most one $\GG$ as in (i) with probability
$d^*(\GG)/p_1$, so with overall probability $d^*(\GG)/p_1$.
We note that $xy$ may instead be included in $H$,
again independently for all edges.
This justifies statement (i), and then (iv) is
immediate by typicality and Chernoff bounds.

For (ii), we start the calculation for each
$\mb{P}^{t_{\hi}}(\ova{xw} \in \Psi)$
by multiplying $\ov{p}_w \pm \DD^{-.9}$ 
for the event $\{x \in \ovXw \}$ and
then $p_1=1-p_0$ for $\{\ova{xw} \notin J_0 \}$.
This gives $\ov{p}_w p_1$, which equals $p_{xw}$
if $w$ is not some $w(\ova{yx})$,
and then we put $\ova{xw} \in \Psi$ with
probability $d^*(\Psi)/p_{xw}$,
giving an overall probability $d^*(\Psi)$.
On the other hand, if $w$ is some $w(\ova{yx})$
then we include $\ova{xw}$ in $J^{\hi}$ with probability
$2\aA_{\hi}/p_1\ov{p}_w$, so with 
overall probability $2\aA_{\hi}$,
or otherwise $\ova{xw}$ is available for other $\Psi$
with probability $\ov{p}_w p_1 - 2\aA_{\hi}$,
which we define to be $p_{xw}$ in this case,
giving the same overall probabilities
for $\mb{P}^{t_{\hi}}(\ova{xw} \in \Psi)$.

For (iii), we emphasise that we only have
conditional independence given the history
at time $t_0$, rather than independence,
due to the dependence between
$\{x \in \ovXw\}$ and $\{y \in \ovXw\}$ 
when $d(x,y) \le 3d$. This still suffices
to prove concentration statements in two steps:
first showing concentration of the conditional
expectation under the random choices in INTERVALS,
and then concentration under the random
choices in DIGRAPHS.
We illustrate this for (v),
omitting the similar proof 
via Lemma \ref{lem:INT} of (vi).
For any $3d$-separated $Y \sub V$,
for each $y \in Y$ independently we have
$\mb{P}(y \in \ovXw) = \ov{p}_w \pm \DD^{-.9}$,
so by Chernoff bounds whp
$|Y \cap \ovXw| = \ov{p}_w |Y| \pm 2\DD^{-.9}n$.
Then for each $y \in Y \cap \ovXw$ we have
$\mb{P}^{t_{\iv}}(y \in N^-_{J_u}(w) 
 \cap \bigcap_{i=1}^{s'} 
 N^+_{\ova{G}_{uv_i}}(x_i))
= \aA_u \ov{p}_w^{-1} \prod_{i=1}^{s'} p_{uv_i}$, 
By partitioning $\bigcap_{i=1}^{s'}  N_G(x_i)$
into $3d$-separated sets we deduce
(v) by a Chernoff bound.

For (vii), first note that for $\circ \in \{\leL,\geL\}$, 
if $m_{\circ} \neq 0$, then 
$m_{\circ} \ge \dD m$, so $p_{\circ} \ge \dD$.
We also recall that
$\sum_{(a,i) \in Q^\LL}M_i^a = m$.
As $B'_{\LL}$ is a union of $m$ 
matchings each of size $m$,
by \cite[Theorem 2]{BGS} of
Bar\'at, Gy\'arf\'as and S\'ark\"ozy
we have $|M'_{\LL}| \ge m - m^{.51}$,
so $M_i^a-|M'_{ai}| \le m^{.51} < n^{-.4}M_i^a$
for each $(a,i) \in Q^\LL$.
For $(a,i) \in Q^\DD$ 
we apply Lemma~\ref{lem:wEGJ} to ${\cal H}$
with $f(v^*w^*\ell_{a'i'j'})=1_{a'=a,i'=i}$. 
To see that this is valid, we take $D = m$,
so each edge weight is $D^{-1}$,
and note that each
$\oO(\mc{H}[v^*])$ or $\oO(\mc{H}[w^*])$ is
$m^{-1}\sum_{(a,i) \in Q^\DD}
M^a_i = 1$. Also, for any $v^*w^*$ with some label $\ell_{aij}$ 
we have $\oO(\mc{H}[v^*w^*]) 
\le m^{-1}
\bcl{\DD^{-.2}\LL/\dD} < D^{1-c^2}$, say.
Thus Lemma~\ref{lem:wEGJ} gives
$|M'_{ai}|=(1\pm n^{-c'})M_i^a$,
recalling that $c' \ll c$, as required for (v).

For (viii), we analyse the construction of $H$,
which is illustrated in Figure~\ref{fig:Hai}.
The disjointness statements
and $d^-_{J^a_i}(w)=d^+_{H^a_i}(\phi_w(a))$
are clear from the definition of the algorithm,
so it remains to establish the degree estimates.
It suffices to show all 
$d^\pm_{H^*_{ai}}(x) =  
(1 \pm .1\dD) p_1 \ov{p} m^a_i n/m$;
indeed, for each $\ova{yx} \in H^*_{ai}$ we have
$\ova{yx} \in H^a_i \Lra \ova{yx} \in H$, where
$\mb{P}(\ova{yx} \in H) 
= \aA_{\hi}/p_1(\ov{p} \pm d^{-.9})$
independently, so the estimates on
$d^\pm_{H^a_i}(x)$ hold whp by Chernoff bounds.

Consider $d^-_{H^*_{ai}}(x)$ for $x \in U_h$, $h \in [m]$.
It suffices to estimate the contribution from $V \sm V_0$,
as $|V_0|=n_0$ is negligible by comparison
with the error term in the required estimate.
Let $\circ \in \{\leL,\geL\}$ be such that $(a,i) \in Q^{\circ}$.
For each $v^*w^* \in M'_{ai}$ we include in 
$M^h_{\circ}$ all edges of $M^*_a$ with label $\ell_{aij}$ 
between $V_{v^*+h}$ and $W_{w^*+h}$. There are
$(1 \pm d^{-.8}) \ov{p} p_1 n/m$ such edges $yw$ 
with $\ova{yx} \in \ova{G}_1$ and $x \in \ovXw$
whp under the choices of $(V_{v*}: v^* \in V^*)$,
intervals and orientation of $G_1$.
For each such $yw$, in some component $P$
of $D^h_x$, we have $\mb{P}(\hi_P=\circ)=p_{\circ}$,
independently for distinct $P$,
so $\mb{E} d^-_{H^*_{ai}}(x) = p_{\circ} |M'_{ai}| 
\cdot (1 \pm 2d^{-.8}) \ov{p} p_1 n/m$.
Under the orientation of $G_1$
whp each $|P| < \log^2 n$ by Chernoff bounds.
Then $d^-_{H^*_{ai}}(x)$ is a $\log^2 n$-Lipschitz
function of independent decisions of all $\hi_P$, 
so Lemma \ref{azuma}
gives the required estimate on $d^-_{H^*_{ai}}(x)$.

Finally, consider $d^+_{H^*_{ai}}(y)$ for $y \in V(G)$.
If $y \in V_0$, then there are exactly $M^a_i$
values of $h \in [m]$ for which there is an edge
$yw \in M^h_{\circ}$ with label $\ell_{aij}$, some $j$.
If $y \notin V_0$, then there are exactly $M'_{ai}$
values of $h \in [m]$ for which there is an edge
$yw \in M^h_{\circ}$ with label $\ell_{aij}$, some $j$,
as  each $v^*w^* \in M'_{ai}$ satisfies $v^* = x_a + w^*$,
determining some $h \in [m]$ such that $y \in V_{v^*+h}$,
and some edge $yw \in M^h_{\circ} \cap M^*_a$,
where $w \in W_{w^*+h}$. By typicality and Chernoff bounds
whp each $|U_h \cap N^+_{\ova{G}_1}(y)| 
= (1 \pm 1.1\xi) p_1 n/m$.
For each $x \in U_h \cap N^+_{\ova{G}_1}(y)$
independently $\mb{P}(\hi_x=\circ)=p_{\circ}$,
writing $\hi_x=\hi_P$ where $P$ is the component
 of $D^h_x$ containing $y$.
The events $\{x \in \ovXw\}$
are independent for distinct $x$ in any $3d$-separated set,
so by partitioning $U_h$ into $3d$ such sets,
applying a Chernoff bound to each,
we obtain the required estimate on $d^+_{H^*_{ai}}(y)$,
noting that it only depends on the number 
$M^a_i$ or $|M'_{ai}|$
of $h \in [m]$ such that $M^h_{\circ}$ 
includes $yw$ with some label $\ell_{aij}$,
and not the set of such $h$, which is yet to be 
determined when choosing the matchings $M_a$.
\end{proof}

\section{Approximate decomposition} \label{sec:approx}

In this section we analyse the subroutine
APPROXIMATE DECOMPOSITION, which applies
hypergraph matchings to embed
most of $F$ in Cases S and P.

\subsection{Hypergraph matchings}

The main goal of this section is the following lemma,
which will allow us to apply Lemma \ref{lem:wEGJ}
to the hypergraph matchings
chosen in APPROXIMATE DECOMPOSITION,
i.e.\ all auxiliary vertices 
have $\oO'$-weighted degree 
close to and not exceeding $1$,
and all $\oO'$-weighted codegrees are small;
statement (ii) concerns the degree that a pair $ux$ 
would have if it were introduced as an auxiliary vertex
(but we do not do this to avoid additional
complications in analysing the relationship
between $\oO'$ and $\oO$). 
The `bad' graphs and sets appearing in the lemma
will be defined and analysed in Lemma \ref{lem:B}.
  
\begin{lemma} \label{lem:Hwdeg}
whp for each $i \in [i^*]$,
\begin{enumerate}
\item
$\oO'({\cal H}_i[\bm{v}]) \in (1-2\eps^{.8},1]$
for all $\bm{v} \in V({\cal H}_i)$,
\item
$\sum_{w \in W} \oO'(\edge{w}{u}{x})
= 1 \pm 2\eps^{.8}$ for all $x \in V(G)$,
$u \in A_i$,
\item
$\oO'(\bm{e}) > (1+2\eps_i)^{-1} \oO(\bm{e})$
for all $\bm{e}=\edge{w}{u}{x}$ with
$u \in A_i \sm ( A^{\bad}_w \cap N_{F'}(A^{\lo}) )$,
$\ova{xw} \in J_i \sm J^{\bad}$,
\item
$\oO'({\cal H}_i[uw]) \in (1-2\eps_i,1-.5\eps_i]$
for all $u \in A_i \sm ( A^{\bad}_w \cap N_{F'}(A^{\lo}) )$,
\item
$\oO'({\cal H}_i[\bm{vv'}]) < \DD^{-.9}$ 
for all $\bm{v},\bm{v'} \sub V({\cal H}_i)$.
\end{enumerate}
\end{lemma}

To satisfy the hypotheses of Lemma \ref{lem:wEGJ}
for $({\cal H}_i,\oO')$ we let $C=n$ and $\bB = c/2$.
Then the edge weights satisfy
$\oO'(\edge{w}{u}{x}) \ge |A_u|^{-1} \ge C^{-1}$,
the codegree condition holds by
Lemma \ref{lem:Hwdeg}.v,
and the vertex weights satisfy
$\oO'({\cal H}_i[\bm{v}])
= \sum \{ \oO'(\bm{e}): \bm{v} \in \bm{e} \}
\le 1-.5\eps_i$ by definition of $\oO'$.

Next we define and bound the `bad' graphs and sets
appearing in Lemma \ref{lem:Hwdeg}.
For each $i$ write $A^{\lo}_i = A^{**}_i \cup A'_i$, where 
for $u \in A^{\lo}_i$ with $N_<(u) \cap A_0 = \{u'\}$
we include $u$ in $A^{**}_i$ if $u' \in A^* \cup A^{**}$
or in $A'_i$ if $u' \in A'_0$.
Let $S^w_i$ be the multiset on $A'_0$ 
where for each $u \in A'_i$, $N_<(u) \cap A_0 = \{u'\}$
we include $\phi_w(u')$, with multiplicity, 
so that $|S^w_i|=|A'_i|$.
Note that all multiplicities in $S^w_i$ 
are $\le D$ by definition of $A'_0$.
Let
\begin{gather*} J^{\bad}_i = \{ \ova{xw} \in J^{\lo}_i :
|S^w_i \cap N^-_{\ova{G}_1}(x) | 
\neq p_1 |S^w_i| \pm \xi' n \} 
\quad\text{and}\quad
\textstyle J^{\bad} = \bigcup_i J^{\bad}_i,\\
\textstyle
B = \{x \in V(G): d_{J^{\bad}}(x) > \dD^3 n\} \quad
\text{and}\quad
A^{\bad}_w := (A^{\lo} \cup A^{\no}) \cap
\bigcup_{x \in B} N_>(\phi_w^{-1}(x)).
\end{gather*}

\begin{lemma}\label{lem:B}
whp $|B|<\dD^4 n$ and
every $d_{J^{\bad}}(w),|A^{\bad}_w| < \dD^3 n$.
\end{lemma}

\begin{proof}
We start by bounding $d_{J^{\bad}}(w)$ for each $w \in W$.
We may assume $|A'_0|>\xi'n/D$,
otherwise each $|S^w_i| \le \xi'n$,
and then $J^{\bad}=\es$.
As $G$ is $(\xi,s)$-typical, 
a well-known non-partite 
variant of Lemma~\ref{lem:DLR}
implies that $G$ is $\xi^{.1}$-regular.
Writing $X=\phi_w(A'_0)$,
as $|X|,|\ov{X}_w\sm X| \ge \xi'n/D$,
standard regularity properties imply that
$G[X,\ov{X}_w\sm X]$ is $\xi^{.01}$-regular
of density $p \pm \xi^{.01}$.
Then Chernoff bounds imply that whp
$\wt{G} = \{ uv: u \in X, v \in \ov{X}_w\sm X, 
\ova{uv} \in \ova{G}_1 \}$ is $\xi^{.001}$-regular
of density $p_1 \pm \xi^{.001}$.
By Lemma~\ref{lem:SRLtoH} applied with $G=\wt{G}$
and $H = \{\{ \phi_w(u') \}: u' \in A'_0 \}$
we deduce $d_{J^{\bad}_i}(w) < \xi'n$,
so $d_{J^{\bad}}(w) < \dD^7 n$, say.
As $\dD^3 n |B| \le \sum_{x \in B} d_{J^{\bad}}(x)
\le \sum_{w \in W} d_{J^{\bad}}(w)
< |W| \dD^7 n$ we have $|B| < \dD^4 n$.
Since $d_{>}(u) \le \pmax^{-1}$ for every $u \in A_{\ge 1}$, 
we conclude that 
$|A^{\bad}_w| < \pmax^{-1}|\phi_w^{-1}(B)| < \dD^3 n$.
\end{proof}

Henceforth, we assume Lemmas \ref{lem:Hwdeg} 
and \ref{lem:B} for all $i'<i$,
our aim being to show that they hold for $i$.
First we establish various properties
of the matchings ${\cal M}_{i'}$ for $i'<i$
that will be used in the proof.
We let $A_{i,w} =  A_{t_i^+,w} \sm A_{t_i,w}$,
which is the set of $u \in A_i$ such that
$\phi_w(u)$ is defined by the matching ${\cal M}_i$,
and let $A^0_{i,w} = A_i \sm A_{i,w}$.

\begin{lemma} \label{lem:Tleave}
For all $0<i'\le i$, $w \in W$, $W' \sub W$, $X \sub V(G)$, 
$U \sub A^{\lo}_{i'} \cup A^{\no}_{i'}$ whp
\begin{enumerate}
\item
$|A^0_{i',w}| < 2.1\eps_{i'} |A_{i'}|$,
\item
$|\{w \in W: \phi_w(u) \in X\}| < 1.1|X| + \DD$
for all $u \in A^{\lo}_{i'} \cup A^{\no}_{i'}$,
\item
$|\{w \in W: \phi^{-1}_w(x) \in U\}| < 1.1|U| + \DD$
for all $x \in V(G)$,
\item
$|\{ w \in W: u \in A^{\bad}_w\}| < 5\dD^4 n$ 
for all $u \in A^{\lo}_{i'} \cup A^{\no}_{i'}$,
\item 
$.4\eps_{i'}|W'|-\dD|A_u| < |\{w \in W': u \in A^0_{i',w}\}| 
< 2.1\eps_{i'}|W'| + \dD|A_u|$
for all $u \in A_{i'}$,
\item $\sum_{u \in A^g_i} |F'[N_>(u),A^0_{i',w}]|$
and $|F'[A^g_i,A^0_{i',w}]|$ 
are $< 9\pmax^{-1} \eps_{i'} |A^g_i|$
for all $g \in \{\hi,\lo,\no\}$.
\end{enumerate}
\end{lemma}

\begin{proof}
We write $|A_{i',w}| = f({\cal M}_{i'})$, 
where $f$ is the function on ${\cal H}_{i'}$ 
defined by $f(\edge{w'}{u'}{x'})=1_{w'=w}$.
We have $f({\cal H}_{i'},\oO') 
= \sum_{u' \in A_{i'}} \oO'({\cal H}_{i'}[u'w])
\ge (1-2\eps_{i'}) |A_{i'}|-\dD^3 n$
by Lemmas \ref{lem:B} and \ref{lem:Hwdeg}.iv. 
For any $\bm{e} \in {\cal H}_{i'}$
we have $f_{\{\bm{e}\}}({\cal H}_{i'},\oO')
\le \oO'(\bm{e})
\le 1 < C^{-\bB} f({\cal H}_{i'},\oO')$.
By Lemma \ref{lem:wEGJ} whp $f({\cal M}_{i'})
=  (1 \pm C^{-\bB}) f({\cal H}_{i'},\oO')
\ge 1-2.1\eps_{i'}$,
so $|A^0_{i',w'}| < 2.1\eps_{i'} |A_{i'}|$.

Statements (ii) and (iii) are similar,
using Lemma \ref{lem:Hwdeg}.ii.
For (iv), we have 
$|\{w \in W: u \in A^{\bad}_w\}|
\leq \sum_{v \in N_<(u)}|\{w: \phi_w(v) \in B\}|
\le 4.4|B| + 4\DD < 5\dD^4 n$
by (ii) and Lemma~\ref{lem:B}.

For (v) we write
$|\{w \in W': u \in A^0_{i',w}\}| 
= |W'|+\DD^{.9} - f({\cal M}_{i'})$
redefining $f$ by setting
$f(\es)=\DD^{.9}$ and each
$f(\edge{w}{u'}{x}) = 1_{w \in W', u'=u}$.
Then $0
\le \DD^{.9} + \sum_{w \in W'}\oO'(\mc{H}_{i'}[uw])
-f(\mc{H}_{i'},\oO') <
|\{w \in W: u \in A^{\bad}_w\}|$.
By Lemma~\ref{lem:Hwdeg}.iv we see that
if $u \in A^{\hi}_{i'}$, then
$.5\eps_{i'}|W'| < |W'|+\DD^{.9}-f({\cal H}_{i'},\oO') 
< 2\eps_{i'}|W'|$ and
if $u \in A^{\lo}_{i'} \cup A^{\no}_{i'}$ then
(iv) implies
$.5\eps_{i'}|W'|-5\dD^4 n < |W'|+\DD^{.9}-f({\cal H}_{i'},\oO') 
< 2\eps_{i'}|W'| + 5\dD^4 n$,
and all $f_{\{\bm{e}\}}({\cal H}_{i'},\oO') \le 1$,
so by Lemma~\ref{lem:wEGJ} whp (ii) holds.

For (vi), we write
$1 + |A^g_i|^{-1}\sum_{u \in A^g_i} |F'[A_{i',w},N_>(u)]|
= f({\cal M}_{i'})$, where $f(\es)=1$
and each $f(\edge{w}{u'}{x'}) 
= 1_{w=w'} |A^g_i|^{-1} \sum_{u \in A^g_i}
 |N_{F'}(u') \cap N_>(u)|$.
Using Lemma~\ref{lem:Hwdeg}.iv we have
\[ f({\cal H}_{i'},\oO')-1 >
 (1-2\eps_{i'})|A^g_i|^{-1} \sum_{u \in A^g_i} 
 \sum_{u' \in A_{i'}}  |N_{F'}(u') \cap N_>(u)|
 = (1-2\eps_{i'})|A^g_i|^{-1} 
 \sum_{u \in A^g_i} |F'[A_{i'},N_>(u)]|\]
and all $f_{\{\bm{e}\}}({\cal H}_{i'},\oO')
\le |A^g_i|^{-1} \max_{u'} 
 \sum_{v \in N_>(u')} |N_<(v)|
< 4\pmax^{-1} |A^g_i|^{-1} < C^{-\bB}f(\mc{H}_{i'},\oO')$.
By Lemma \ref{lem:wEGJ} whp 
$|A^g_i|^{-1}\sum_{u \in A^g_i}
 |F'[A^0_{i',w},N_>(u)]| 
< 2.1 \eps_{i'}  |A^g_i|^{-1} 
 \sum_{u \in A^g_i} |F'[A_{i'},N_>(u)]|
\le 9\pmax^{-1} \eps_{i'}$.
The second bound is similar so we omit the proof.
\end{proof}

We conclude this subsection by 
deducing Lemma \ref{lem:Hwdeg} (for $i$)
from the following estimates 
on $\oO$-weighted degrees, which thus
become the main goal of this section
and which we will prove assuming
Lemmas~\ref{lem:B}--\ref{lem:bad} for all $i'<i$.
 
\begin{lemma} \label{lem:bad}
whp for each $i \in [i^*]$,
\begin{enumerate}
\item
$\oO({\cal H}_i[uw]) = 1 \pm \eps_i$ 
for all $uw \in A_i \times W$,
\item 
$\oO({\cal H}_i[\ova{xw}])$ is $1 \pm .5\eps^{.8}$
for all $xw$ in $J_i$ and is $1 \pm \eps_i$ 
if $\ova{xw} \notin J^{\bad}$,
\item 
$\oO({\cal H}_i[\lova{xy}])$ is $1 \pm \eps^{.8}$
for all $\lova{xy} \in \ova{G}_{uv}$,
and is $1 \pm \eps_i$ if
$v \notin A^{\lo}$ or $y \notin B$,
\item
$\sum_{w \in W} \oO(\edge{w}{u}{x})
= 1 \pm \eps_i$ for all $x \in V(G)$,
$u \in A_i$.
\end{enumerate}
\end{lemma}

\begin{proof}[Proof of Lemma \ref{lem:Hwdeg} for $i$]
We have already noted that all $\oO'$-weighted degrees
are $\le 1-.5\eps_i$, so it remains to prove the lower bounds.
Statements (i) and (ii) are immediate by applying
the definition of $\oO'$, using Lemma~\ref{lem:bad},
which gives $\oO'(\bm{e}) 
> (1+\eps^{.8})^{-1}(1-.5\eps_i) \oO(\bm{e})$
for any $\bm{e}  \in {\cal H}_i$.
 
For (iii), if $\bm{e}=\edge{w}{u}{x}$ with
$u \in A_i \sm ( A^{\bad}_w \cap N_{F'}(A^{\lo}) )$, then 
$\oO({\cal H}_i[\lova{xy}]) = 1 \pm \eps_i$ 
for all $y=\phi_w(v)$, $v \in N_<(u)$ and
if $\ova{xw} \in J_i \sm J^{\bad}$, then
$\oO({\cal H}_i[\ova{xw}]) = 1 \pm \eps_i$ 
by Lemma \ref{lem:bad},
so $\oO'(\bm{e})>(1+\eps_i)^{-1}(1-.5\eps_i)\oO(\bm{e})
>(1+1.6\eps_i)^{-1}\oO(\bm{e})$ 
by definition.

For (iv), note that 
for any $\bm{e} \in {\cal H}_i$ containing $uw$ 
with $u$ in $A^{\hi}_i$ or $A^{\no}_i \sm A^{\bad}_w$
we have $\oO'(\bm{e}) > (1+1.6\eps_i)^{-1} \oO(\bm{e})$,
so $\oO'({\cal H}_i[uw]) > (1+1.6\eps_i)^{-1}
 \oO({\cal H}_i[uw]) > 1-2\eps_i$.
If instead $u \in A^{\lo} \sm A^{\bad}_w$ 
then, by Lemma~\ref{lem:Hwdeg}.iii,
$\oO'({\cal H}_i[uw])$ 
is at least the sum of $(1+1.6\eps_i)^{-1}\oO(\bm{e})$ 
where $\bm{e}=\edge{w}{u}{x}$ 
over $x$ with $\ova{xw} \notin J^{\bad}$,
so $\oO'({\cal H}_i[uw]) > 
 (1+1.6\eps_i)^{-1}\oO({\cal H}_i[uw]) 
- \dD^3 n \cdot (1\pm .5\eps^{.8})^{-1}
\pmax^{-4}|A_u|^{-1} > 1-2\eps_i$.

For (v), we consider codegrees according to the
various types of vertices. First we note that 
each $\oO'(\edge{w}{u}{x}) \le 2\pmax^{-4} |A_u|^{-1}$.
Each $|A_u| \ge \DD$, so
this easily gives the codegree bound
for the pairs appearing in the following bounds:
$|{\cal H}_i[uw,vw]| = 0$,
$|{\cal H}_i[\ova{xw},\ova{yw}]| = 0$,
$|{\cal H}_i[uw,\ova{xw}]| \le 1$,
$|{\cal H}_i[uw,\lova{xy}]| \le 2$.
If $\edge{w}{u}{x}$ contains $\ova{xw}$ and $\lova{xy}$
then $u \in N_>(v)$ where $\phi_w(v)=y$, 
so $|{\cal H}_i[\ova{xw},\lova{xy}]|$
is at most $\DD$,
or at most $\pmax^{-1}$ if $v \in A_{\ge 1}$,
or $0$ if $v \in A_0$ and $\ova{xw} \in J^{\hi}$.
These weighted codegrees are therefore at most
$2\DD \pmax^{-4}  (\dD n)^{-1}$ or
$2\pmax^{-4} \DD^{-1} < \DD^{-.9}$.

It remains to bound
$\oO'({\cal H}_i[\lova{xy},\lova{xy}'])$.
Suppose $\lova{xy} \in \ova{G}_{u_0v_0}$ 
and $\lova{xy}' \in \ova{G}_{u_0v_0'}$,
with $u_0v_0 \in F'[A_i,A_j]$
and $u_0v_0' \in F'[A_i,A_{j'}]$,
say with $j' \le j$.
We have $j>0$ as $|N_<(u) \cap A_0| \le 1$ 
for all $u \in A_{\ge 1}$, and
as $F[A^{\hi}]=\es$ we have
$|A_{u_0}||A_{v_0}| \ge \dD n \DD$.

Suppose first that $j'<j$.
Let $f$ be the function 
on ${\cal H}_j \cup \{\es\}$,
where $f(\es) = \DD^{-.99}$
and each $f(\edge{w}{v}{z})$ 
is $2\pmax^{-4} |A_{u_0}|^{-1}$ 
if $z=y$, $v \in A_{v_0}$
and there are $u \in A_{u_0} \cap N_>(v)$
and $v' \in N_<(u) \cap A_{v'_0}$ 
with $\phi_w(v')=y'$, or $0$ otherwise.
If $j'>0$, then 
for each $w \in W$ there is at most one $v'$
with $\phi_w(v')=y'$, at most $\pmax^{-1}$
choices of $u \in A_i \cap N_>(v')$,
and at most $4$ choices of $v \in N_<(u)$,
so $f({\cal H}_j,\oO') \le \DD^{-.99} +
8n\pmax^{-9} |A_{v_0}|^{-1}|A_{u_0}|^{-1}
< 2\DD^{-.99}$.
If $j'=0$, then there are at most $\DD$
choices of $u \in A_i \cap N_>(v')$,
and $u \in A^{\lo}$, so every $v \in N_{<}(u)$
in $A_{v_0}$ lies in $A^{\no}$,
so $f({\cal H}_j,\oO') \le \DD^{-.99} +
8n\DD \pmax^{-8} (\delta n)^{-2}
< 2\DD^{-.99}$.
As each  $f_{\{\bm{e}\}}({\cal H}_j,\oO')
\le 2\pmax^{-4} \DD^{-1} < C^{-\bB} 
f({\cal H}_j,\oO')$, by Lemma \ref{lem:wEGJ} whp 
$\oO'({\cal H}_i[\lova{xy},\lova{xy}'])
< f({\cal M}_j)
=  (1 \pm C^{-\bB}) f({\cal H}_j,\oO')
< \DD^{-.9}$.

Now suppose $j'=j$. Then $j \ne 0$ and
we cannot have $v_0$, $v'_0$ both in $A^{\hi}$ 
by definition of $A_0$ as a span,
so $|A_{v_0}||A_{v_0'}| \ge \dD n \DD$.
Let $f$ be the function 
on $\tbinom{{\cal H}_j}{2} \cup \{\es\}$,
where $f(\es) = \DD^{-.99}$
and each $f(\edge{w}{v}{z},\edge{w}{v'}{z'})$ 
is $2\pmax^{-4} |A_{u_0}|^{-1}$ if $z=y$, $z'=y'$, 
$v \in A_{v_0}$, $v \in A_{v'_0}$ and there is 
$u \in A_{u_0} \cap N_>(v) \cap N_>(v')$, 
or $0$ otherwise.
For each $u \in A_{u_0}$ there are
$\le 12$ choices of $(v,v')$
so $f({\cal H}_j,\oO') \le \DD^{-.99} +
24n\pmax^{-12} |A_{v_0}|^{-1}|A_{v'_0}|^{-1}
< 2\DD^{-.99}$.
As above, by Lemma~\ref{lem:wEGJ}
whp $f({\cal M}_j) < \DD^{-.9}$.
\end{proof}

\subsection{Potential embeddings}

We define a hypergraph ${\cal H}$ with vertex parts 
$G$, $A_{\ge 1} \times W$ and $V(G) \times W$,
which contains all potential edges of all ${\cal H}_i$,
in the following sense. Given $w \in W$, $u \in V(T)$
and an injection $f:N_\le(u) \to V(G)$ such that
$f(u')f(u) \in G$ for all $u' \in N_<(u)$
we let $P_w(f)$ be the `potential edge'
containing $u \in V(T)$, $f(u) \in V(G)$
and $f(u')f(u) \in G$ for all $u' \in N_<(u)$.
For any $u \in S \sub N_\le(u)$ 
and injection $f':S \to V(G)$ we let 
$P_w(f')$ be the set of all $P_w(f) \in {\cal H}$
such that $f$ restricts to $f'$ on $S$.
We use the notation $P_w(u \to x)$
when $S=\{u\}$ with $f(u)=x$
and $P_w(uv \to \lova{xy})$ when 
$S=\{u,v\}$ with $f(u)=x$, $f(v)=y$.

For each time $t$ we introduce
a measure $\oO_t$ on ${\cal H}$
where each $\oO_t(P_w(f))$ estimates 
the probability given the history at time $t$ that 
the $w$-embedding will be consistent with $f$.
We define $\oO_t$ by the following formula 
involving other estimated probabilities that
will be discussed below: \[ \oO_t(P_w(f)) 
= \oO^*_t(P_w:u \to x) \prod_{v \in N_<(u)} 
p_t(uv)^{-1} \oO^*_t( P_w:v \to f(v) ).\]
The key parameter in this formula is  
$\oO^*_t( P_w:u \to x )$, which will estimate
the probability $\mb{P}^t(\phi_w(u)=x)$
given the history at time $t$ 
that we will have $\edge{w}{u}{x}$.
We also associate an edge probability 
$p_t(uv)$ to each $v \in N_<(u)$,
where $p_t(uv)=1$ if $t \ge t_v$,
otherwise $p_t(uv)$ is $p$ if $t<t_0$,
is $p_1$ if $t_0 \le t < t_1$,
or is $p_{uv}$ for $t \ge t_1$.
The intuition for the formula is that
conditional on $\edge{w}{u}{x}$, 
the events $\phi_w(v)=f(v)$ become
about $p_t(uv)^{-1}$ times more likely
and are roughly independent. 
In our calculations it will be sufficient
to work only with $\oO^*_t( P_w: u \to x )$,
so the formula for the measure $\oO_t$ above 
can be thought of as just an intuitive explanation
for why the calculations work (it is not logically
necessary for the proof).

Note that we have introduced similar notation
for two different quantities, 
namely $\oO^*_t( P_w: u \to x )$
and $\oO_t( P_w(u \to x) ) = \sum \{ 
  \oO_t(P_w(f)): P_w(f) \in  P_w(u \to x) \}$;
they will be approximately equal.
In general, for any $u \in S \sub N_\le(u)$ 
and injection $f':S \to V(G)$ we will have
\[ \oO_t(P_w(f')) \approx  \oO^*_t(P_w:f') :=
  \oO^*_t(P_w:u \to x) \prod_{u' \in S \sm \{u\}} 
  p_t(uu')^{-1} \oO^*_t( P_w:u' \to f(u') ).\]
Another important example of this will be
\[ \oO_t(P_w(uv \to \lova{xy})) 
\approx  \oO^*_t(P_w: uv \to \lova{xy}) 
= p_t(uv)^{-1} \oO^*_t(P_w:u \to x) \oO^*_t(P_w:u \to y) .\]
Initially, we let all 
$\oO^*_0(P_w:u \to x):=n^{-1}$.
(One can check that typicality of $G$ 
gives $\oO_0(P_w(u \to x)) = (1\pm\xi^{.9})
\oO^*_0(P_w:u \to x)$.)
For $t \ge t_{uw}$, i.e.\ times after $\phi_w(u)$
has been defined, we let
$\oO^*_t(P_w:u \to x):=1_{\edge{w}{u}{x}}$.
We thus have $\oO_t(P_w(f'))=\oO^*_t(P_w:f')$
at times $t$ after $\phi_w(u)$ has been defined
for all $u \in N_\le(u) \sm S$. 
In particular, if $t \ge t_{u'}$ 
for all $u' \in N_<(u)$ 
then there is at most one $f$ with $f(u)=x$
consistent with the history, and we have
$\oO_t(P_w(f)) = \oO^*_t(P_w:u \to x)$.
Furthermore, for $u \in A_i$, when we come 
to step $i$ of APPROXIMATE DECOMPOSITION
we will have $\oO^*_{t_i}(P_w:u \to x) 
= \oO(\edge{w}{u}{x})$. 

Now we define $\oO^*_t(P_w:u \to x)$ 
for general $t$
and $u \in V(F)$.
As mentioned above,
we let $\oO^*_0(P_w:u \to x)=n^{-1}$
and $\oO^*_t(P_w:u \to x)=1_{\edge{w}{u}{x}}$
for $t \ge t_{uw}$. At each time $t<t_{uw}$
where the possibility of $\edge{w}{u}{x}$ 
depends on an event in the algorithm,
if the event fails we let
$\oO^*_t(P_w:u \to x)=0$,
and if it succeeds we will divide by an estimate 
for its probability, thus approximately preserving 
the conditional expectation of the surviving weight.
We let $P^t_w(u \to \cdot)$ be the set
of $x$ such that $\oO_t^*(P_w:u \to x) \ne 0$
and define $P^t_w(\cdot \to x)$, $P^t_{\cdot}(u \to x)$
in analogy, and also define
$P^t_w(\cdot \to \lova{xy})$ to be the set of 
$uv \in F'$ such that $\oO_t^*(P_w:u \to x) \ne 0$ and 
$\oO_t^*(P_w:v \to y) \ne 0$.

During HIGH DEGREES, when we embed
any $u' \in N_<(u) \cap A^*$ we will have
$\mb{P}^{t_{u'}^-}(\phi_w(u')x \in G) \approx p$,
so if this occurs we let
$\oO^*_{t_{u'}}(P_w:u \to x) 
:= p^{-1} \oO^*_{t_{u'}^-}(P_w:u \to x)$.
So at the end of HIGH DEGREES,
$\oO^*_{t_{\hi}}(P_w:u \to x)$
is $p^{-|N_<(u) \cap A^*|} n^{-1}$
for $x \in P^{t_{\hi}}_w(u \to \cdot)$
or $0$ otherwise.
(Note that our estimate ignores the possibility
that $\edge{w}{u}{x}$ may be impossible due to
requiring an edge of $G \sm G^*$.)

In INTERVALS, we require $x \in \ov{X}_w$, 
which by Lemma~\ref{lem:INT}.i
occurs with probability $\approx \ov{p}_w$, 
and then we let $\oO^*_{t_{\iv}}(P_w:u \to x)$
be $\ov{p}_w^{-1} \oO^*_{t_{\hi}}(P_w:u \to x)$
for $x \in P^{t_{\iv}}_w(u \to \cdot)$
or $0$ otherwise. 

In EMBED $A_0$, after choosing $G_0$ and $J_0$,
there are two cases. 
If $u \in A_{\ge 1}$
we let $\oO^*_{t_{G_0}}(P_w: u\to x)=
p_1^{-|N_<(u) \cap A_0\sm A^*|}
\oO^*_{t_{\hi}}(P_w:u \to x)$
for $x \in P_w^{t_{G_0}}(u \to \cdot)$
and $0$ otherwise.
Indeed (recalling $u \in A_{\ge 1}$), 
for any $x \in P_w^{t_{\iv}}(u \to \cdot)$
we have $x \in P_w^{t_{G_0}}(u\to \cdot)$
iff for $u' \in N_<(u) \cap A_0\sm A^*$
we have $x\phi_w(u') \in G_1$,
which occurs with probability $p_1$.
If $u \in A_0 \sm A^*$
we require $\ova{xw} \in J_0$,
which for available $x$ 
has probability $p_0/\ov{p}_w$,
and then we let $\oO^*_{t_{G_0}}(P_w:u \to x)$
be $p_0^{-1} \oO^*_{t_{\hi}}(P_w:u \to x)$
for $x \in P^{t_{G_0}}_w(u \to \cdot)$
or $0$ otherwise. 

During the embedding of $A_0$, 
i.e. EMBED $A_0$ (ii),
after choosing some $\phi_w(a)$
at time $t=t_a$,
for $u \in A_0 \sm A^*$ we let
$\oO^*_t(P_w:u \to x)$ be 
$\oO^*_{t^-}(P_w:u \to x)$ 
if $au \notin T$, or if $au \in T$
we let $\oO^*_t(P_w:u \to x)$ be 
$p_0^{-1} \oO^*_{t^-}(P_w:u \to x)$
for $x \in P^t_w(u \to \cdot)$
or $0$ otherwise.
So $\oO^*_{t_0}(P_w:u \to x)$
is $p_0^{-|N_T(v) \cap A_0\sm A^*|} 
\oO^*_{t_{\hi}}(P_w:u \to x)$
for $x \in P^{t_0}_w(u \to \cdot)$
or $0$ otherwise. 
For $u \in A_{\ge 1}$,
if $au \in F'$
we let $\oO^*_t(P_w:u \to x)$ be 
$p_1^{-1} \oO^*_{t^-}(P_w:u \to x)$
for $x \in P^t_w(u \to \cdot)$
or $0$ otherwise.

For DIGRAPH, we let 
$\oO^*_{t_1}(P_w:u \to x)$ be
$|A_u|^{-1} \prod_{v \in N_<(u) \cap A_0} p_{uv}^{-1}$
for $x \in P^{t_1}_w(u \to \cdot)$
or $0$ otherwise. To justify this,
we first consider $u \in A^a_i$, $a \in A^\DD_i$,
when $N_<(u) \cap A_0 = \es$, and all
$d^\pm_{H^a_i}(y) = (1 \pm \dD)|A_u|$
by Lemma \ref{lem:DI}.viii. 
If $u \in A^{\no}_i$ we must choose
$\ova{xw} \in J^{\no}_i$
with probability $p_{xw}/\ov{p}_w \cdot \aA^{\no}_i/p_{xw}
=\aA^{\no}_i/\ov{p}_w$.
If instead 
$u \in A^{\lo}_i$, 
we must choose $\lova{x\phi_w(v)}$
as an arc of $\ova{G}_{uv}$ for all edges 
$x\phi_w(v)$ of $G_1$ with $v \in N_<(u) \cap A_0$,
each with probability 
$\tfrac12 \cdot 2p_{uv}/p_1 = p_{uv}/p_1$ independently,
and $\ova{wx} \in J_u$ with probability
$\aA^{\lo}_i/\ov{p}_w$.

During APPROXIMATE DECOMPOSITION, we let
$\oO^*_t(P_w:u \to x)$ be $|A_u|^{-1}
 \prod_{v \in N_<(u) \cap A_{t,w}} p_{uv}^{-1}$
for $x \in P^t_w(u \to \cdot)$, or $0$ otherwise;
we will see that whenever we embed $v \in N_<(u)$,
we have $\mb{P}(\lova{x\phi_w(v)} \in \ova{G}_u) 
\approx p_{uv}$.
We note that $\oO^*_{t_i}(P_w:u \to x) 
= \oO(\edge{w}{u}{x})$, as mentioned above.
We emphasise that the $\oO^*_t(P_w: u \to x)$
are definitions (with justifications provided 
only for intuition), and it is the sets $P_w^t(u \to x)$
which change during the algorithm.

We use the notation $\oO_t^*(P_w:u \to x)$
with a set of vertices in place of one or more of $w,u,x$
to denote a sum of $\oO_t^*(P_w:u \to x)$
over these sets, for example
$\oO_t^*(P_W: u \to x)$ or
$\oO_t^*(P_w: V(T) \to x)$ or
$\oO_t^*(P_W: F'[A_i,A_j] \to \lova{xy})$.
To see the connection to weighted degrees
in Lemma~\ref{lem:bad},
observe that $\edge{w}{u}{x} \in \mc{H}_i$
iff $x \in P_w^{t_i}(u \to \cdot)$, so
$$
\oO(\mc{H}_i[uw])=
\sum\{\oO(\edge{w}{u}{x}):x \in P_w^{t_i}(u \to \cdot)\}
=\oO_{t_i}^*(P_w: u \to N^-_{J_u}(w))
=\oO_{t_i}^*(P_w: u \to V(G))
$$
and similarly
$\oO(\mc{H}_i[\ova{xw}])=\oO_{t_i}^*(P_w: A_{\ova{xw}} \to x)$ and
$\oO(\mc{H}_i[\lova{xy}])=\oO_{t_i}^*(P_W: F'[A^g_i,A^{g'}_j]
\to \lova{xy})$ when $\lova{xy}\in G^{gg'}_{ij}$.

We conclude this subsection with the following lemma,
which implies the estimate on the weighted degrees
$\oO({\cal H}_i[uw])$ in Lemma \ref{lem:bad}.i.
The proof is immediate from Lemma \ref{lem:DI}.v,
as $x \in P^t_w(u \to \cdot)$ for $t \ge t_1$
iff $x \in N_{J_u}(w) \cap
 \bigcap_{v \in N_<(u) \cap A_{t,w}} 
   N^+_{\ova{G}_{uv}}(\phi_w(v))$.

\begin{lemma} \label{lem:Px}
For all $u \in A_{\ge 1}$, $w \in W$ 
and $t \ge t_1$ we have 
\[ |P^t_w(u \to \cdot)| 
= (1 \pm 5\xi)\oO^*_t( P_w:u \to x )^{-1}
= (1 \pm 5\xi) |A_u|
 \prod_{v \in N_<(u) \cap A_{t,w}} p_{uv}.\]
\end{lemma}

\subsection{$J$ degrees}

In this subsection we prove the estimates
in Lemma \ref{lem:bad}.v concerning weighted
degrees $\oO({\cal H}_i[\ova{xw}]) =
\oO^*_{t_i}( P_w: A_{\ova{xw}} \to x )$ 
for $\ova{xw}$ in $J_i$.
We start with the following estimate at time $t_1$. 
 
\begin{lemma} \label{lem:x:t1}
For all $\ova{xw} \in J^{\hi}_i \cup J^{\no}_i$
we have $\oO^*_{t_1}( P_w:A_{\ova{xw}} \to x )=1$.
For all $\ova{xw} \in J^{\lo}_i$ whp
$\oO^*_{t_1}( P_w:A_{\ova{xw}} \to x ) 
= 1 \pm .3\eps^{.8}$.
\end{lemma}

\begin{proof}
If $\ova{xw} \in J^{\hi}_i \cup J^{\no}_i$
then $P^{t_1}_w(\cdot \to x) =  A^g_i$,
where $g=\no$ if $\ova{xw} \in J^{\no}_i$
or $g=a$ if  $\ova{xw} \in J^a_i$
for some $a \in A^\DD_i$.
We have $\oO^*_{t_1}( P_w:u \to x ) = |A^g_i|^{-1}$
for all $u \in A^g_i$, so
$\oO^*_{t_1}( P_w:A^g_i \to x )=1$.

Now we consider the evolution
of $\oO^*_t( P_w:A^{\lo}_i\to x )$.
Initially $\oO^*_0( P_w:A^{\lo}_i\to x )
= |A^{\lo}_i| n^{-1} = \aA^{\lo}_i$.
During HIGH DEGREES,
for each $u \in A^{\lo}_i$, 
when we embed some $a \in A^*$,
if $a \notin N_<(u)$ we have
$\oO^*_{t_a}( P_w:u \to x )
 = \oO^*_{t_a^-}( P_w:u \to x )$,
whereas if $a \in N_<(u)$, 
as $\mb{P}^{t_a^-}(\phi_w(a) \in N_G(x) )  
= (1 \pm \dD)p$ by Lemma \ref{lem:hi} we have
\[\mb{E}^{t_a^-} \oO^*_{t_a}( P_w:u \to x )
 = \mb{P}^{t_a^-}(\phi_w(a) \in N_G(x) ) 
 p^{-1} \oO^*_{t_a^-}( P_w:u \to x )
 = (1 \pm \dD) \oO^*_{t_a^-}( P_w:u \to x ).\]
As each $|N_<(u)| \le 4$ we have
 $\mb{E}^0 \oO^*_{t_{\hi}}( P_w:A^{\lo}_i \to x )
 = (1 \pm \dD)^4 \aA^{\lo}_i$.
For concentration, we bound
each $|\oO^*_{t_a}( P_w:A^{\lo}_i \to x )
- \oO^*_{t_a^-}( P_w:A^{\lo}_i \to x )|$
by $\oO^*_{t_a^-}( P_w: A^{\lo}_i\cap N_>(a) \to x )
< \DD \pmax^{-4} (\delta n)^{-1}$,
so by Lemma \ref{freedman} whp 
$\oO^*_{t_{\hi}}( P_w:A^{\lo}_i \to x )
= (1 \pm 5\dD)\aA^{\lo}_i$.

After INTERVALS, we can assume $x \in \ovXw$,
and then $\oO^*_{t_{\iv}}( P_w:A^{\lo}_i \to x )
= \ov{p}_w^{-1} 
 \oO^*_{t_{\hi}}( P_w:A^{\lo}_i \to x )$.
After EMBED $A_0$, similarly to the above
analysis for HIGH DEGREES, using 
Lemma \ref{lem:rga} in place of Lemma \ref{lem:hi},
whp $\oO^*_{t_{**}}( P_w:A^{\lo}_i \to x ) 
= (1 \pm 5D^{-.9}) \oO^*_{t_{\iv}}( P_w:A^{\lo}_i \to x )$
and $\oO^*_{t_0}( P_w:A^{\lo}_i \to x ) 
= (1 \pm 5D^{-.9} \pm 5\aA_0^{.9})
\oO^*_{t_{\iv}}( P_w:A^{\lo}_i \to x )
= (1 \pm .1\eps^{.8})
\oO^*_{t_{\iv}}( P_w:A^{\lo}_i \to x )$.

After DIGRAPH, the part $J^g_i$ of $J$ containing
$\ova{xw}$ is determined; we can assume $g=\lo$,
as we have already considered the other cases.
Each $\oO^*_{t_1}( P_w:u \to x )$ is $0$ unless 
$u \in A^{\lo}_i$ and we have the event $E_u$ that
$\lova{x\phi_w(u')}$ in $\ova{G}_{uu'}$ for the unique 
$u' \in N_<(u) \cap A_0$, in which case 
$\oO^*_{t_1}( P_w:u \to x ) = |A^{\lo}_i|^{-1} 
p_{uu'}^{-1}$.
By Lemma \ref{lem:DI} we have 
\begin{align*}
& \quad \mb{E}^{t_0} \oO^*_{t_1}( P_w:A^{\lo}_i \to x ) 
= \sum_{u \in A^{\lo}_i} \mb{P}^{t_0}(E_u) |A^{\lo}_i|^{-1}
  p_{uu'}^{-1} \\
& = \sum_{u \in A^{\lo}_i} 
 (p_{uu'}/p_1)
 \cdot \oO^*_{t_0}( P_w:u \to x ) \ov{p}_w (\aA^{\lo}_i)^{-1}
  (p_1/p_{uu'}) \pm \DD^{-.9}
= 1 \pm .2\eps^{.8}. 
\end{align*}
For concentration, note that for each $x\phi_w(u') \in G^*$,
the assignment in DIGRAPH affects 
$\oO^*_{t_1}( P_w:A^{\lo}_i \to x )$ by
$\le |N_>(u') \cap A^{\lo}_i| \pmax^{-1} |A^{\lo}_i|^{-1}
\le \DD \pmax^{-1} (\dD n)^{-1}$.
Thus by Lemma \ref{freedman} whp
$\oO^*_{t_1}( P_w:A^{\lo}_i \to x ) 
= 1 \pm .3\eps^{.8}$.
\end{proof}

Next we give a significantly better estimate for
$\oO^*_{t_1}( P_w:A^g_i \to x )$ 
for $\ova{xw} \notin J^{\bad}_i$.

\begin{lemma} \label{lem:x:t1:good}
If $\ova{xw} \in J^{\lo}_i \sm J^{\bad}_i$ then whp
$\oO^*_{t_1}( P_w:A^{\lo}_i \to x ) = 1 \pm \eps_1$. 
\end{lemma}

\begin{proof}
By the proof of Lemma \ref{lem:x:t1}, whp 
$\oO^*_{t_0}( P_w:A^{**}_i \to x )
= \oO^*_{t_{**}}( P_w:A^{**}_i \to x )
= (1 \pm 6D^{-.9}) (\ov{p}_w n)^{-1} |A^{**}_i| \pm \dD$,
and it suffices to show
$\oO^*_{t_0}( P_w:A_i' \to x )
= (\ov{p}_w n)^{-1} |A_i'| \pm \dD$.
For any $u \in A_i'$ we have
$\oO^*_{t_0}( P_w:u \to x ) 
= p_{uu'}/p_1 \cdot \oO^*_{t_{\iv}}(P_w: u \to x)
\cdot 1_{u' \in S_i^w \cap N^-_{\ova{G}_1}(x)}$, so
by definition of $J_i^{\bad}$ we have
$\oO^*_{t_0}( P_w:A_i' \to x ) 
=  (p_1 \ov{p}_w n)^{-1} | S^w_i \cap N^-_{\ova{G}_1}(x)|
= (\ov{p}_w n)^{-1} |A'_i| \pm \dD$.
The lemma follows. 
\end{proof}

Next we give an estimate that will be used 
in several further lemmas below.
For any $U \sub V(T) \sm A_0$ we let 
$\GG^2(U) = \{ v: \text{dist}_{T \sm A_0}(v,U) \le 2\}$,
where dist denotes graph distance.
For $w \in W$ we let
\[ U_w = \{u \in P^\tau_w(\cdot \to x) 
\cap \GG^2(A^{\lo}): N_<(u) \cap A^{\bad}_w \ne \es \}.\]

\begin{lemma} \label{lem:x'}
If $u \notin U_w$ and $u' \in N_<(u) \cap A_{i'}$ then
$\sum_{x' \in N^-_{\ova{G}_{uu'}}(x)}
    \oO'( \edge{w}{u'}{x'} ) 
 = (1 \pm 2.2\eps_{i'})p_{uu'}$.
\end{lemma}

\begin{proof}
We note that $\oO'( \edge{w}{u'}{x'} )
  = (1 \pm 2\eps_{i'}) \oO( \edge{w}{u'}{x'} )$ for any
$x' \in P^{t_{i'}}_w( u' \to \cdot) \sm N_{J^{\bad}}(w)$.
Indeed, this holds by Lemma~\ref{lem:Hwdeg}.iii for $i'$,
as $u' \notin A^{\bad}_w \cap N_{F'}(A^{\lo})$ 
by definition of $U_w$.
We deduce $ \oO( \edge{w}{u'}{x'})
= \oO^*_{t_{i'}}(P_w:u' \to x') = |A_{u'}|^{-1} 
 \prod_{v \in N_<(u') \cap A_{<i'}} p_{u'v}^{-1}
= ((1 \pm 5\xi) |P^{t_{i'}}_w(u' \to \cdot)|)^{-1}$
by Lemma \ref{lem:Px}, so
\[ \sum_{x' \in N^-_{\ova{G}_{uu'}}(x)}
    \oO'( \edge{w}{u'}{x'} )
=(1 \pm 2.1\eps_{i'} ) |P^{t_{i'}}_w(u' \to \cdot)|^{-1}
 \cdot |P^{t_{i'}}_w(u' \to \cdot) 
  \cap N^-_{\ova{G}_{uu'}}(x)|
 \pm \pmax^{-4}\dD^2,\]
with $\pm \pmax^{-4}\dD^2$ accounting 
for $x' \in N_{J_{\bad}}(w)$ when $u' \notin A^{\hi}$,
and so $\oO'( \edge{w}{u'}{x'} ) \le \pmax^{-4} (\dD n)^{-1}$.
The lemma now follows by Lemma \ref{lem:DI},
similarly to the discussion before Lemma~\ref{lem:Px}.
\end{proof}

Now we consider the evolution of
$\oO^*_t( P_w:A_{\ova{xw}} \to x )$
during some step $i'<i$
of APPROXIMATE DECOMPOSITION.
For lighter notation we write 
$\tau=t_{i'}$ and $\tau' = t^+_{i'}$.

\begin{lemma} \label{lem:x:ti:hyp}
whp $\oO^*_{\tau'}( P_w:A_{\ova{xw}} \to x ) = 
(1 \pm \eps_{i'}^{.9}) \oO^*_{\tau}( P_w:A_{\ova{xw}} \to x )$.
\end{lemma}

\begin{proof}
We consider the function $\psi$ 
on $\tbinom{{\cal H}_{i'}}{{\le}4}$,
where $\psi(E_u)=0$ except if there is 
$u \in P^\tau_w(\cdot \to x) \cap A_{>i'} \sm U_w$ 
such that  $E_u$ consists of disjoint edges $\edge{w}{u'}{x'}$ 
with $\lova{xx}' \in \ova{G}_{uu'}$ 
for each $u' \in N_<(u) \cap A_{i'}$,
and then $\psi(E_u) = |A_u|^{-1}
 \prod_{u' \in N_<(u) \cap A_{\le i'}} p_{uu'}^{-1}$.
Now, for $u \notin U_w$ with $N_{<}(u) \cap A^0_{i',w} =\es$, 
given $x \in P^{\tau}_w(u \to \cdot)$,
we have $x \in P^{\tau'}_w(u \to \cdot)$
iff $E_u \subseteq \mc{M}_{i'}$, and for such $x$
we have $\psi(E_u)=\oO(\edge{w}{u}{x})$.
Thus $\oO^*_{\tau'}( P_w:A_{\ova{xw}} \to x )
= \sum\{\oO(\edge{w}{u}{x}):x \in P_w^{\tau'}(u \to \cdot)\}
= \psi({\cal M}_{i'}) \pm \DD_\psi \pm \DD'_\psi$,
where 
\begin{align*} 
\DD_\psi &= \sum \{ \oO^*_{\tau'}( P_w:u \to x ) 
 : N_<(u) \cap A^0_{i',w} \ne \es  \}
\le \pmax^{-4} |A_{\ova{xw}}|^{-1} 
 |F'[A^0_{i',w},A_{\ova{xw}}]| 
< 9\eps_{i'} \pmax^{-5},\\
\DD'_\psi 
&= \sum_{u \in U_w} \oO^*_{\tau'}( P_w:u \to x )
\le \pmax^{-5} |A^{\bad}_w| |A_{\ova{xw}}|^{-1}
< \pmax^{-5} \dD^2.
\end{align*} 
Here we bounded $\DD_\psi$ by Lemma~\ref{lem:Tleave}.vi,
and $\DD'_\psi$ by Lemma~\ref{lem:B}, also assuming
$|A_{\ova{xw}}| \geq \dD n$, as we may,
as if $\ova{xw} \in J^{\hi}$ then $U_w=\es$,
using $A^{\hi} \cap \GG^2(A^{\lo}) = \es$
by the definition of $A_0$ as a $4$-span.

Next we estimate
\[ \psi({\cal H}_{i'},\oO') = 
\sum_{u \in P^\tau_w(\cdot \to x) \sm U_w} |A_u|^{-1}
   \prod_{u'' \in N_<(u) \cap A_{\le i'} } p_{uu''}^{-1} 
   \prod_{u' \in N_<(u) \cap A_{i'} } 
   \sum_{x' \in N^-_{\ova{G}_{uu'}}(x)}
    \oO'( \edge{w}{u'}{x'} ) .\]
By Lemma \ref{lem:x'} we have    
\begin{align*}
 \psi({\cal H}_{i'},\oO') 
 & = \sum_{u \in P^\tau_w(\cdot \to x) \sm U_w} |A_u|^{-1}
 \prod_{u'' \in N_<(u) \cap A_{< i'} } p_{uu''}^{-1} 
 \prod_{u' \in N_<(u) \cap A_{i'} } 
  p_{uu'}^{-1} (1 \pm 2.2\eps_{i'}) p_{uu'} \\
 & = (1 \pm 8.9\eps_{i'})  (\oO^*_{\tau}( P_w:A_{\ova{xw}} \to x )
\pm \pmax^{-5} \dD^2),
\end{align*}
with the error term as in the estimate for $\DD'_{\psi}$.
The lemma now follows from Lemma \ref{lem:wEGJ},
noting that each $|E_u \cap E_{u'}| \le 1$
(otherwise $u,u'$ would have two common neighbours),
and for any $\bm{e}=\edge{w}{u'}{x'}$ we have
$f_{\{\bm{e}\}}(\mc{H}_{i'},\oO')
\le \sum_{u \in N_>(u')}2|A_u|^{-1}\pmax^{-4}
\le 2\pmax^{-5}\DD^{-1} 
< C^{-\beta}\psi(\mc{H}_{i'},\oO')$.
\end{proof}

Similarly to Lemma \ref{lem:rga},
we have the following estimates during the
embedding of $A^0_{i',w}$
(which has size $< 2.1\eps_{i'} n$
by Lemma \ref{lem:Tleave}).

\begin{lemma} \label{lem:rga2}
For any $w \in W$, $W' \sub W$,
$a \in A^0_{i',w} \cap A^{g'}_{i'}$, 
$x,y \in V(G)$, $i',i \in [i^*]$, $g,g' \in \{\hi,\lo,\no\}$,
writing $A^w_a$ for the set
of $y$ such that $\phi_w(a)=y$ is possible
given the history at time $t_a^-$, whp 
\begin{enumerate}
\item
$\mb{P}^{t_a^-}(\phi_w(a)=y)  
= (1 \pm \eps_{i'}^{.9})|A^w_a|^{-1}$, 
\item
$\mb{P}^{t_a^-}(\phi_w(a) \in N^-_{\ova{G}^{gg'}_{ii'}}(x) )  
= (1 \pm \eps_{i'}^{.9})p^{gg'}_{ii'}$,
\item
$|\{w \in W': \phi_w(a) \in N^-_{\ova{G}^{gg'}_{ii'}}(x) \}|
= (1 \pm \eps_{i'}^{.9})p^{gg'}_{ii'} |W'| \pm n^{.8}$.
\end{enumerate}
\end{lemma}

\begin{proof} 
Recall that for each $a \in A_{i'}$ in any order,
we let $W_a = \{w \in W: \phi_w(a)$ undefined$\}$,
let $V_a \in \tbinom{V}{|W_a|}$
be uniformly random,
and define $M_a = \{ \phi_w(a) w: w \in W_a\}
{=}$MATCH$(B_a,Z_a)$,
where $Z_a = \{ \phi_w(b)w: b \in N_<(a) \}$
and $B_a \sub V_a \times W_a$ 
consists of all $vw$ with 
$v \in N_{J'_i}(w) \sm \im \phi_w$
and each $\lova{v\phi_w(b)}$ for $b \in N_<(a)$ 
an unused edge of $G'_{i'}$. 

To justify the application of Lemma \ref{lem:match}
in defining $M_a$, we first note that
by Lemma~\ref{lem:Tleave}.v,
$|W_a| > .3\eps_{i'}n$
for all $a \in A_{i'}$.
Also $Z_a$ has maximum degree $\le 4$.
We also claim whp $B_a$ is $\eps_{i'}^{.7}$-super-regular
of density $(1 \pm \eps_{i'}^{.7}) \pmax^{|N_<(a)|+1}$.
To see this, we argue similarly 
to Lemmas \ref{lem:hi} and \ref{lem:rga},
except that Lemma \ref{lem:SRLtoH} is not applicable,
so we instead apply Lemma \ref{lem:DLR}.
We have 
$|W_a|=|V_a| \ge .3\eps_{i'} n$.
We let $G_{\free} \sub G'_{i'}$ denote
the graph of unused edges, and let ${\cal B}$
be the bad event that $G'_{i'} \sm G_{\free}$
has any vertex of degree $>.1\eps_{i'}^{.9} n$.
We will establish the claim at any step 
before ${\cal B}$ occurs, 
assuming the claim for any $b \prec a$,
and deduce that whp ${\cal B}$ does not occur.

Consider any $R \in \tbinom{W_a}{\le 2}$.
We have $N_{B_a}(R)  =
 V_a \cap N^-_{J'_{i'}}(R)
 \cap N^+_{G_{\free}}\big(
 \bigcap_{w \in R} \phi_w(N_<(a)) \big)
\sm \bigcup_{w \in R} \im \phi_w$.
As ${\cal B}$ does not occur, 
by Lemma \ref{lem:DI} and a Chernoff bound whp 
$|N_{B_a}(R)| = (1 \pm \eps_{i'}^{.8})
 (\pmax^{|N_<(a)|+1})^{|R|} |V_a|$,
unless $R=\{w,w'\}$ with 
$\phi_w(N_<(a)) \cap \phi_{w'}(N_<(a)) \ne \es$;
by Lemma \ref{lem:wEGJ} there are
whp $<n^{1.5}$ such pairs $R$.

Now consider any $R' \in \tbinom{V_a}{\le 2}$.
Let $W^t$ be the set of $w$ such that
$\phi_w(b) \in N^-_{G'_{i'}}(R')$ for all $b \in N_<(a)$ 
with $\phi_w(b)$ defined at time $t$.
As $\mc{B}$ does not occur, 
$|N_{B_a}(R')| = |W_a \cap N^+_{J'_{i'}}(R')\cap W^{t^-_a}|
\pm .5\eps_{i'}^{.9}n$.
For any $b \in N_<(a)$ and $w \in N^+_{J'_{i'}}(R')$,
if $\phi_w(b)$ is defined by Lemma \ref{lem:match}
during HIGH DEGREES or APPROXIMATE DECOMPOSITION
then similarly to the proof of Lemma \ref{lem:rga},
writing $\tau=t_b^-$ and $\tau'=t_b$ whp
$|W^{\tau'}| = |M^b[N^-_{G'_{i'}}(R'),W^{\tau}]|
= (1 \pm \eps_j^{.8}) \pmax^{|R'|} |W^{\tau}|$ 
by Lemma \ref{lem:match},
assuming the claim for $b \in A_j$, $j<i'$.

Now suppose $\phi_w(b)$ is defined 
by the hypergraph matching in ${\cal H}_j$.
We consider the function $f$ on ${\cal H}_j$ 
with $f(\edge{w}{u}{x})
= 1_{w \in W^{\tau}} 1_{u=b}  1_{x \in N^-_{G'_{i'}}(R')}$,
so $f({\cal H}_j,\oO') = \sum_{x \in N^-_{G'_{i'}}(R')}
\sum_{w \in W^{\tau}} \oO'(\edge{w}{b}{x})$.
Similarly to the proof of Lemma \ref{lem:x'},
if $b \notin A^{\bad}_w$ then
$\sum_{x \in N^-_{G'_{i'}}(R')} \oO'(\edge{w}{b}{x})
= (1 \pm 2.2\eps_j)\pmax^{|R'|}$.
By Lemma \ref{lem:Tleave}.iv whp
$|\{w: b \in A^{\bad}_w\}| < 5\dD^4 n$,
so $f({\cal H}_j,\oO') = \sum_{x \in N^-_{G'_{i'}}(R')}
\sum_{w \in W^{\tau}} \oO'(\edge{w}{b}{x})
= (1 \pm 2.3\eps_j) |N^-_{G'_{i'}}(R')| |W^{\tau}|$,
and by Lemma \ref{lem:wEGJ} whp
$|W^{\tau'}| = (1 \pm 3\eps_j) \pmax^{|R'|} |W^{\tau}|$. 

Together with Lemma~\ref{lem:DI}.ii
this proves the claim, and so justifies the definition 
of $M_a$. Statements (i--iii) 
of the lemma now follow directly from Lemma \ref{lem:match},
considering $M_a[W',N^-_{\ova{G}^{gg'}_{ii'}}(x)]$ for (iii).
Also, from (i) whp every vertex degree 
in $G'_{i'} \sm G_{\free}$ 
is $(.1\eps_{i'}^{.9} n,4)$-dominated,
so whp ${\cal B}$ does not occur.
\end{proof}

We deduce the following estimate
similarly to the proof of Lemma \ref{lem:x:t1},
using Lemma \ref{lem:rga2} in place of Lemma \ref{lem:rga}.

\begin{lemma} \label{lem:x:ti:rga}
whp all $\oO^*_{t_{i'+1}}( P_w:A_{\ova{xw}} \to x ) 
= (1 \pm \eps_{i'}^{.8}) 
 \oO^*_{t^+_{i'}}( P_w:A_{\ova{xw}} \to x )$.
\end{lemma}

We conclude by deducing the estimates on
$\oO({\cal H}_i[\ova{xw}]) = \oO^*_{t_i}( P_w:A_{\ova{xw}} \to x )$
required for Lemma \ref{lem:bad}.
By Lemma \ref{lem:x:t1} whp
$\oO^*_{t_1}( P_w:A_{\ova{xw}} \to x ) = 1 \pm .3\eps^{.8}$,
and by Lemmas \ref{lem:x:ti:hyp} and \ref{lem:x:ti:rga} 
whp each $\oO^*_{t_{i'+1}}( P_w:A_{\ova{xw}} \to x ) 
 = (1 \pm 2\eps_{i'}^{.8}) 
 \oO^*_{t^+_{i'}}( P_w:A_{\ova{xw}} \to x )$,
so $\oO^*_{t_i}( P_w:A_{\ova{xw}} \to x ) 
= 1 \pm .4\eps^{.8}$.
Also, if $\ova{xw} \notin J^{\bad}_i$
by Lemma \ref{lem:x:t1:good} whp
$\oO^*_{t_1}( P_w:A_{\ova{xw}} \to x ) = 1 \pm \eps_1$,
and repeating the previous calculations gives
$\oO^*_{t_i}( P_w:A_{\ova{xw}} \to x ) 
= 1 \pm 3\eps_{i-1}^{.8} = 1 \pm \eps_i$.
This completes the proof of Lemma~\ref{lem:bad}.ii.

\subsection{$G$ degrees}

This subsection concerns $\oO({\cal H}_i[\lova{xy}])$
for $\lova{xy} \in \ova{G}_i$. 
We start by establishing Lemma \ref{lem:bad}.iii, 
which is Lemma~\ref{lem:ux} below,
as this is needed for the analysis
(and also for Lemma \ref{lem:Gleave} below).

\begin{lemma} \label{lem:ux}
whp $\oO^*_{t_i}( P_W: u \to x ) = 1 \pm \eps_i$
for each $x \in V(G)$ and $u \in A_i$.
\end{lemma}

We start with the corresponding estimate at time $t_1$.

\begin{lemma} \label{lem:t1:ux}
whp $\oO^*_{t_1}( P_W: u \to x ) = 1 \pm \eps_1$
for each $x \in V(G)$ and $u \in A_i$.
\end{lemma}

\begin{proof}
We consider cases according to the location of $u$.

If $u \in A^a_i$ 
with $a \in A^\DD_i$
then $\oO^*_{t_1}( P_w: u \to x ) 
 = |A_u|^{-1} 1_{\phi_w(a)x \in H^a_i}$,
so $\oO^*_{t_1}( P_W: u \to x )
= |A_u|^{-1} d^-_{H_i^a}(x)
= 1 \pm \eps_1$ by Lemma \ref{lem:DI}.viii.

If $u \in A^{\no}_i$ then
$\oO^*_{t_1}( P_w: u \to x ) 
= |A_u|^{-1} 1_{\ova{xw} \in J_u}$,
so by Lemma \ref{lem:DI} and a Chernoff bound
whp $\oO^*_{t_1}( P_W: u \to x ) = 1 \pm \DD^{-.1}$.

If $u \in A^{\lo}_i$ 
then $\oO^*_{t_1}( P_w: u \to x ) 
=  p_u^{-1} |A_u|^{-1} 1_{\ova{xw} \in J_u}
 1_{\phi_w(u') \in N^-_{\ova{G}_{u0}}(x)}$,
where $N_<(u) \cap A_0 = \{u'\}$.
For each $x' \in N_{G^*}(x)$ there is a unique 
$w' \in W$ with $\phi_{w'}(u')=x'$, so
$\oO^*_{t_1}( P_W: u \to x ) 
= \sum_{x' \in N_{G^*}(x)} p_{u0}^{-1} |A_u|^{-1}
 1_{\lova{xx}' \in \ova{G}_{u0}} 1_{\ova{xw}' \in J_u}$.
As $|N_{G^*}(x)| = (1 \pm 2\xi)p n$,
by Lemma \ref{lem:DI} and a Chernoff bound
whp $\oO^*_{t_1}( P_W: u \to x ) = 1 \pm 3\xi$.
\end{proof}

Next we consider the evolution 
of $\oO^*_t( P_W:u \to x )$ at step $i'<i$
in the approximate decomposition,
again writing $\tau=t_{i'}$, $\tau' = t^+_{i'}$.
 
\begin{lemma} \label{lem:ux:ti:hyp} 
whp $\oO^*_{\tau'}( P_W:u \to x ) 
= (1 \pm \eps_{i'}^{.8}) 
\oO^*_{\tau}( P_W: u \to x )$.
\end{lemma}

\begin{proof}
Let $W' = \{w \in P^{\tau}_{\cdot}(u \to x): u \notin U_w\}$.
Consider the function 
$\psi$ on $\tbinom{{\cal H}_{i'}}{{\le}4}$
where $\psi(I)=0$ except if there is $w \in W'$ 
such that $I$ consists of  disjoint edges 
$\edge{w}{u'}{x'}$ with $\lova{xx}' \in \ova{G}_{uu'}$ 
for each $u' \in N_<(u) \cap A_{i'}$,
and then $\psi(I) = |A_u|^{-1}
\prod_{u' \in N_<(u) \cap A_{\le i'}} p_{uu'}^{-1}$.
Note that $\oO^*_{\tau'}( P_W:u \to x )
= \psi({\cal M}_{i'}) \pm \DD_\psi \pm \DD'_\psi$,
where
\begin{align*}
\DD_{\psi} &= 
\pmax^{-4} |A_u|^{-1}\sum_{u' \in N_<(u) \cap A_{i'}} 
|\{w \in N_{J_u}(x): u' \in A^0_{i',w}\}|
< 4\pmax^{-4} |A_u|^{-1}
(2.1\eps_{i'}d_{J_u}(w) + \dD|A_u|)
< \eps_{i'}^{.9},\\
\DD'_{\psi} &= \pmax^{-4} |A_u|^{-1}
|\{w \in W: u \in U_w\}| \le
\pmax^{-4} (\dD n)^{-1} 
\sum_{u' \in N_<(u)} |\{w: u' \in A^{\bad}_w \}|
< \dD.
\end{align*}
Here we used Lemma~\ref{lem:Tleave}
and Lemma~\ref{lem:DI}.ii,viii
to estimate $\DD_{\psi}$,
and for $\DD'_\psi$ we used
Lemma~\ref{lem:B}.v and Lemma~\ref{lem:Tleave}.iv,
also noting that if $\DD'_\psi \ne 0$
then $u,u' \notin A^{\hi}$ by definition of $A_0$,
so $|A_u| \ge \dD n$. Finally
\begin{align*}
\psi({\cal H}_{i'},\oO') 
& = \sum_{w \in W'} |A_u|^{-1}
   \prod_{u' \in N_<(u) \cap A_{\le i'} } p_{uu'}^{-1} 
   \prod_{u' \in N_<(u) \cap A_{i'} } 
   \sum_{x' \in N^-_{\ova{G}_{uu'}}(x)}
    \oO'( \edge{w}{u'}{x'} ) \\
& = (1 \pm 9\eps_{i'}) \oO^*_{\tau}( P_W:u \to x ) 
\end{align*}
by Lemma \ref{lem:x'}, and 
the lemma now follows from Lemma \ref{lem:wEGJ}.
\end{proof}

We deduce Lemma \ref{lem:ux} 
(i.e.\ Lemma~\ref{lem:bad}.iv)
from the previous two lemmas
and the following estimate
which holds similarly to
Lemma \ref{lem:x:ti:rga},
using Lemma \ref{lem:rga2}.iii.

\begin{lemma} \label{lem:ux:ti:rga}
whp all $\oO^*_{t_{i'+1}}( P_W:u \to x ) 
= (1 \pm \eps_{i'}^{.8}) 
 \oO^*_{t^+_{i'}}( P_W:u \to x )$.
\end{lemma}

Now we turn to the degrees of $\lova{xy} \in \ova{G}_i$.
We consider the evolution of
$\oO^*_t( P_W:F'[A^g_i,A^{g'}_j] \to \lova{xy} )$,
where $\lova{xy} \in \ova{G}^{gg'}_{ij}$,
$0 \le j < i$, and for convenient notation 
we label arcs of $F'[A^g_i,A^{g'}_j]$
as $uv$ with $u \in A^g_i$, $v \in A^{g'}_j$.
Recall that $\ova{G}^{gg'}_{i0}=\ova{G}^g_{i0}$
for all $g'$.
We start with an estimate at time $t_1$.

\begin{lemma} \label{lem:t1:xy}
whp $\oO^*_{t_1}( P_W: F'[A^g_i,A^{g'}_j] \to \lova{xy} )
= 1 \pm \eps_1$ for all $\lova{xy} \in \ova{G}^{gg'}_{ij}$.
\end{lemma}

\begin{proof} 
We consider cases according to $g,g',i,j$ 
where $0 \le j < i$.

We start with the case 
$\lova{xy} \in \ova{G}^{\lo}_{i0}$.
For each $uv \in F'[A^{\lo}_i,A_0]$, $w \in W$
we have $\oO^*_{t_1}( P_w:uv \to \lova{xy} )=0$, except 
for the unique $w^v \in W$ with $\phi_{w^v}(v)=y$,
for which  $\oO^*_{t_1}( P_{w^v}:uv \to \lova{xy} )
 = \oO^*_{t_1}( P_{w^v}:u \to x)
 = (p^{\lo}_{i0} |A^{\lo}_i|)^{-1} 
 1_{xw^v \in J^{\lo}_i}$,
so $\oO^*_{t_1}( P_W:uv \to \lova{xy} )
=  (p^{\lo}_{i0} |A^{\lo}_i|)^{-1} 
 1_{xw^v \in J^{\lo}_i}$.
The $w^v$ are distinct, so the events 
$\{xw^v \in J^{\lo}_i\}$ are independent.
Each affects $\oO^*_{t_1}( P_W: uv \to \lova{xy} )$
by $< p_{uv}^{-1}\pmax^{-1} (\dD n)^{-1}$,
so by Lemma \ref{freedman} whp 
$\oO^*_{t_1}( P_W:F'[A^{\lo}_i,A_0] \to \lova{xy} )
= (p_{uv} n)^{-1} |F'[A^{\lo}_i,A_0]| \pm n^{-.4}
= p_{uv}^{-1}(p_{uv} - \pmin) \pm n^{-.4}
= 1 \pm \eps_1$.

Next consider the case
$\lova{xy} \in \ova{G}^{\hi,\no}_{ij}$, $j \in [i-1]$.
For each $uv \in F'[A^{\hi}_i,A^{\no}_j]$, $w \in W$
we have $\oO^*_{t_1}( P_w:uv \to \lova{xy} )=0$, 
except if $\ova{xw} \in J_u$ and $\ova{yw} \in J_v$,
when $\oO^*_{t_1}( P_w:uv \to \lova{xy} )
 = (p_{uv} |A_u| |A_v|)^{-1}$.
We have $\ova{xw} \in J_u$ iff
$\phi_w(a) \in N^-_{H^a_i}(x)$, where $u \in A^a_i$,
and as $d(x,y)>3d$ since any close edges were removed
from $G_1$, the events $\lova{xy} \in \ova{G}_{uv}$, 
$\{\ova{xw} \in J_u\}$ and $\{\ova{yw} \in J_v\}$
are conditionally independent given HIGH DEGREES. 
By Lemma~\ref{lem:DI} and a Chernoff bound whp there are
$(1 \pm \DD^{-.6}) \aA_v d^-_{H_i^a}(x)$
choices of $w$ with 
$\oO^*_{t_1}( P_w:uv \to \lova{xy} ) \ne 0$,
so $\oO^*_{t_1}( P_W:uv \to \lova{xy} )
= (1 \pm \DD^{-.6}) d^-_{H_i^a}(x)
p_{uv}^{-1} |A_u|^{-1}n^{-1}
 = (1 \pm .2\eps_1) (p_{uv} n)^{-1}$ 
by Lemma \ref{lem:DI},
giving the required estimate.
 
The case $\lova{xy} \in \ova{G}^{\no,\hi}_{ij}$,
$j \in [i-1]$ is similar to the previous one.

Now consider the case
$\lova{xy} \in \ova{G}^{\no,\no}_{ij}$, $j \in [i-1]$.
For each $uv \in F'[A^{\no}_i,A^{\no}_j]$, $w \in W$
we have $\oO^*_{t_1}( P_w:uv \to \lova{xy} )=0$, 
except if $\ova{xw} \in J_u$ and $\ova{yw} \in J_v$,
when $\oO^*_{t_1}( P_w:uv \to \lova{xy} )
 = (p_{uv} |A_u||A_v|)^{-1}$.
By a Chernoff bound whp
$\oO^*_{t_1}( P_W:uv \to \lova{xy} )
= \sum_{w \in W} \oO^*_{t_1}( P_w:uv \to \lova{xy} )
= (p_{uv} n)^{-1} \pm n^{-1.4}$,
giving the required estimate.

Finally, consider the case
$\lova{xy} \in \ova{G}^{\lo,\no}_{ij}$,  $j \in [i-1]$.
(The case $\lova{xy} \in \ova{G}^{\no,\lo}_{ij}$
is similar, and this is the last case by the
definition of $A_0$ as a $4$-span.)
For each $uv \in F'[A^{\lo}_i,A^{\no}_j]$, $w \in W$
we have $\oO^*_{t_1}( P_w:uv \to \lova{xy} )=0$, 
except in the event $E_{uvw}$ 
that $\ova{xw} \in J_u$, $\ova{yw} \in J_v$ 
and $\lova{xx}' \in \ova{G}^{\lo}_{i0}$, 
where $x'=\phi(u')$, $\{u'\}=N_<(u) \cap A_0$,
when $\oO^*_{t_1}( P_w:uv \to \lova{xy} )
 = (p_{uv} p^{\lo}_{i0} |A_u||A_v|)^{-1}$.
For each $x' \in V(G)$ there is a unique 
$w' \in W$ with $\phi_{w'}(u')=x'$, 
so $\oO^*_{t_1}( P_W:uv \to \lova{xy} )
 = \sum_{w \in W} \oO^*_{t_1}( P_w:uv \to \lova{xy} )
= \sum_{x' \in N^-_{\ova{G}_1}(x)} 1_{E_{uvw}}
 (p_{uv} p^{\lo}_{i0} |A_u||A_v|)^{-1}$.
We have $\mb{E}^{t_0} \oO^*_{t_1}( P_W:uv \to \lova{xy} )
 = (p^{\lo}_{i0}/p_1) d^-_{\ova{G}_1}(x)
 (p_{uv} p^{\lo}_{i0} n^2)^{-1}$, where 
whp $d^-_{\ova{G}_1}(x) = (1 \pm 2\xi) p_1 n$.
The decisions on $\ova{xw}$ and $\ova{yw}$ affect
$\oO^*_{t_1}( P_W:uv \to \lova{xy} )$ by
$< (p_{uv}p^{\lo}_{i0})^{-1}(\pmax \dD n)^{-2}$.
For each $\lova{xx}'$,
note that there are $n$ choices of $w'$,
which determines $u' = \phi_{w'}^{-1}(x')$,
then  $<\DD$ choices for each of $u$ and $v$,
so the decision on $\{\lova{xx}' \in G^{\lo}_{i0}\}$
affects $\oO^*_{t_1}( P_W:uv \to \lova{xy} )$ by
$< n \DD (p_{uv}p^{\lo}_{i0})^{-1}(\pmax \dD n)^{-2}$.
The required estimate now follows
from Lemma \ref{bernstein}.
\end{proof}

Next we consider the evolution of 
$\oO^*_t( P_W:F'[A^g_i,A^{g'}_j] \to \lova{xy} )$
at step $i'<i$ in the approximate decomposition,
again writing $\tau=t_{i'}$, $\tau' = t^+_{i'}$.
  
\begin{lemma} \label{lem:xy:ti:hyp} 
At step $i'<i$ whp 
$\oO^*_{\tau'}( P_W:F'[A^g_i,A^{g'}_j] \to \lova{xy} )$ 
is $(1 \pm .7\eps^{.8}) 
\oO^*_{\tau}( P_W:F'[A^g_i,A^{g'}_j] \to \lova{xy} )$
for each $0 \le j < i$ 
and $\lova{xy} \in \ova{G}^{gg'}_{ij}$, 
and is $(1 \pm \eps_{i'}^8) 
\oO^*_{\tau}( P_W:F'[A^g_i,A^{g'}_j] \to \lova{xy} )$
if $j \ne i'$ or $y \notin B$ 
or $\lova{xy} \notin G^{\no,\lo}$.
\end{lemma}

\begin{proof} 
We start with the case $j<i'$. We note
for each $uv \in F'[A^g_i,A^{g'}_j]$, $w \in W$
that $\oO^*_{\tau}( P_w:uv \to \lova{xy}\ )$ and
 $\oO^*_{\tau'}( P_w:uv \to \lova{xy} )$ are $0$
unless $\phi_w(v)=y$, in which case
 $\oO^*_{\tau}( P_w:uv \to \lova{xy} )
 = \oO^*_{\tau}( P_w:u \to x)$,
and  $\oO^*_{\tau'}( P_w:uv \to \lova{xy} )
 = \oO^*_{\tau'}( P_w:u \to x)$.
By Lemma \ref{lem:ux:ti:hyp} we deduce
\begin{align*}
\oO^*_{\tau'}( P_W:F'[A^g_i,A^{g'}_j] \to xy )
& = \sum_{u \in A^g_i} |N_<(u) \cap A^{g'}_j|
(1 \pm \eps_{i'}^{.8})
\oO^*_{\tau}( P_W:u \to x ) \\
& = (1 \pm \eps_{i'}^{.8})
\oO^*_{\tau}( P_W:F'[A^g_i,A^{g'}_j] \to xy ).
\end{align*}

Now we may assume $j \ge i'$. Suppose $j=i'$.
We consider the function $\psi$ 
on $\tbinom{{\cal H}_{i'}}{{\le}4}$,
where $\psi(E)=0$ except if there are $w \in W$ and 
 $uv \in P^\tau_w(\cdot \to \lova{xy})
  \cap F'[A^g_i,A^{g'}_{i'}]$ 
 with $u \notin U_w$ such that 
$E$ consists of  disjoint edges $\edge{w}{u'}{x'}$ 
with $\lova{xx}' \in \ova{G}_{uu'}$ 
for each $u' \in N_<(u) \cap A_{i'}$
and $x'=y$ when $u'=v$, and then 
$\psi(E) = \oO^*_{\tau'}( P_w:u \to x ) = |A_u|^{-1}
\prod_{u' \in N_<(u) \cap A_{\le i'}} p_{uu'}^{-1} $.

Note that $\oO^*_{\tau'}( P_w:F'[A^g_i,A^{g'}_{i'}] \to \lova{xy} )
= \psi({\cal M}_{i'}) \pm \DD_\psi \pm \DD'_\psi$,
where
\begin{align*}
\DD_\psi  &= \sum_{uvw} \{
  \oO^*_{\tau'}( P_w:uv \to \lova{xy} ) 
 : N_<(u) \cap A^0_{i',w} \ne \es  \} \\
& \le  \sum_{w: \ova{xw} \in J^g_i}
 \sum_{v \in A_{\ova{yw}}}
 |F'[A^0_{i',w},N_>(v) \cap A^g_i]|  
 \pmax^{-8} |A^g_i|^{-1} |A_{\ova{yw}}|^{-1} 
< 9\eps_{i'} \pmax^{-9} 
 d^+_{J^g_i}(x) |A^g_i|^{-1}
< \eps_{i'}^{.9},\\
\DD'_{\psi} &= \sum_{vw}\{\oO^*_{\tau'}
(P_w: uv \to \lova{xy}): u \in U_w\}
\le \pmax^{-8} n (\dD n)^{-2} 
\sum_{u'\in N_<(u)}|\{w: u' \in A^{\bad}_w\}|
< \dD^{.9}.
\end{align*}
Here the bound on $\DD_\psi$ follows from 
Lemmas~\ref{lem:Tleave}.vi and~\ref{lem:DI},
and the bound on $\DD'_\psi$ by
Lemmas \ref{lem:Tleave}.iv and \ref{lem:B},
also noting that if $\DD'_\psi \ne 0$
then $u,v \notin A^{\hi}$
by definition of $A_0$,
so $|A^g_i|,|A^{g'}_j| \ge \dD n$.

Next we estimate
\begin{align*}
& \psi({\cal H}_{i'},\oO') = \sum_{uvw \in S, u \notin U_w}
 \left[  |A_u|^{-1}
  \prod_{u'' \in N_<(u) \cap A_{\le i'} } p_{uu''}^{-1} 
\ \oO'( \edge{w}{v}{y} )  
\prod_{u' \in N_<(u) \cap A_{i'} \sm \{v\} } 
   g_{uu'}(xw) \right] ,
\end{align*}
where $S := \{uvw: w \in W, uv \in P^\tau_w(\cdot \to \lova{xy}) 
  \cap F'[A^g_i,A^{g'}_j]\}$ and by Lemma \ref{lem:x'}
\[ g_{uu'}(xw) :=
\sum_{x' \in N^-_{\ova{G}_{uu'}}(x)}
    \oO'( \edge{w}{u'}{x'} )
 =  (1 \pm 2.2\eps_{i'})p_{uu'}  .\]
To obtain the required estimates in the case $j=i'$,
by Lemma~\ref{lem:wEGJ} it suffices to show that
$\oO'( \edge{w}{v}{y} )$ is
$(1 \pm .6\eps^{.8}) \oO( \edge{w}{v}{y} )$ 
when $v \in A^{\lo}$ 
(which is equivalent to $\lova{xy} \in G^{\no,\lo}$)
and that if $v \notin A^{\lo}$ or $y \notin B$
then the sum of
\[ f(uvw) : =(|A_u||A_v|)^{-1}
  \prod_{u' \in N_<(u) \cap A_{\le i'} } p_{uu'}^{-1} 
  \prod_{v' \in N_<(v) \cap A_{\le i'} } p_{vv'}^{-1}\]
over $uvw \in S$ 
for which $\ova{yw}\in J^{\bad}$ is at most $\sqrt{\dD}$,
the sum of $f(uvw)$ over $uvw \in S$ with 
$v \in A^{\bad}_w \cap N_<(A^{\lo})$
is at most $\sqrt{\dD}$, and for every other $uvw \in S$
we have $\oO'(\edge{w}{v}{y})=(1\pm 2\eps_{i'})\oO(\edge{w}{v}{y})$.

So first assume $v \in A^{\lo}$.
For each $y'=\phi_w(v')$, $v' \in N_<(v)$,
as $v \in A^{\lo}$ we have 
$v' \notin A^{\lo}$,
so $\oO({\cal H}_{i'}[\lova{yy}']) = 1 \pm \eps_{i'}$
by Lemma~\ref{lem:bad}.iii for $i'$.
Parts (i), (iii) imply that 
$\oO({\cal H}_{i'}[\bm{v}]) = 1 \pm .5\eps^{.8}$
for $\bm{v}=vw,\ova{yw}$, so
$\oO'( \edge{w}{v}{y} ) = (1- .5\eps_i)(1 \pm .5\eps^{.8})
\oO( \edge{w}{v}{y} )$, as required.

Next, suppose $y \notin B$ or $v \notin A^{\lo}$.
Consider those $uvw \in S$ with $\ova{yw} \in J^{\bad}$.
Then $y \notin B$ and $v \in A^{\lo}$, since otherwise 
$\ova{yw} \notin J^{\bad}$.
So $u \notin A^{\hi}$ by the definition of $A_0$, 
and therefore $|A_u|,|A_v| \ge \dD n$.
Since $y \notin B$, there are  $<\dD^3 n$ choices of $w$
with $\ova{yw} \in J^{\bad}$, 
so the sum of $f(uvw)$ with 
$\ova{yw} \in J^{\bad}$ is
$\le \dD^3 n \sum_{uv} 
\pmax^{-8} |A_u|^{-1}|A_v|^{-1}  < \sqrt{\dD}$.

All other terms $uvw$ have $\ova{yw} \notin J^{\bad}$.
Consider those $uvw \in S$ with
$v \in A^{\bad}_w \cap N_<(A^{\lo})$.
Then $v,u \in A^{\no}$
by the definition of $A_0$ as a $4$-span, so
$|A_u|,|A_v| \ge \dD n$.
Each remaining $w$ has $<\dD^3 n$ choices 
of $v \in A^{\bad}_w$ by Lemma~\ref{lem:B}, 
so the total sum of these terms $f(uvw)$ is
$\le \sum_w \dD^3 n   
 \pmax^{-9} |A_u|^{-1}|A_v|^{-1}  < \sqrt{\dD}$.

Since all other terms $uvw$ have $\ova{yw} \notin J^{\bad}$,
by Lemma~\ref{lem:bad}.iii for $i'$, we may assume that
there are (not necessarily distinct)
$v',v'' \in N_<(v)$ with
$\phi_w(v') \in B$ and $v'' \in A^{\lo}$,
or else we have
$\oO'(\edge{w}{v}{y})=
(1-.5\eps_{i'})(1\pm \eps_{i'})\oO(\edge{w}{v}{y})$.
But then $v \in A^{\bad}_w \cap N_<(A^{\lo})$,
proving the claim
and completing the case $j=i'$.
 
Finally, we suppose $i'<j<i$.
We consider the function $\psi$ 
on $\tbinom{{\cal H}_{i'}}{{\le}8}$,
where $\psi(E)=0$ except if there are $w \in W$ and 
 $uv \in P^\tau_w(\cdot \to xy)\cap F'[A^g_i,A^{g'}_j]$ 
 with $u \notin U_w$ such that 
$E$ consists of  disjoint edges $\edge{w}{u'}{x'}$ 
with $\lova{xx}' \in \ova{G}_{uu'}$ 
for each $u' \in N_<(u) \cap A_{i'}$
and $\lova{yx}' \in \ova{G}_{vv'}$ 
for each $v' \in N_<(v) \cap A_{i'}$, and then 
$\psi(E) = f(uvw)$.
Note that $\oO^*_{\tau'}( P_W:F'[A^g_i,A^{g'}_j] \to \lova{xy} )
= \psi({\cal M}_{i'}) \pm \DD_\psi \pm \DD'_{\psi}$,
with $\DD_\psi$ as in the case $j=i'$ and 
\begin{align*}
 \DD_{\psi'} & = \sum_{uvw} 
 \{ \oO^*_{\tau'}( P_w:uv \to \lova{xy} ) 
 :  N_<(v) \cap A^0_{i',w} \ne \es  \}
\le  \sum_{w: \ova{xw} \in J^g_i} \pmax^{-8} 
 |A^g_i|^{-1} |A_{\ova{yw}}|^{-1} 
  \pmax^{-1} |F'[A^0_{i',w},A_{\ova{yw}}]|  \\
& <  9\eps_{i'} \pmax^{-10} 
 \sum_{w: \ova{xw} \in J^g_i} |A^g_i|^{-1}
< \eps_{i'}^{.9}
\end{align*}
by Lemmas \ref{lem:Tleave}.vi and \ref{lem:DI}.
Now we estimate
\begin{align*}
& \psi'({\cal H}_{i'},\oO')  =   \sum_{uvw \in S, u \notin U_w}
 \left[ f(uvw)  \prod_{u' \in N_<(u) \cap A_{i'} } 
     g_{uu'}(xw)
   \prod_{v' \in N_<(v) \cap A_{i'} } 
g_{vv'}(yw) \right].
\end{align*}
By Lemma \ref{lem:x'}
each $g_{uu'}(xw)
 = (1 \pm 2.2\eps_{i'})p_{uu'}$
and $g_{vv'}(yw)
 = (1 \pm 2.2\eps_{i'})p_{vv'}$,
so $\psi'({\cal H}_{i'},\oO') 
= (1 \pm \eps_{i'}^{.9}) 
 \oO^*_{\tau}( P_W:F'[A^g_i,A^{g'}_j] \to \lova{xy} )$.
The lemma now follows from Lemma \ref{lem:wEGJ}.
\end{proof}

Similarly to Lemma \ref{lem:x:ti:rga},
we also have the following estimate.

\begin{lemma} \label{lem:xy:ti:rga}
whp all $\oO^*_{t_{i'+1}}(  P_W:F'[A^g_i,A^{g'}_j] \to \lova{xy} ) 
= (1 \pm \eps_{i'}^{.8}) 
 \oO^*_{t^+_{i'}}(  P_W:F'[A^g_i,A^{g'}_j] \to \lova{xy})$.
\end{lemma}

We conclude the proof of Lemma \ref{lem:bad}
(and so of Lemma \ref{lem:Hwdeg}) 
by deducing the estimates on
$\oO({\cal H}_i[\lova{xy}]) 
= \oO^*_{t_i}( P_W:F' \to \lova{xy})$
required for Lemma \ref{lem:bad}.
For any $\lova{xy} \in \ova{G}^{gg'}_{ij}$, 
by Lemma \ref{lem:t1:xy}
whp $\oO^*_{t_1}( P_W: F'[A^g_i,A^{g'}_j] \to \lova{xy} )
= 1 \pm \eps_1$. At step $i'<i$,
by Lemmas \ref{lem:xy:ti:hyp} and \ref{lem:xy:ti:rga} 
whp $\oO^*_{\tau'}( P_W:F'[A^g_i,A^{g'}_j] \to \lova{xy} )$ 
is $(1 \pm .7\eps^{.8}) 
\oO^*_{\tau}( P_W:F'[A^g_i,A^{g'}_j] \to \lova{xy} )$,
and is $(1 \pm 3\eps_{i'}^{.8}) 
\oO^*_{\tau}( P_W:F'[A^g_i,A^{g'}_j] \to \lova{xy} )$
if $j \ne i'$ or $y \notin B$ 
or $\lova{xy} \notin G^{\no,\lo}$.
Thus $\oO({\cal H}_i[\lova{xy}]) 
= \oO^*_{t_i}( P_W:F'[A^g_i,A^{g'}_j]  \to \lova{xy})$
is $1 \pm \eps^{.8}$, and is 
$1 \pm 4\eps_{i-1}^{.8} = 1 \pm \eps_i$
if $j \ne i'$ or $y \notin B$ or $\lova{xy} \notin G^{\no,\lo}$.

\section{Exact decomposition} \label{sec:exact}

In this section we complete the proof of our main theorem,
in each of the cases S, P and L. We start in the first
subsection with some properties of the leftover graph
from the approximate decomposition 
required for cases S and P, then analyse each case
separately over the following subsections.

\subsection{Leftover graph}

In both cases S and P the approximate decomposition
constructs edge-disjoint copies $F_w$, $w \in W$
of $F = T \sm P_{\ex}$. The leftover graph 
$G'_{\ex} = G \sm \bigcup_{w \in W} \phi_w(F)$
is obtained from $G_{\ex}$ by adding all unused
edges of $G \sm G_{\ex}$ (and removing any orientations). 
We require the following
typicality properties.

\begin{lemma} \label{lem:leftover}
For any $w \in W$ and $S \in \tbinom{V(G)}{{\le}s}$ 
whp $|N^-_{J_{\ex}}(w) \cap G'_{\ex}(S)|
= (1 \pm p_0^{.9}) p'_{\ex} (2p_{\ex})^{|S|} n$.
\end{lemma} 

As in Lemma \ref{lem:DI}, a stronger form of this estimate 
holds with $G_{\ex}$ in place of $G'_{\ex}$,
so it suffices to bound the maximum degree
in the unused subgraph of $G \sm G_{\ex}$.
Given the trivial bounds whp $\DD(G_0) < 1.1p_0 n$
and $\DD(G'_i) < 1.1\pmax n$,
the following estimate implies
Lemma \ref{lem:leftover}.

\begin{lemma} \label{lem:Gleave}
whp the unused subgraph of each $G^{gg'}_{ii'}$
has maximum degree $< 5\eps^{.8} n$.
\end{lemma}

\begin{proof}
We fix $x \in V(G)$ and consider separately the 
contributions to the unused degree $u_x$ of $x$ 
from $N^\pm_{\ova{G}^{gg'}_{ii'}}(x)$. For indegrees,
let $f$ be the function on ${\cal H}_i$ 
defined by $f(\edge{w}{u}{x'}) =
1_{x'=x} 1_{u \in A^g_i} |N_<(u) \cap A^{g'}_{i'}|$.
We have $f({\cal H}_i,\oO') 
= \sum \{ \oO'({\cal H}_i[\lova{xy}]) 
 : y \in N^-_{\ova{G}^{gg'}_{ii'}}(x) \}
\ge (1-2\eps^{.8}) d^\pm_{\ova{G}^{gg'}_{ii'}}(x)$
by Lemma \ref{lem:Hwdeg}, so by Lemma \ref{lem:wEGJ} 
whp this contribution to $u_x$
is $< 2.1\eps^{.8} n$.

For outdegrees, first note that if $i'=0$
then for each $u \in A_0$ there is 
a unique $w \in W$ with $\edge{w}{u}{x}$,
for which we use $|N_>(u) \cap A^g_i|$
out-arcs at $x$. Thus we use exactly 
$F'[A_0,A^g_i] = n(p^g_{i0}-\pmax)$
out-arcs at $x$, so this contribution
is whp $<2\pmax n$.
Now for $i' \in [i-1]$, let $f_{i'}$
be the function on ${\cal H}_{i'}$ 
defined by $f_{i'}(\edge{w}{u}{x'})
= 1_{x'=x} 1_{u \in A^{g'}_{i'}} |N_>(u) \cap A^g_i|$.
By Lemma \ref{lem:Hwdeg} we have 
\[f_{i'}({\cal H}_{i'},\oO') = \sum_{u \in A^{g'}_{i'}} 
 |N_>(u) \cap A^g_i| \sum_{w \in W} \oO'(\edge{w}{u}{x})
 \ge (1-2\eps^{.8}) |F'[A^{g'}_{i'},A^g_i]|.\]
As $|F'[A^{g'}_{i'},A^g_i]| = n(p^{gg'}_{ii'}-\pmax)$,
this contribution is whp $<2.1\eps^{.8} n$.
\end{proof}
  
\subsection{Small stars}

Here we conclude the proof of Theorem \ref{main}
in Case S, where $P_{\ex}$ 
is a union of leaf stars in $T \sm T[A^*]$,
each of size $\le \LL = n^{1-c}$, with 
$|P_{\ex}| = p_{\ex} n = p_- n/2 \pm n^{1-c}$.
We start with some further properties of
the approximate decomposition needed in this case.

\begin{lemma} \label{lem:sprops} $ $
\begin{enumerate}
\item
For any $x \in V$ and $R \in \tbinom{V}{{\le}2}$
whp $\Ss := \sum_{w \in N_{J_{\ex}}(R)} 
 d_{P_{\ex}}(\phi_w^{-1}(x))
 = (1 \pm \eps) (p'_{\ex})^{|R|}|P_{\ex}|$.
\item
For any $y \in V$ and $w \in W$
whp $\sum_{x \in G_{\ex}(y)} 
 d_{P_{\ex}}(\phi_w^{-1}(x))
 = (1 \pm \eps) 2p_{\ex} |P_{\ex}|$.
\end{enumerate}
\end{lemma}
 
\begin{proof}
We prove (i) and omit the similar proof of (ii).
We consider the contribution to $\Ss$
from each $a \in V(F)$
according to its location in $T$.

For each $a$ in $A_0$
we define $M_a = \{ \phi_w(a) w: w \in W \}
{=}$MATCH$(B_a,Z_a)$. By Lemma \ref{lem:match},
for each $b \in N_<(a)$ and $w \in W$ we have
$\mb{P}^{t^-_b}(\phi_w(b) \in N_G(x)) = (1 \pm .1\xi')p$,
and if $\phi_w(b) \in N_G(x)$ for all $b \in N_<(a)$ then 
$\mb{P}^{t^-_a}(\phi_w(a)=x) = (1 \pm .1\xi')(p^{|N_<(a)|}n)^{-1}$.
By Lemma \ref{freedman} 
whp the contribution to $\Ss$ from $A_0$ 
is $\Ss[A_0] := 
\sum_{a \in A_0} \sum_{w \in N_{J_{\ex}}(R)} 
1_{\phi_w(a)=x}  d_{P_{\ex}}(a)
= (1 \pm .3\xi') |N_{J_{\ex}}(R)|
\sum_{a \in A_0} d_{P_{\ex}}(a)/n
= (1 \pm \xi') (p'_{\ex})^{|R|}
\sum_{a \in A_0} d_{P_{\ex}}(a)$.

Now we consider the contribution 
from $A_i$ with $i \in [i^*]$.
By the proof of Lemma \ref{lem:ux} whp
$\sum_{w \in J_{\ex}(R)} \oO'(\edge{w}{u}{x}) 
= \oO^*_{t_i}( P_{J_{\ex}(R)}: u \to x ) 
= (1 \pm \eps_i) (p'_{\ex})^{|R|}$
for each $i \in [i^*]$, $u \in A_i$.
The function $f$ on ${\cal H}_i$ 
defined by $f(\edge{w}{u}{x'})
= 1_{x'=x} 1_{w \in N_{J_{\ex}}(R)} d_{P_{\ex}}(u)$
has
$$
f({\cal H}_i,\oO') 
= \sum_{u \in A_i}  d_{P_{\ex}}(u) 
 \sum_{w \in N_{J_{\ex}}(R)} \oO'(\edge{w}{u}{x})
= (1 \pm \eps_i) (p'_{\ex})^{|R|}
\sum_{u \in A_i}  d_{P_{\ex}}(u).
$$
By Lemma \ref{lem:wEGJ} whp
the contribution to $\Ss$ from the
hypergraph matching embedding $A_i$ 
is $\Ss[A_i] := f(\mc{M}_i)=
\sum_{u \in A_i}  d_{P_{\ex}}(u)
 |\{w \in N_{J_{\ex}}(R): \edge{w}{u}{x} \in \mc{H}_i, 
u \notin A^0_{i,w}\}|
= (1 \pm 2\eps_i) (p'_{\ex})^{|R|}
\sum_{u \in A_i}  d_{P_{\ex}}(u)$.

It remains to consider the contribution 
from defining $\phi_w(a)$ for $w \in W_a$
by $\{ \phi_w(a) w: w \in W_a\}
{=}$MATCH$(B_a,Z_a)$,
where $Z_a = \{ \phi_w(b)w: b \in N_<(a) \}$
and $B_a \sub V_a \times W_a$
consists of all $vw$ with 
$v \in N_{J'_i}(w) \sm \im \phi_w$
and each $\phi_w(b) v$ for $b \in N_<(a)$ 
an unused edge of $G'_i$.
Here $V_a \in \tbinom{V}{|W_a|}$
is uniformly random, so
$\mb{P}(x \in V_a) = |W_a|/n$.
By Lemma~\ref{lem:Tleave}.ii,
$.3\eps_i n < |W_a| < 2.2\eps_i n$.

Similarly to the above
analysis of $A_0$, whp there are
$(1 \pm .1\xi')\pmax^{|N_<(a)|} |N_{J_{\ex}}(R)|$
choices of $w \in N_{J_{\ex}}(R)$ with
$\phi_w(b) \in N_{G'_i}(x)$ for all $b \in N_<(a)$,
and for each such $w$ we have
$\mb{P}^{t^-_a}(\phi_w(a)=x) 
= (1 \pm .1\xi')(\pmax^{|N_<(a)|}|W_a|)^{-1}$.
Thus the contribution
from defining $\phi_w(a)$ 
for $a \in A_i$, $w \in W_a$
is whp $\sum_{a \in A_i}  d_{P_{\ex}}(a)
 |\{w \in W_a \cap N_{J_{\ex}}(R) \}|/n
< 3.1\eps_i  \Ss[A_i]$.

Summing all contributions gives the stated estimate.
\end{proof}

In the subroutine SMALL STARS 
we start by finding an orientation $D$
of the leftover graph $G'_{\ex}$ such that 
each $d^+_D(x)=|L_x|$, where $L_x$ 
is the set of all $uw$ where $u$ 
is a leaf of a star in $P_{\ex}$
with centre $\phi_w^{-1}(x)$.
By the case $R=\es$ of Lemma \ref{lem:sprops}.i
whp all $|L_x| = (1 \pm \eps) |P_{\ex}|$. 
To construct $D$, we start 
with a uniformly random orientation of $G'_{\ex}$,
and while not all $d^+_D(x)=|L_x|$,
choose uniformly random $x,y,z$ 
with $|L_x|>d^+_D(x)$, $|L_y|<d^+_D(y)$,
$z \in N^+_D(y) \cap N^-_D(x)$ 
and reverse $\ova{yz}$, $\ova{zx}$.

To analyse this process, we first note that by 
typicality of $G'_{\ex}$ (Lemma~\ref{lem:leftover})
and a Chernoff bound, 
whp each $d^+_D(x) = (1 \pm 1.1p_0^{.7}) p_{\ex} n
= |L_x| \pm 2p_0^{.7} n$ and every
$|N^+_D(y) \cap N^-_D(x)| \ge p_{\ex}^2 n/2$.
Thus each vertex $v$ plays the role of $x$
or $y$ at most $2p_0^{.7}n$ times. We let ${\cal B}$ 
be the bad event that we reverse $>.2p_0^{.6}n$
arcs at any vertex $v$. We will show that
whp ${\cal B}$ does not occur. 
At any step before ${\cal B}$ occurs
where we consider $x$ and $y$ as above,
the number of choices for $z$ is whp 
$>.49p_{\ex}^2 n - p_0^{.6}n > .48p_{\ex}^2 n$.
Thus at any step $v$ plays the role of $z$
with probability $< 1/.48p_{\ex}^2 n$,
so the number of such steps 
is $(\mu,1)$-dominated with 
$\mu < 2p_0^{.7}n^2/.48p_{\ex}^2 n 
< .1p_0^{.6}n$. By Lemma \ref{freedman} we deduce
that whp ${\cal B}$ does not occur,
so we can construct $D$ with all $d^+_D(x)=|L_x|$.

Now for each $x \in V(G)$ in arbitrary order,
we define $\phi_w(u)$ for all $uw \in L_x$ by  
$M_x = \{ \{uw,\phi_w(u)\} : uw \in L_x \}
{=}$MATCH$(F_x,\es)$, where
$F_x \sub L_x \times N^+_D(x)$
consists of all $\{uw,y\}$ with $uw \in L_x$, 
$y \in N^+_D(x) \cap N_{J_{\ex}}(w) \sm \im \phi_w$.

To analyse this process, we consider
$Z_x = \{ \{uw,y\} \in  
L_x \times (N^+_D(x) \cap J_{\ex}(w) \cap \im \phi_w)\}$
and let ${\cal B}_x$ be the bad event that 
$Z_x$ has any vertex of degree 
$>.1p_{\ex}^{.9}|L_x|$.
Recall that $p_{\ex} \ll p'_{\ex} \ll 1$ in Case S
and at the beginning of SMALL STARS we have
$\im \phi_w\cap N_{J_{\ex}}(w) = \es$.

\begin{lemma}
whp under the construction of $D$,
if ${\cal B}_x$ does not occur
then $F_x$ is $p_{\ex}^{.02}$-super-regular
of density $(1 \pm p_{\ex}^{.8}) p'_{\ex}$.
\end{lemma}

\begin{proof}
Any $R \in \tbinom{N^+_D(x)}{{\le}2}$ has 
$|N_{F_x}(R)| = \sum_{w \in N_{J_{\ex}}(R)} 
d_{P_{\ex}}(\phi_w^{-1}(x))
\pm .1|R| p_{\ex}^{.9}|L_x| 
= (1 \pm .1p_{\ex}^{.8})(p'_{\ex})^{|R'|} |P_{\ex}|$ 
by Lemma \ref{lem:sprops}.
Any $R' \in \tbinom{L_x}{{\le}2}$ has $|N_{F_x}(R')| 
= |N^+_D(x) \cap \bigcap_{uw \in R'} J_{\ex}(w) |
 \pm .1|R'| p_{\ex}^{.9}|L_x|$,
which by Lemma \ref{lem:leftover} and a Chernoff
bound is whp 
$(1 \pm .1p_{\ex}^{.8})(p'_{\ex})^{|R'|} |P_{\ex}|$
unless $R'=\{uw,u'w\}$ for some $w$; there are 
$< \sum_{w \in W}  d_{P_{\ex}}(\phi_w^{-1}(x))^2
< n^{2-c}$ such $R'$. The lemma
now follows from Lemma \ref{lem:DLR}.
\end{proof}

By Lemma \ref{lem:match} we can choose
$M_x = \{ \{uw,\phi_w(u)\} : uw \in L_x \}
{=}$MATCH$(F_x,\es)$, and $\mb{P}(\phi_w(u)=y) = 
(1 \pm p_{\ex}^{.01}) ( p'_{\ex} |P_{\ex}|)^{-1}$
for all $\{uw,y\} \in F_x$.

It remains to show whp no ${\cal B}_x$ occurs.
We define a stopping time $\tau$ as the first $x$
for which ${\cal B}_x$ occurs and bound $\mb{P}(\tau=x)$.

First we bound $d_{Z_x}(uw)$ for $uw \in L_x$.
For any $y \in N^+_D(x)$, when processing 
any $x'$ before $x$ we defined
$\phi_w(u')$ for $d_{P_{\ex}}(\phi_w^{-1}(x'))$
leaves $u'$ of $\phi_w^{-1}(x')$,
each of which could be $y$ 
if $y \in N^+_D(x') \cap J_{\ex}(w)$,
with probability $<(.9p'_{\ex} |P_{\ex}|)^{-1}$.
Thus $d_{Z_x}(uw)$ is $(\mu,n^{1-c})$-dominated
with $\mu = \sum_{x'} 
|N^+_D(xx') \cap  J_{\ex}(w)|
d_{P_{\ex}}(\phi_w^{-1}(x'))
(.9p'_{\ex} |P_{\ex}|)^{-1}
< 1.2 p_{\ex}^2 n$, so by Lemma~\ref{freedman} whp
$d_{Z_x}(uw) < 3p_{\ex}|L_x|$.

Now we bound $d_{Z_x}(y)$ for $y \in N^+_D(x)$.
For any $uw \in L_x$ with $w \in J_{\ex}(y)$, 
when processing  any $x' \in N^-_D(y)$ before $x$,
we had $\phi_w(u)=y$ for some leaf $u$
with probability $<d_{P_{\ex}}(\phi_w^{-1}(x'))
(.9p'_{\ex} |P_{\ex}|)^{-1}$.
Thus $d_{Z_x}(y)$ is $(\mu,n^{1-c})$-dominated
with $\mu = (.9p'_{\ex} |P_{\ex}|)^{-1}
\sum_{uw \in L_x}  1_{w \in J_{\ex}(y)}
\sum_{x' \in N^-_D(y)} 
d_{P_{\ex}}(\phi_w^{-1}(x'))
< (.9p'_{\ex} |P_{\ex}|)^{-1}
\cdot |L_{\ex}|
\cdot (1 + \eps) p_{\ex} |P_{\ex}|$
by Lemma \ref{lem:sprops}.ii,
so whp $d_{Z_x}(w) < 3|L_x|p_{\ex}/p'_{\ex}$.

Thus whp no ${\cal B}_x$ occurs, as required.

\subsection{Paths}

Here we conclude the proof of Theorem \ref{main} in Case P, 
where $P_{\ex}$ is the vertex-disjoint union 
of two leaf edges in $T \sm T[A^*]$ and 
$p_+ n/101K$ bare $8K$-paths in $T \sm T[A^*]$.

The first phase of the PATHS subroutine
fixes parity, as follows.
We call $x \in V(G)$ \emph{odd} if the parity 
of $d_{G'_{\ex}}(x)$ differs from that
of the number of $w$ such that $x=\phi_w(a)$
where $a$ is the end of a bare path in $P_{\ex}$.
We let $X$ be the set of odd vertices and
$a_1 \ell_1$, $a_2 \ell_2$ be the leaf edges
in $P_{\ex}$, with leaves $\ell_1$, $\ell_2$.

First we define all $\phi_w(\ell_1)$ by
$M_1 = \{ \phi_w(\ell_1) w: w \in W \}
{=}$MATCH$(B_1,Z_1)$, where 
$Z_1 = \{ \phi_w(a_1)w \}_{w \in W}$
and $B_1 = \{ vw: v \in N_{J_{\ex}}(w),
 v\phi_w(\ell_1) \in G_{\free} \}$.
Lemma \ref{lem:match} applies,
as $Z_1$ is a matching and similarly
to the proof of Lemma \ref{lem:rga2} whp
$B_1$ is $p_0^{.5}$-super-regular with density
$(1 \pm p_0^{.5}) p'_{\ex} p_{\ex}$.
Similarly, Lemma \ref{lem:match} applies
to justify the definition of
$\phi_w(\ell_2)$ for $w \in W'$ by
$M'_2 = \{ \phi_w(\ell_2) w: w \in W' \}
{=}$MATCH$(B'_2,Z'_2)$
and $\phi_w(\ell_2)$ for $w \in W \sm W'$ 
by $M_2 = \{ \phi_w(\ell_2) w: w \in W \sm W' \}
{=}$MATCH$(B_2,Z_2)$. By construction,
there are no odd vertices after
the embeddings of $\ell_1$ and $\ell_2$.

Next for each $w \in W$ 
we need $8d(x,y)$-paths $P^{xy}_w$ in $P_{\ex}$
for each $[x,y] \in {\cal Y}_w$ centred in vertex-disjoint
bare $(8d(x,y)+2)$-paths in $P_{\ex}$. 
We greedily choose these paths within the bare $8K$-paths 
in $P_{\ex}$ that exist by definition of Case P.
By Lemma \ref{lem:INT}, the total number of vertices
required by these paths is
$ \sum \{ 8d(x,y) + 2: [x,y] \in {\cal Y}_w \}
 = 8|Y_w| + |{\cal Y}_w| 
 = (1-\eta)|P_{\ex}| \pm nd^{-.9}$.
At most $d^{-.9}|P_{\ex}|$ vertices of the bare $8K$-paths
cannot be used due to rounding errors, so as $d^{-1} \ll \eta$
the algorithm to choose all $P^{xy}_w$ can be completed. 

Now we extend each $\phi_w$ to an embedding
of $P_{\ex} \sm \bigcup_{xy} P^{xy}_w$
so that $\phi_w^{-1}(x)$, $\phi_w^{-1}(y^+)$
are the ends of $P^{xy}_w$,
according to a random greedy algorithm,
where in each step, in any order, 
we define some $\phi_w(a)=z$, 
uniformly at random with 
$z \in J_{\ex}(w) \sm \im \phi_w$
and $zz' \in G_{\free}$ whenever
$z'=\phi_w(b)$ with $b \in N_T(a)$.
Writing $E$ for the set of ends of paths 
in $P_{\ex}$, for any vertex $y$ we use
$|\{w: \phi^{-1}(y) \in E\}| < 1.1|E| < |P_{\ex}|/3K$
edges at $y$ due to it playing the role of an end.

Let $X_y$ be the number of additional edges used
at $y$ during the random greedy algorithm,
and let ${\cal B}$ be the bad event
that any $X_y > .1\eta^{.9} n$.
We claim whp ${\cal B}$ does not occur.
To see this, consider any step before ${\cal B}$ 
occurs, and suppose we are defining $\phi_w(a)$.
Let $R$ be the set of $b \in N_T(a)$
such that $\phi_w(b)$ has been defined
and note that $|R| \le 2$.
By Lemma \ref{lem:leftover} there are
$((1 \pm p_0^{.9})p_{\ex})^{|R|} p'_{\ex} n$ 
choices of $z \in J_{\ex}(w) \cap N_{G'_{\ex}}(R)$,
of which we forbid $< 2\eta p_{\ex} n$ in $\im \phi_w$
and $<\eta^{.9} n$ if ${\cal B}$ has not occurred.
As $\eta = \eta_+ \ll p_+$ 
(and $p_+ \le 13p_{\ex}$) we can choose $z$,
and any $z$ is chosen
with probability $<(.9p_{\ex}^2 p'_{\ex} n)^{-1}$.
Thus $X_y$ is $(\mu,2)$-dominated with 
$\mu = (.9p_{\ex}^2 p'_{\ex} n)^{-1}
\sum_{w \in J_{\ex}(y)} 
|P_{\ex} \sm \bigcup_{xy} P^{xy}_w|
< 1.2 \eta p_{\ex}^{-1} n$,
so by Lemma \ref{freedman} whp $<.1\eta^{.9} n$,
which proves the claim.

Thus the random greedy algorithm can be completed,
and the remaining graph $G_{\free}$
is an $\eta^{.9}$-perturbation of $G_{\ex}$,
i.e.\ $|G_{\ex}(x) \sd G_{\free}(x)| 
< \eta^{.9} n$ for any $x \in V$.
By construction, every $d_{G_{\free}}(x)$
is even.
The following lemma will complete the proof 
of Theorem \ref{main} in Case P.

\begin{lemma} \label{lem:path}
One can decompose $G_{\free}$
into $(G_w: w \in W)$ such that each $G_w$
is a vertex-disjoint union of $8d(x,y)$-paths
$\phi_w(P^{xy}_w)$ between $x$ and $y^+$
for $[x,y] \in {\cal Y}_w$,
internally disjoint from $\im \phi_w$.
\end{lemma}

The proof of Lemma \ref{lem:path} is similar to
the corresponding arguments in \cite{factors},
so we will be brief and give more details
only where there are significant differences.
We require the following result on wheel decompositions; 
see \cite{factors} for its derivation from \cite{K2}
and discussion of how it provides the required paths.
The statement requires a few definitions.
An $8$-wheel consists 
of a directed $8$-cycle (called the rim),
another vertex (called the hub),
and an arc from each rim vertex to the hub.
We obtain the special $8$-wheel $\wK$ 
by giving all arcs colour $0$ except 
that one rim edge $\ova{xy}$ 
and one spoke $\ova{yw}$ have colour $K$.

\begin{theo} \label{decompK}
Let $n^{-1} \ll \dD \ll \oO \ll 1$, $s=2^{50\cdot 8^3}$ and $d \ll n$.
Let $J = J^0 \cup J^K$ be a digraph with arcs coloured
$0$ or $K$, with $V(J)$ partitioned as $(V,W)$
where $\oO n \le |V|, |W| \le n$, such that
all arcs in $J[V,W]$ point towards $W$ and $J[W]=\es$.
Then $J$ has a $\wK$-decomposition such that
every hub lies in $W$ if the following hold:

{\em Divisibility:}
$d^-_J(w) = 8d^-_{J^K}(w)$ for all $w \in W$,
and for all $v \in V$ we have
$d_J^-(v,V)=d_J^+(v,V)=d_J^+(v,W)$
and $d^-_{J^K}(v,V)=d^+_{J^K}(v,W)$.

{\em Regularity:}
each $3d$-separated copy of $\wK$ in $J$ has a weight 
in $[\oO n^{-7}, \oO^{-1} n^{-7}]$ such that
for any arc $\ova{e}$ there is total weight $1 \pm \dD$ 
on wheels containing $\ova{e}$.

{\em Extendability:}
for all disjoint $A,B \sub V$ and $L \sub W$
each of size $\le s$,
for any $a, b, \ell \in \{0,K\}$ we have
$|N^+_{J^a}(A) \cap N^-_{J^b}(B) 
\cap N^-_{J^\ell}(L)| \ge \oO n$,
and furthermore, if $(A,B)$ is $3d$-separated then
$|N^+_{J^0}(A) \cap N^+_{J^K}(B) \cap W| \ge \oO n$.
\end{theo}

\begin{proof}[Proof of Lemma \ref{lem:path}]
Recall that we constructed $J_{\ex}$ in DIGRAPH,
such that for every $xy \in G_{\ex}$, we have exactly one of
$\ova{xy} \in J_{\ex}^0$,
$\ova{yx} \in J_{\ex}^0$,
$\ova{xy}^- \in J_{\ex}^K$,
$\ova{yx}^- \in J_{\ex}^K$,
and there are also $\ova{yw} \in J_{\ex}^0[V,W]$.
Add the arcs $J^K_{\ex}[V,W]=\{\ova{xw}: x \in Y_w\}$.
It suffices to find an $\eta^{.6}$-perturbation $L$ of $J_{\ex}$,
i.e.\ $L$ is obtained from $J_{\ex}$ by adding, deleting
or recolouring at most $\eta^{.6} n$ arcs at each vertex,
where $L[V]$ corresponds to $G_{\free}$ under twisting, 
and each $N^-_L(w) \subseteq V \sm \im \phi_w$, and
a set $E$ of edge-disjoint copies of $\wK$ in $L$,
such that Theorem \ref{decompK} applies to give 
a $\wK$-decomposition of $L' := L \sm \bigcup E$.
This will suffice, by taking each $G_w$ to consist of the
union of the $8$-paths that correspond under twisting to
the rim $8$-cycles of the copies of $\wK$ containing $w$.
Here an arc of $L$ corresponds to an edge $xy \in G_{\free}$
under twisting if it is $\ova{xy} \in L^0$
or $\ova{yx} \in L^0$ or $\ova{xy}^- \in L^K$ 
or $\ova{yx}^- \in L^K$
(which is a more flexible notion than in \cite{factors},
as it does not depend on the orientation of $L$.)

Whenever we make a series 
of $\gamma n^2$ modifications to $J_{\ex}$ of some type
which involves changing edges at some intermediate vertex $z$,
we always ensure that no vertex plays the role of $z$
more than $\gamma\eta^{-.1}n$ times.
There will always be more than, say, $2\eta^{.1}n$
valid choices of $z$, 
by Lemmas \ref{lem:INT} and \ref{lem:DI},
and thus we can avoid the set
of at most $\eta^{.1}n$ overused vertices.
This series of modifications
will add $\gamma \eta^{-.1}$ to the perturbation constant.

We start by deleting arcs  
corresponding to $G_{\ex}\sm G_{\free}$,
adding arcs $\ova{xy}$ for each $xy \in G_{\free}\sm G_{\ex}$,
replacing any $\ova{xy}$ of colour $K$ 
where $d(x,y)<3d$ with $\ova{xy}^+$ of colour $0$
and deleting arcs $\ova{yw}$ in $L[V,W]$ 
with $y \in N^-_{J_{\ex}}(w) \cap \im \phi_w$.
Next we delete or add arbitrary arcs $\ova{yw}$ with
$y \in V \sm (\im \phi_w \cup (Y_w)^+ \cup  N^-_{J_{\ex}}(w))$
until each $d^-_L(w)=8|Y_w|$, and so $|L[V,W]|=|L[V]|$.
We require $<\eta^{.9}n$ such arcs for each $w$,
by Lemma \ref{lem:INT} and the bounds on $X_y$
during the embedding of
 $P_{\ex} \sm \bigcup_{xy} P^{xy}_w$.

While $|L^0[V]|>|L^0[V,W]|$ we
replace some $\ova{xy} \in L^0[V]$ by $\ova{xy}^- \in L^K[V]$,
or while $|L^0[V]|<|L^0[V,W]|$ we replace 
some $\ova{xy} \in L^K[V]$ by $\ova{xy}^+ \in L^0[V]$,
continuing until $|L^0[V,W]|=|L^0[V]|$,
and hence $|L^K[V,W]|=|L^K[V]|$.

Next we balance degrees in $L^K$.
While there are $x,y$ in $V$ with
$d^-_{L^K}(x,V) > d^+_{L^K}(x,W)$ and
$d^-_{L^K}(y,V) < d^+_{L^K}(y,W)$,
we choose $z \in V$ such that
$\ova{zx} \in L^K$, $\ova{zy}^+ \in L^0$
and replace these arcs by 
$\ova{zx}^+ \in L^0$, $\ova{zy} \in L^K$.
While there are $x,y$ in $V$ with
$d^+_{L^K}(x,V) > d^+_{L^K}(x,W)$ and
$d^+_{L^K}(y,V) < d^+_{L^K}(y,W)$,
we choose $z \in V$ such that
$\ova{xz} \in L^K$, $\ova{yz}^+ \in L^0$
and replace these arcs by 
$\ova{yz} \in L^K$, $\ova{xz}^+ \in L^0$.
We continue until every $d^+_{L^K}(v,V) 
= d^-_{L^K}(v,V) = d^+_{L^K}(v,W)$.

Now we require some new modifications
which do not appear in \cite{factors}.
We start by noting that each $d_L(x,V)$ is even.
To see this, note that
as $L[V]$ corresponds to $G_{\free}$ under twisting we have
$d_{L}(x,V)=d_{G_{\free}}(x)+d^-_{L^K}(x^-,V)-d^-_{L^K}(x,V)
=d_{G_{\free}}(x)+d^+_{L^K}(x^-,W)-d^+_{L^K}(x,W)=d_{G_{\free}}(x)$
where the last equality follows from interval properties
(listed before the definition of INTERVALS).
While there are $x,y \in V$ with $d_L(x,V) < 2d^+_L(x,W)$ 
and $d_L(y,V) > 2d^+_L(y,W)$, we add $\ova{yw}$ to $L^0$ 
and remove $\ova{xw}$ from $L^0$
for some $w \in N^+_{L^0}(x,W)\sm N^+_{L}(y,W)$
with $y \notin \im \phi_w$.
We continue until every $d_L(v,V)=2d^+_L(v,W)$.

While there are $x,y$ in $V$ with
$d^+_{L^0}(x,V) > d^+_{L^0}(x,W)$ and
$d^+_{L^0}(y,V) < d^+_{L^0}(y,W)$,
we choose $z \in V$ such that
$\ova{xz} \in L^0$, $\ova{yz} \in L^0$
and replace these arcs by 
$\ova{zx} \in L^0$, $\ova{zy} \in L^0$.
Now every $d^+_L(v,V)=d^+_L(v,W)$.
Thus $L$ satisfies the required divisibility conditions,
and is an $\eta^{.6}$-perturbation of $J_{\ex}$,
and $L[V]$ corresponds to $G_{\free}$ under twisting.
It remains to satisfy the extendability and regularity
conditions of Theorem \ref{decompK}. A summary of the
argument is as follows (we omit the details as
they are very similar to those in \cite{factors}).
There are many wheels on each arc,
so we can greedily cover all $\ova{xy} \in L[V]$ with edge-disjoint
wheels, incurring an insignificant perturbation of $L$.
A stronger version of the extendability hypothesis 
with $J_{\ex}$ in place of $L$
holds by Lemmas \ref{lem:INT} and \ref{lem:DI},
and so it holds for the perturbation $L$.
By typicality, the regularity condition is satisfied 
by assigning the same weight $\hat{W}$ to every wheel,
choosing $\hat{W}$ so that any arc is in
$\approx \hat{W}^{-1}$ wheels.
\end{proof}

\subsection{Large stars}

Here we conclude the proof of Theorem \ref{main}
in Case L,
where all but at most $p_+ n$ vertices of $T$ 
belong to leaf stars of size $\geL = n^{1-c}$.
The argument is self-contained:
there is no approximate step,
and the entire embedding is achieved
by the subroutine LARGE STARS.

We start by letting
${\cal S}$ be the union of all maximal
leaf stars in $T$ that have size $\geL$.
We let $F = T \sm {\cal S}$;
by assumption $|V(F)| \le p_+ n$.
We let $S^+ = \{v \in V(T): d_T(v) \geL\}$,
so that $|S^+| < 2\DD$ and
$S \sub S^+ \sub V(F)$, where
$S$ is the set of star centres of ${\cal S}$.

We partition $W$ as $W_1 \cup W_2 \cup W_3$
with each $||W_i|-n/3|<1$. For each $v \in V(G)$, 
we independently choose at most one of 
$\mb{P}(v \in U^a_i) = d_{\cal S}(a)/3|{\cal S}|$ 
with $a \in S$, $i \in [3]$. 
By Chernoff bounds, whp each
$|U^a_i| = nd_{\cal S}(a)/3|{\cal S}| \pm n^{.9}$.
We let $U_i = \bigcup_a U^a_i$.
While $\sum_{i=1}^3 ||W_i|-|U_i||>0$, 
we relocate a vertex so as to decrease this sum,
thus relocating $<n^{.9}$ to or from any $U^i_a$,
so $<3\DD n^{.9} < n^{.99}$ in total.

Noting that $F$ is a tree,
we can fix an order $\prec$ on $V(F)$ such that
$N_<(u) = \{v \prec u: vu \in F\} = \{u^-\}$
has size $1$ for all $u \ne u_0 \in V(F)$.
We fix distinct $\phi_w(u_0)$, $w \in W$
with $\phi_w(u_0) \in U_i$ whenever $w \in W_i$.
We construct edge-disjoint copies $F_w$ of $F$ 
by considering $a \in F_w$ in $\prec$ order,
defining all $\phi_w(a)$ by
$M^a_i = \{ \phi_w(a)w: w \in W_i\}
{=}$MATCH$(B^a_i,Z^a_i)$, $i \in [3]$,
and updating 
\begin{align*}
& G_{\free} = \{\text{unused edges}\},
\quad Z = \{ vw: v \in V(F_w) \}, \text{ and } \\
& J = \{ \ova{xx}': x=\phi_w(a), 
x' \in Z(w) \cap U^a  \}_{w \in W, a \in S}.
\end{align*}
By construction $G \sm G_{\free}$ and $Z$ both 
have maximum degree $\le |V(F)| \le p_+ n$.

\begin{lemma}
Every edge is used at most once
and $J$ has no $2$-cycles.
\end{lemma}

\begin{proof}
First note that as each
$B^a_i(w) \sub G_{\free}(\phi_w(a^-)) 
\sm Z(w)$ we embed each $\phi_w(a)$
to a vertex not yet used by $F_w$ so that 
$\phi_w(a^-)\phi_w(a)$ is an unused edge.
Furthermore, when $a \in S$, for each $\lova{xx}' \in J$,
by excluding $\phi_w(a) \in N^+_J(Z(w) \cap U^a)$ 
we do not add $\ova{xx}'$ to $J$ 
due to $x=\phi_w(a)$, $x' \in Z(w) \cap U^a$,
and by excluding $N^-_J(\phi_w(b)) \cap U^b$
where $x' \in U^b$ we do not add $\ova{xx}'$ to $J$ 
due to $x=\phi_w(b)$, 
$x' = \phi_w(a) \in Z(w) \cap U^b$.
As before, by including all $\phi_w(a^-)w$ in $Z^a_i$
we ensure that $M^a_i$ does not require
the same edge of $G_{\free}$ twice.
Furthermore, when $a \in S$, by including all
$vw$ with $v \in U^a \cap Z(w)$
we ensure that $M^a_i$ does not add
both arcs of any $2$-cycle to $J$:
we cannot add $\ova{xx}'$, $\lova{xx}'$
with $x'=\phi_{w'}(a) \in U^a \cap Z(w)$
and $x=\phi_w(a) \in U^a \cap Z(w')$
as $xwx'w'$ would be an $M^a_i Z^a_i M^a_i Z^a_i$.
The lemma follows.
\end{proof}

Next we note for all $i \in [3]$, $w \in W_i$
that $B^a_i(w) \sub U_{i'}$, 
where $i' = i-1_{a \notin S^+}$.
We record some simple consequences of this observation.
\begin{itemize}
\item
$Z(w) \cap U_{i+1} = \es$ for any $w \in W_i$.
\item $Z(x) \cap W_{i-1} 
= N^-_J(x) \cap U_{i-1}
= N^+_J(x) \cap U_{i+1} = \es$ for any $x \in U_i$.
\item
If $w \in W_i$  then $Z(w) \cap U_i$
only contains $\phi_w(a)$ with $a \in S^+$,
so has size $\le |S^+| \le 2\DD$.
\item
If $x \in U_i$, $b \in S$ then
$N^+_J(x) \cap U^b_i = Z(M^b_i(x)) \cap U^b_i$
has size $\le 2\DD$, as $M^b_i(x) \in W_i$.
\item
If $x \in U^b_i$ then
$Z(x) \cap W_i$ only contains 
$M^a_i(x)$ with $a \in S^+$, 
so $|Z(x) \cap W_i| \le 2\DD$.
\item
If $x \in U^b_i$ then
$N^-_J(x) \cap U_i = M^b_i(Z(x) \cap W_i)$
has size $\le 2\DD$.
\item
Each $Z^a_i$ has maximum degree $\le 2\DD$.
\end{itemize} 

By construction, $B^a_i$ is a balanced bipartite graph.
To justify the application of Lemma \ref{lem:match}
in choosing $M^a_i$ it remains to establish the following.

\begin{lemma}
$B^a_i$ is $p_+^{.15}$-super-regular.
\end{lemma}

\begin{proof}
We first consider $\GG^a_i \sub U_{i'} \times W_i$ 
with $i' = i-1_{a \notin S^+}$
defined by $N_{\GG^a_i}(w) 
= U_{i'} \cap G(\phi_w(a^-))$.
For any $R \in \tbinom{U_{i'}}{{\le 2}}$
we have $N_{\GG^a_i}(R) = \{w \in W_i: 
   R \sub G(\phi_w(a^-)) \}
   = M^{a^-}_i( N_G(R) \cap U_{i'} )$,
so $|N_{\GG^a_i}(R)| = |N_G(R) \cap U_{i'}|
= ((1 \pm 1.1\xi)p)^{|R|}|W_i|$
whp by typicality and a Chernoff bound. 
Similarly, 
$N_{\GG^a_i}(R') = U_{i'} \cap 
 \bigcap_{w \in R'} G(\phi_w(a^-))$
for $R' \in \tbinom{W_i}{{\le 2}}$
whp has size $((1 \pm 1.1\xi)p)^{|R'|}|U_{i'}|$.
 
Now we will show that $\GG^a_i \sm B^a_i$
has maximum degree $\le 5p_+ n$. 
To see this, we first note that we have
a contribution $\le 4|V(F)| \le 4p_+ n$
to any degree in $\GG^a_i \sm B^a_i$
due to edges in $Z$ or $G \sm G_{\free}$
(including the $\le 1$ vertex
that is the image of $a \notin S^+$ for two $w,w'$).
There are no other contributions for $a \notin S^+$,
so we consider $a \in S^+$ and so $i'=i$.
First we estimate the contribution to
degrees of $w \in W_i$ and to degrees
of $x \in U_i$ due to $x \in N^-_J(\phi_w(b) \cap U^b_i)$
for $b \in S$,
which we claim are both $\le 4\DD^2$.
Indeed, for $w \in W_i$ the
contribution 
is $\le \sum_{b \in S} |N^-_J(\phi_w(b)) \cap U_i|
   \le 2|S|\DD \le 4\DD^2$.
For $x \in U_i$, 
we count $w \in W_i$  
if $\phi_w(b) = y \in N^+_J(x)$,
where $y \in U_i$ as $w \in W_i, b \in S$,
so this contribution is
$\le \sum_b |N^+_J(x) \cap U^b_i| \le 4\DD^2$. 

It remains to estimate the contribution to
degrees of $w \in W_i$ and to degrees
of $x \in U_i$ due to $x \in N^+_J(Z(w) \cap U^a)$,
which we claim are both $\le 8\DD^3$. 
To see this, first note that
we must have $w \in Z(y)$ 
for some $y \in N^-_J(x) \cap U^a$, and 
$x \in Z(w') \cap U^b_i$ for some $b$
with $\phi_{w'}(b) = y \in Z(w)$.
We note that $w' \in W_i$, as otherwise
$x \in Z(w')$ implies $w' \in W_{i+1}$
and $y=\phi_{w'}(b)$ implies $y \in U_{i+1}$,
which contradicts $y \in Z(w)$.
Thus we have $\le 2\DD$ choices 
for each of $w' \in Z(x) \cap W_i$,
then $y \in Z(w') \cap U_i$,
then $w \in Z(y) \cap W_i$,
which proves the claim.
The lemma now follows from Lemma \ref{lem:DLR}.
\end{proof}

Thus we can apply Lemma \ref{lem:match}, 
so each $M^a_i = \{ \phi_w(a)w: w \in W_i\}
{=}$MATCH$(B^a_i,Z^a_i)$ can be chosen and has 
$\mb{P}(vw \in M^a_i) = (1 \pm p_+^{.1})(pn)^{-1}$
for all $vw \in B^a_i$. 
In particular, we can complete step (iv), 
thus choosing edge-disjoint copies $F_w$ of $F$.

\begin{lemma} \label{lem:init}
For $x \in V$, $w \in W$, $a \in S$ 
whp $|U^a \cap V(F_w)| < 1.1p_+ |U^a|$ and
$|N^\pm_J(x) \cap U^a| < .1p_+^{.9} |U^a|$.
\end{lemma}

\begin{proof}
The first statement holds by Lemma \ref{freedman},
as $|U^a \cap V(F_w)|$ is $(\mu,1)$-dominated 
with \[ \mu = (1 \pm p_+^{.1})(pn)^{-1}
\sum_{u \in V(F)} |B^u_i(w) \cap U^a|
= (1 \pm 1.1p_+^{.1}) |V(F)| n^{-1} |U^a|.\]
Next recall for $x \in U_i$ 
that $N^-_J(x) \cap U_{i-1}
= N^+_J(x) \cap U_{i+1} = \es$,
$|N^-_J(x) \cap U_i| \le 2\DD$
and $|N^+_J(x) \cap U_i|
= \sum_b |N^+_J(x) \cap U^b_i| \le 4\DD^2$.

To bound $|N^+_J(x) \cap U^a_{i-1}|$, note that
for any $w \in W$ there are $<1.1p_+ |U^a|$ 
choices of $x' \in U^a \cap V(F_w)$,
for which we add $\ova{xx}'$ to $J$ 
if $M^a_i$ chooses $xw$.
Thus $|N^+_J(x) \cap U^a_{i-1}|$
is $(\mu,1)$-dominated with 
$\mu < n \cdot 1.1p_+ |U^a| \cdot 
(1 \pm p_+^{.1})(pn)^{-1}$,
so by Lemma \ref{freedman} whp $< .01p_+^{.9} |U^a|$. 

Finally, for $x \in U^b_i$ we have
$N^-_J(x) \cap U^a_{i+1} = M^b(Z(x)) \cap U^a_{i+1}$,
which by Lemma \ref{lem:match} whp has size 
$|M^b_{i+1}[Z(x) \cap W_{i+1},U^a_{i+1}]|
< |Z(x)||U^a_{i+1}|/.99p|W_{i+1}| + n^{.8}
< .01p_+^{.9} |U^a|$.
\end{proof}

We deduce $|N^\pm_J(x)| < .1p_+^{.9} n$,
so the underlying graph $\wt{J}$ of $J$  
has maximum degree $< .2p_+^{.9} n$.

In step (v) we orient $G_{\free}$ as
$D = \bigcup_{w \in W} D_w$,
where for each $xy \in G_{\free}$
with $x \in U^a$ and $y \in U^b$,
if $\ova{xy} \in J$ we have
$\ova{yx} \in D_w$ where $\phi_w(a)=y$,
if $\ova{yx} \in J$ we have
$\ova{xy} \in D_w$ where $\phi_w(b)=x$,
or otherwise we make one of these choices
independently with probability $1/2$.
We define $Z^+ \sub V \times W$ by 
$Z^+(w) = V(F_w) \cup V(D_w)$.

\begin{lemma}
whp $d^+_{D_w}(x)$ and $|Z^+(w) \cap U^a|$
are $(1 \pm p_+^{.8}) d_{\cal S}(a)$
for all $x=\phi_w(a)$, $w \in W$, $a \in S$.
\end{lemma}

\begin{proof}
First note by typicality and Chernoff bounds
that whp there are $(1+2\xi)nd_{\mc{S}}(a)p/|\mc{S}|
\pm 1.1p_+|U^a|=(1 \pm p_+^{.85}) 2d_{\cal S}(a)$
choices of $v \in U^a \cap G_{\free}(x)$
after step (iv).
Excluding $< .2p_+^{.9} |U^a|$ 
choices with $xv \in \wt{J}$,
for all other $v$ independently
$\mb{P}(\ova{xv} \in D_w) = 1/2$.
The lemma follows by a Chernoff bound
and Lemma \ref{lem:init}.
\end{proof}

To analyse  step (vi), we first observe that 
initially the sets $N^+_{D_w}(\phi_w(a))$ 
are disjoint over $a \in S$ and disjoint from $V(F_w)$,
and this is preserved by each move;
moreover, each move decreases $\Ss$ by $2$,
and if (vi) does not abort we have every $D_w \cup F_w = T$.
So it suffices to show that (vi) does not abort.
We start with an estimate
for the number of moves for any $uwu'w'$ that
are \emph{original}, meaning that they are present
at the end of step (v) before any arcs are moved.

\begin{lemma}
Any $u = \phi_w(a)$, $u' = \phi_{w'}(a')$ 
whp have $> 9000^{-1} p^3 n^2 d_{\cal S}(a')$
original $uwu'w'$-moves.
\end{lemma}

\begin{proof}
We estimate the number of moves by sequentially
choosing $x$, $v$ then $z$.
Suppose $u \in U^b$ and $w' \in W_i$, 
so $u'=\phi_{w'}(a') \in U_i$.
Suppose $u' \in U^{b'}$.
We claim there are whp $>.08pn$ choices 
of $x \in U_i \cap N^-_D(u) \sm Z^+(w)$.
To see this, note that there are
$(1 \pm p_+^{.8}) pn/3$ choices 
of $x \in U_i \cap G_{\free}(u)$.
Excluding $< p_+^{.9} n$  
with $xu$ or $x\phi_w(b)$ in $\wt{J}$,
for all  others independently
$\mb{P}(x \in N^-_D(u) \sm Z^+(w)) \ge 1/4$,
so the claim holds by a Chernoff bound.
Consider any such $x$, say with $x \in U^c_i$,
and let $w^x=M^b(x) \in W_i$.

We claim there are whp $> .01p^2 n$ choices of
 $v \in U_{i-1} \cap N^-_D(x) \cap N^+_D(u') 
 \sm (Z^+(w') \cup Z^+(w^x) \cup M^c(Z^+(u')))$.
To see this, note that there are 
$(1 \pm p_+^{.8}) p^2 n/3$ choices 
of $v \in U_{i-1} \cap G_{\free}(x) \cap G_{\free}(u')$.
For any such $v$, say in $U^{d}_{i-1}$,
we have $v' := \phi_{w'}(d) \in U_i$, 
so $\ova{vv}' \notin J$, and so 
$\mb{P}(\lova{vv}' \in D) \le 1/2$.
Similarly, $v^x := \phi_{w^x}(d)$ has
$\mb{P}(\lova{vv}^x \in D) \le 1/2$ independently.
Also $v \in M^c(Z^+(u') \sm Z(u')) 
 \Lra w^v := M^c(v) \in Z^+(u') \sm Z(u')
 \Lra \ova{y_vu}' \in D$, where 
$y_v = \phi_{w^v}(b') \in U_{i-1}$ as $w^v \in W_{i-1}$.
Since each $v$ corresponds to a unique $y_v$,
the number of $y_v$ with $\lova{y_vu}' \in J$ is
$d^+_J(u') \le .1p_+^{.9}n$.
Excluding such $v$, for all others we have
$\lova{y_vu}' \notin J$ and so
$\mb{P}(v \in M^c(Z^+(u') \sm Z(u'))) \le 1/2$ 
independently. Excluding a further $< p_+^{.9} n$  
choices of $v$ in 
$Z(w') \cup Z(w^x) \cup M^c(Z(u'))$
or with $vx$ or $vu'$ in $\wt{J}$,
any other $v$ contributes independently
with probability $\ge 2^{-5}$,
so the claim follows by a Chernoff bound.

Fix any such $v$, say with $v \in U^d$,
and let $w^{u'} = M^d(u') \in W_{i-1}$.
Similarly to the above, there are
whp $> .16d_{\cal S}(a')$ choices of
$z \in U_{i-1} \cap N^+_{D_{w'}}(u') \sm Z^+(w^{u'})$,
as there are $(1 \pm p_+^{.8}) d_{\cal S}(a')/3$
choices  in $U^{a'}_{i-1} \cap G_{\free}(u')$
and letting $z' = \phi_{w^{u'}}(a') \in U_i$,
excluding at most $.1p_+^{.9}|U^{a'}|$
vertices $z$ such that $\lova{zz}' \in J$,
each other $z$ has $\mb{P}(\lova{zz}' \in D) \le 1/2$.
The lemma follows.
\end{proof}
 
After $t$ moves, we let ${\cal B}_t$
denote the bad event that any vertex $y$ 
is incident to $>p_+^{.7} n$ moved arcs
or to $>p_+^{.7}|U^q|$ arcs $\ova{yy}'$
with $y' \in U^q$ for some $q$.
We let $\tau$ be the smallest $t$ such that
${\cal B}_t$ occurs, or $\infty$ if there is no such $t$.
At any step $t<\tau$ requiring a move for some 
$uwu'w'$ with $\phi_{w'}(a')=u'$, as ${\cal B}_t$ does not hold
there are $>10^{-4} p^3 n^2 d_{\cal S}(a')$ moves.
To complete the proof it therefore suffices
to show whp $\tau=\infty$.
We fix $t$ and bound $\mb{P}(\tau=t)$ as follows. 

We start by showing that whp 
$< p_+^{.7} n$ arcs are moved at any $y$.
To see this, note first that the number of times 
$y$ plays the role of $u$ or $u'$ in a move is 
$\sum_{w \in W} |d^+_{D_w}(y) - d_{\cal S}(\phi_w^{-1}(y))| 
 < \sum_a p_+^{.8} d_{\cal S}(a) < p_+^{.8} n$.
Now fix $uwu'w'$ with $\phi_{w'}(a')=u'$.
Then $y$ plays the role of $x$ or $v$
in $< n d_{\cal S}(a')$ moves,
so with probability $< 10^4 p^{-3} n^{-1}$.
The number of moves where $y$ plays $x$ or $v$
is therefore $(\mu,1)$-dominated, where
$\mu < p_+^{.8} n^2 \cdot 10^4 p^{-3} n^{-1}
 < .1p_+^{.7} n$, so by Lemma \ref{freedman}
 whp $< .2p_+^{.7} n$.
Furthermore,
$y \in U^{a'}$ plays the role of $z$ 
in $<n^2$ moves, 
so with probability 
$< 10^4 p^{-3} d_{\cal S}(a')^{-1}$.
The number of such moves
is therefore $(\mu,1)$-dominated, where
$\mu < 10^4 p^{-3} d_{\cal S}(a')^{-1}
 \sum_w |d^+_{D_w}(\phi_w(a')) - d_{\cal S}(a')| 
 < .1p_+^{.7} n$, so by 
 Lemma \ref{freedman} whp $< .2p_+^{.7} n$.
The claim follows.

Now, given $uwu'w'$, any arc $\ova{yy}'$ with $y' \in U^q$ 
plays the role of $\lova{vu}'$ or $\ova{xu}$
in $<n d_{\cal S}(q)$ moves,
so with probability $< 10^4 p^{-3} n^{-1}$,
and the role of $\ova{vx}$ in $<d_{\cal S}(q)$ moves,
so with probability $< 10^4 p^{-3} n^{-2}$,
and the role of $\lova{zu}'$ in $< n^2$ moves,
so with probability $10^4 p^{-3} d_{\cal S}(q)^{-1}$.
Thus for any $q \in S$ and $y = \phi_w(q)$,
the number of moved $\ova{yy}'$ with $y' \in U^q$
is $(\mu,1)$-dominated, where
$\mu <
  10^4 p^{-3} n^{-1} \cdot p_+^{.8} n |U^q|
  + 10^4 p^{-3} n^{-2} \cdot p_+^{.8} n^2 |U^q|
+ 10^4 p^{-3} d_{\cal S}(q)^{-1} \cdot 
|d^+_{D_w}(y) - d_{\cal S}(q)| |U^q|
 < .1p_+^{.7} |U^q|$,
so by Lemma \ref{freedman} is whp $< p_+^{.7} |U^q|$.
This completes the proof.

\section{Concluding remarks}

In this paper we have developed a variety of embedding
techniques that are sufficiently flexible to resolve
a generalised form of Ringel's Conjecture 
that applies to quasirandom graphs,
and which promise to have more general applications
to packings of a family of trees, 
as would be required for a solution
of Gy\'arf\'as' Conjecture.

\end{document}